\documentclass[11pt]{amsart}

\usepackage{amssymb, amstext, amscd, amsmath, color}

\usepackage{url}
\usepackage{tikz}
\usepackage{epstopdf}
\usepackage{amsfonts}
\usepackage{amsthm}
\usepackage{amsfonts}
\usepackage{array}
\usepackage{epsfig}
\usepackage{eucal}
\usepackage{latexsym}
\usepackage{mathrsfs}
\usepackage{textcomp}
\usepackage{verbatim}
\usepackage{setspace}
\usepackage{enumerate}
\usepackage[colorlinks=true,linkcolor=blue,urlcolor=blue,citecolor=red]{hyperref}

\newtheorem{thm}{Theorem}[section]

\newtheorem{cor}[thm]{Corollary}

\newtheorem{lem}[thm]{Lemma}

\newtheorem{prop}[thm]{Proposition}
\newtheorem{rem}[thm]{Remark}

\numberwithin{equation}{section}

\newcommand{\tr}{\mathrm{tr}}

  \newcommand{\R}{{\mathbb{R}}}

  \newcommand{\T}{{\mathcal{T}}}
  \newcommand{\U}{{\mathcal{U}}}

  \newcommand{\Y}{{\mathcal{Y}}}
  
\newcommand{\Aut}{\operatorname{Aut}}

\usepackage{mathtools}

\def\quotient#1#2{
    \raise1ex\hbox{$#1$}\big/\lower1ex\hbox{$#2$}}
\def\Bigquotient#1#2{
    \raise1ex\hbox{$#1$}\Big/\lower1ex\hbox{$#2$}}
\usepackage{upgreek}
\tikzset{math3d/.style= {x={(1cm,0cm)}, y={(0.353cm,0.353cm)}, z={(0cm,1cm)}}}
\usepackage{tikz-3dplot}

\usepackage{blkarray}

\tikzstyle{vertex}=[circle, draw, fill=black, inner sep=0pt, minimum size=3pt]
\tikzstyle{fadedvertex}=[vertex, fill=black!30]
\tikzstyle{edge}=[line width=1pt]
\tikzstyle{fadededge}=[edge,color=black!30]

\begin{document}

\title[Rigidity of symmetric frameworks on the cylinder]{Rigidity of symmetric frameworks on the cylinder}
\author[Anthony Nixon]{Anthony Nixon}
\address{Dept.\ Math.\ Stats.\\ Lancaster University\\
Lancaster LA1 4YF \\U.K.}
\email{a.nixon@lancaster.ac.uk}
\author[Bernd Schulze]{Bernd Schulze}
\address{Dept.\ Math.\ Stats.\\ Lancaster University\\
Lancaster LA1 4YF \\U.K.}
\email{b.schulze@lancaster.ac.uk}
\author[Joseph Wall]{Joseph Wall}
\address{Dept.\ Math.\ Stats.\\ Lancaster University\\
Lancaster LA1 4YF \\U.K.}
\email{j.wall@lancaster.ac.uk}

\begin{abstract}
A bar-joint framework $(G,p)$ is the combination of a finite simple graph $G=(V,E)$ and a placement $p:V\rightarrow \mathbb{R}^d$. The framework is rigid if the only edge-length preserving continuous motions of the vertices arise from isometries of the space. This article combines two recent extensions of the generic theory of rigid and flexible graphs by considering symmetric frameworks in $\mathbb{R}^3$ restricted to move on a surface. In particular necessary combinatorial conditions are given for a symmetric framework on the cylinder to be isostatic (i.e. minimally infinitesimally rigid) under any finite point group symmetry. 
In every case when the symmetry group is cyclic, which we prove restricts the group to being inversion, half-turn or reflection symmetry, these conditions are then shown to be sufficient under suitable genericity assumptions, giving precise combinatorial descriptions of symmetric isostatic graphs in these contexts.
\end{abstract}

\date{}
\maketitle

\section{Introduction}

A (bar-joint) \emph{framework} $(G,p)$ is the combination of a finite simple graph 
$G=(V,E)$ and a map $p:V\rightarrow \mathbb{R}^d$ which assigns positions to the vertices, and hence lengths to the edges. With stiff bars for the edges and full rotational freedom for the joints representing the vertices, the topic of rigidity theory concerns whether the framework may be deformed without changing the graph structure or the bar lengths. While `trivial' motions are always possible due to actions of the Euclidean isometry group, the framework is \emph{flexible} if  a non-trivial motion is possible and \emph{rigid} if no non-trivial motion exists.

The problem of determining whether a given framework is rigid is computationally difficult for all $d\geq 2$ \cite{Abb}. However, every graph has a typical behaviour in the sense that either all `generic' (i.e. almost all) frameworks with the same underlying graph are rigid or all are flexible. So, generic rigidity depends only on the graph and is often studied using a linearisation known as infinitesimal rigidity, which is equivalent to rigidity for generic frameworks  
\cite{AR}.
On the real line it is a simple folklore result that rigidity coincides with graph connectivity. In the plane a celebrated theorem due to Polaczek-Geiringer \cite{PG}, often referred to as Laman's theorem due to a rediscovery in the 1970s \cite{laman}, characterises the generically rigid graphs precisely in terms of graph sparsity counts, and these combinatorial conditions can be checked in polynomial time. However when $d\geq 3$ little is known. This motivated extensions and generalisations of the types of framework and ambient spaces under consideration. One such case is to replace $\mathbb{R}^d$ with a $d$-dimensional manifold (or $d$-fold for short). It seems unlikely that rigidity becomes easier on a $d$-fold when $d\geq 3$ and hence it is natural to consider rigidity for frameworks realised on $2$-folds.

Specifically, let $S$ be a $2$-fold embedded in $\mathbb{R}^3$ and let the framework $(G,p)$ be such that $p:V\rightarrow S$, but the `bars' are straight Euclidean bars (and not surface geodesics). Supposing $S$ is smooth, an irreducible real algebraic set and the subgroup of Euclidean isometries that preserve $S$ has dimension at least 1,  characterisations of generic rigidity were proved in \cite{NOP,NOP14}. In particular the first important case with distinct combinatorics to the Euclidean plane is the case of an infinite circular cylinder.

Separately, the genericity hypothesis, while natural from an algebraic geometry viewpoint, does not apply in many practical applications of rigidity theory. In particular, a number of applications require frameworks to admit non-trivial symmetry. This has motivated multiple groups of researchers to study symmetric rigidity theory over the last two decades. We direct the reader to \cite{cg,SW2} for details. Importantly, there are two quite different notions of symmetric rigidity that one may consider. Firstly, \emph{forced symmetric rigidity} concerns frameworks that are symmetric and only motions that preserve the symmetry are allowed (that is, a framework may be flexible but since the motions destroy the symmetry it can still be `forced symmetrically rigid'). Secondly, \emph{incidental symmetric rigidity} where the framework is symmetric, but the question of whether it is rigid is the same as in the non-symmetric case. 

It is incidental symmetry that we focus on in this article. More specifically, we are interested in describing, combinatorially, when a generic symmetric framework on a surface such as the infinite cylinder is \emph{isostatic}, i.e. minimally infinitesimally rigid in the sense that it is infinitesimally rigid but ceases to be so after deleting any edge. The corresponding question in the Euclidean plane has been studied in  \cite{schulze,BS4}. In these papers, 
Laman-type results in the plane have been established for the groups generated by a reflection, the half-turn and a three-fold rotation, but these problems remain open for the other groups that allow isostatic frameworks.

In \cite{NS}, the first two authors   studied the forced symmetric rigidity of frameworks on $2$-folds. The present article gives the first analysis of incidental symmetric rigidity on $2$-folds. We focus our attention mostly on the important special case of the cylinder. To see why, first consider the `simplest' $2$-fold: the sphere. In this case, Laman-type theorems either follow from a projective transfer between infinitesimal rigidity in the plane and on the sphere \cite{CNSW,EJNSTW} or seem to be equally as challenging as the open problems in the plane. In the final section (Section~\ref{sec:final}) we point out the precise possibilities  for isostatic frameworks on the sphere which can be established using similar techniques to those we employ in Section \ref{sec:necrep} below. The cylinder provides the first case when the combinatorial sparsity counts change and hence lead to new classes of graphs and rigidity matroids to investigate. 

Our main results are  representation-theoretic necessary conditions for isostaticity on the cylinder for all relevant symmetry groups (Section~\ref{sec:necrep}), as well as complete combinatorial characterisations of symmetry-generic isostatic frameworks on the cylinder for the groups generated by an inversion (Section~\ref{sec:inversion}), a half-turn (Section~\ref{sec:half}) and a reflection (Section \ref{sec:mirror}).
The proofs rely on symmetry-adapted Henneberg-type recursive construction moves described in Section~\ref{sec:ops}.

In the case of isostatic frameworks in $\mathbb{R}^2$  only the well known 0- and 1-extension operations are needed to prove Laman's theorem \cite{PG,laman}.
For the cylinder several additional operations were needed with associated combinatorial and geometric difficulties \cite{NOP}. The additional conditions isostatic frameworks under symmetry must satisfy differ for each group, necessitating group-by-group combinatorial (and hence geometric) analyses. Fortunately, in each of the cases we study in detail only moderate extensions of existing geometric arguments are needed and hence we present a number of those for an arbitrary symmetry group (Section \ref{sec:ops}). On the other hand there are significant additional combinatorial difficulties in the recursive construction proof technique which takes up the main technical parts of this article (Sections~\ref{sec:inversion}, \ref{sec:half} and \ref{sec:mirror}).

\section{Rigidity theory}\label{sec:basics}

\subsection{Frameworks on surfaces}

Let $S$ denote a surface in $\mathbb{R}^3$.
A framework $(G,p)$ on $S$ is the combination of $G=(V,E)$ and a map $p:V\rightarrow \mathbb{R}^3$ such that $p(v)\in S$ for all $v\in V$ and $p(u)\neq p(v)$ for all $uv\in E$. We also say that $(G,p)$ is a \emph{realisation} of the graph $G$ on $S$. $(G,p)$ is \emph{rigid on $S$} if every framework $(G,q)$ on $S$ that is sufficiently close to $(G,p)$ arises from an isometry of $S$. 

While much of this section remains true for arbitrary choices of $S$ in all the sections that follow we will focus on the  important case when $S$ is a cylinder.
Throughout this paper, $\Y$ denotes the infinite circular cylinder; that is the real algebraic subvariety of $\mathbb{R}^3$ defined by the irreducible polynomial $x^2+y^2=1$. 

As in the Euclidean case, it is a computationally challenging problem to determine if a given framework $(G,p)$ is rigid on $\Y$. Hence we follow the standard path of linearising and considering infinitesimal motions as follows. 

Given a framework $(G,\hat p)$ on $\Y$, we are interested in the set of frameworks $(G,p)$  on $\Y$ which are equivalent to $(G,\hat p)$ where $\hat p(v_i)=(\hat x_i,\hat y_i,\hat z_i)$ and $p(v_i)=(x_i,y_i,z_i)$. 
The set of all frameworks on $\Y$ that are equivalent to $(G,\hat p)$ is given by the set of solutions to the following system of equations:
\begin{eqnarray}
\label{eqn:a} \|p(v_i)-p(v_j)\|^2&=& c_{ij} \hspace{3cm} (v_iv_j\in E) \\
\label{eqn:c} x_i^2+y_i^2&=& 1 \hspace{3.3cm} (v_i\in V)
\end{eqnarray}
where $c_{ij}= \|\hat p(v_i)-\hat p(v_j)\|^2$.
We can differentiate these equations to obtain the following linear system for the unknowns $\dot p(v_i)$, $v_i\in V$:
\begin{eqnarray}
\label{eqn:1} (p(v_i)-p(v_j))\cdot (\dot p(v_i)-\dot p(v_j))&=& 0 \hspace{3cm} (v_iv_j\in E) \\
\label{eqn:3} x_i\dot x_i+y_i\dot y_i&=& 0 \hspace{3cm} (v_i\in V).
\end{eqnarray}
Solutions to this linear system are \emph{infinitesimal motions}.
We say that $(G,\hat p)$ is \emph{infinitesimally rigid} if the only infinitesimal motions are the trivial solutions that arise from Euclidean congruences of $\mathbb{R}^3$ that preserve $\Y$ (that is, translations in the $z$-direction and rotations about the $z$-axis, or combinations thereof). If $(G,p)$ is not infinitesimally rigid it is called \emph{infinitesimally flexible}. The trivial solutions may be referred to as the \emph{trivial} infinitesimal motions, or simply \emph{trivial motions}. 
Equivalently, $(G,\hat p)$ is \emph{infinitesimally rigid} if the rank of the matrix of coefficients of the system is $3|V|-2$. 
This matrix,
the {\em rigidity matrix of $(G,p)$ on $\Y$}, denoted $R_{\Y}(G,p)$ has $3|V|$ columns and $|E|+|V|$ rows. The rows corresponding to (\ref{eqn:1}) have the form
$$\begin{pmatrix}
\dots & 0 & p(v_i)-p(v_j) & 0 & \dots & 0 & p(v_j)-p(v_i) & 0 & \dots
\end{pmatrix}
$$
and the rows corresponding to (\ref{eqn:3}) have the form
$$\begin{pmatrix}
\dots & 0 & (x_i,y_i,0) & 0 & \dots
\end{pmatrix}. $$

A framework $(G,p)$  is called \emph{isostatic} if it is infinitesimally rigid and \emph{independent} in the sense that the rigidity matrix of $(G,p)$ on $\Y$ has no non-trivial row dependence. Equivalently, $(G,p)$ is isostatic if it is infinitesimally rigid and deleting any single edge results in a framework that is not infinitesimally rigid.
A framework $(G,p)$ on $\Y$ is \emph{completely regular} if 
the rigidity matrix $R_{\Y}(K_{|V|},p)$ of the complete graph on $V$ (and every square submatrix) has maximum rank among all realisations of $K_{|V|}$ on $\Y$.
In the completely regular case, rigidity and infinitesimal rigidity on $\Y$ coincide \cite[Theorem 3.8]{NOP}. Note that the set of all completely regular realisations of $G$ on $\Y$ is an open dense subset of the set of all realisations of $G$ on $\Y$.
Thus, we may define a graph $G$ to be \emph{isostatic (independent, rigid) on $\Y$} if there exists 
a framework $(G,p)$ on $\Y$ that is  isostatic (independent, infinitesimally rigid) on $\Y$. 

It follows from the definitions that the smallest (non-trivial) rigid (or isostatic) graph on $\Y$ is the complete graph $K_4$. In \cite{NOP} exactly which graphs are rigid on $\Y$ was characterised. The characterisation uses the following definition which will be one of the fundamental objects of study in this paper.
A graph $G=(V,E)$ is \emph{$(2,2)$-sparse} if  $|E'|\leq 2|V'|-2$ for all subgraphs $(V',E')$ of $G$. $G$ is \emph{$(2,2)$-tight} if it is $(2,2)$-sparse and $|E|=2|V|-2$.

\begin{thm} \label{thm:lm}
A graph $G$ is isostatic on $\Y$ if and only if $G$ is $(2,2)$-tight.
\end{thm}

While the theorem gives a complete answer in the generic case, the present article will improve this answer to apply under the presence of non-trivial symmetry.
To see the potential complications that can arise when the genericity hypothesis is weakened one might consider the results of \cite{JKN}  which apply to frameworks on $\Y$ that are generic except for one simple failure: two vertices are located in the same place. 

\subsection{Symmetric frameworks on the cylinder}
Let $G= (V,E)$ be a graph and $\Gamma$ be a finite group.
Then the pair $(G,\phi)$ is called $\Gamma$-symmetric if $\phi: \Gamma \to \Aut (G)$ is a homomorphism, where $\Aut (G)$ denotes the automorphism group of $G$. If $\phi$ is clear from the context we often also simply write $G$ instead of $(G,\phi)$.

Let $(G,\phi)$ be a $\Gamma$-symmetric graph.
Then, for a homomorphism $\tau: \Gamma \to O(\R^{3})$ and the cylinder $\Y$, we say that a framework $(G,p)$ is \emph{$\Gamma$-symmetric on  $\Y$ (with respect to $\phi$ and $\tau$)}, or simply \emph{$\tau(\Gamma)$-symmetric}, if $\tau(\gamma)p_{i} = p_{\phi(\gamma)i}$ for all $i\in V$ and all $\gamma \in \Gamma$ and $p:V\rightarrow \mathbb{R}^3$ is such that $p(v) \in \Y$ for all $v\in V$.
We will refer to $\tau(\Gamma)$ as a \emph{symmetry group} and to  elements of $\tau(\Gamma)$ as \emph{symmetry operations} or simply \emph{symmetries} of $(G,p)$. 
We will often need to work with symmetric subgraphs and their frameworks. 
So for a $\Gamma$-symmetric graph $(G,\phi)$ we often consider a $\Gamma$-symmetric subgraph $(H,\phi')$, where $\phi'(\gamma)=\phi(\gamma)|_{V(H)}$. In that case we often slightly abuse notation and write $(H,\phi)$ (or even just $H$) instead of $(H,\phi')$.

A $\Gamma$-symmetric framework $(G,p)$ on $\Y$ (with respect to $\tau$ and $\phi$) is \emph{completely $\Gamma$-regular (with respect to $\tau$ and $\phi$)} if the rigidity matrix $R_{\Y}(K_{|V|},p)$ of the complete graph on $V$ and every square submatrix has maximum rank among all $\Gamma$-symmetric realisations of $K_{|V|}$  on $\Y$ (with respect to $\tau$ and $\phi$).
The set of all completely $\Gamma$-regular realisations of $G$ on $\Y$ (with respect to $\tau$ and $\phi$) is an open dense subset of the set of all $\Gamma$-symmetric realisations of $G$ on $\Y$ (with respect to $\tau$ and $\phi$).
Thus, we may say that a graph $G$ is \emph{$\tau(\Gamma)$-isostatic (independent, infinitesimally rigid, rigid)} on $\Y$ if there exists a $\Gamma$-symmetric framework $(G,p)$ on $\Y$ (with respect to $\tau$ and $\phi$) which is isostatic (independent, infinitesimally rigid, rigid).
Later we will often remove $\phi$ from this notation and simply refer to a  $\tau(\Gamma)$-isostatic (independent, infinitesimally rigid, rigid) graph on $\Y$ (where $\phi$ is clear from the context). 

 An isometry of $\mathbb{R}^3$ that maps $\Y$ onto itself is called a  \emph{surface-preserving isometry}. A symmetry group of a framework on $\Y$   consisting of surface-preserving  isometries is called a \emph{surface-preserving symmetry group}. 

Throughout this paper, we will use a version of the Schoenflies notation for symmetry operations and groups of frameworks on $\Y$. The relevant symmetry operations are the identity, denoted by  $\textrm{id}$; rotations by $\frac{2\pi}{n}$, $n\in \mathbb{N}$, denoted by  $c_n$, where the rotational axis is the $z$-axis for $n\geq 3$, and either the $z$-axis or any line in the $xy$ plane going through the origin for $n=2$;  reflections in the $xy$ plane or any plane containing the $z$-axis, denoted by $\sigma$; and improper rotations (i.e. rotations $c_n$ followed by a reflection in the plane through the origin that is perpendicular to the $c_n$ axis), denoted by $s_n$. Note that $s_2$ is the inversion in the origin; this operation is  also denoted by $\varphi$. The relevant  symmetry
groups for this paper are the  group  $C_i$ generated by the inversion $\varphi$, the group $C_s$ generated by a reflection $\sigma$, the cyclic groups $C_n$ generated by a rotation $c_n$, and the  dihedral groups $C_{nv}$, $C_{nh}$ and $D_n$, where $C_{nv}$ ($C_{nh}$) is generated by a rotation $c_n$ and a reflection $\sigma$ whose mirror plane contains (is perpendicular to) the axis of $c_n$, and $D_n$ is generated by a rotation $c_n$ and a half-turn $c_2$ whose axis is perpendicular to the $c_n$ axis. 

\section{Necessary Conditions for Isostatic Frameworks}
\label{sec:necrep}

\subsection{Block-diagonalization of the rigidity matrix}

In this section we show that the rigidity matrix of a symmetric framework on $\Y$ can be transformed into a block-decomposed
form using techniques from group representation theory. It follows immediately from Theorem~\ref{thm:lm} and the discussion in Section~\ref{sec:basics} that for an isostatic framework $(G,p)$ on $\Y$, the graph $G$ must be $(2,2)$-tight.
If $(G,p)$ is symmetric then the block-decomposition of the rigidity matrix can be used to obtain additional  necessary conditions for the framework to be isostatic. To obtain this block-decomposition of the rigidity matrix, we need to define analogues of the internal and external representation defined in \cite{KG2,owen,BS2}. 

Let $A$ be a $m \times n$ matrix and $B$ be a $p \times q$ matrix. The Kronecker product $A\otimes B$ is the $pm \times qn$ block matrix:
\[A\otimes B = \left[
\begin{tabular}{ c c c }
        $b_{11}A$ & $\dots$ & $b_{1q}A$\\
        $\vdots$ & $\ddots$ & $\vdots$ \\
        $b_{p1}A$ & $\dots$ & $b_{pq}A$ \\
    \end{tabular}
\right].
\]
Let $\Gamma$ be a group and let $\tau:\Gamma \to O(\mathbb{R}^3)$ be a group representation.
Further, for each $\gamma\in\Gamma$, let $P_V(\gamma)$  be the permutation matrix of $V$ induced by $\gamma$. That is, $P_V(\gamma)= (\delta_{i,\gamma(i')})_{i,i'}$ for each $\gamma\in \Gamma$, where $\delta$ denotes the Kronecker delta. Similarly, let $P_E(\gamma)$  be the permutation matrix of $E$ induced by $\gamma$. 

The \emph{external representation} is defined as $\tau \otimes P_V : \Gamma \to GL(\mathbb{R}^{3|V|})$  and the \emph{internal representation} is defined as $\Tilde{P}_E := P_E \oplus P_V : \Gamma \to GL(\mathbb{R}^{(|E|+|V|)})$. 

For a point $p_k=(x,y,z)^T$ on the cylinder $\Y$, we define the normal $n(p_k)$ to $\Y$ at $p_k$ as $n(p_k)=(x,y,0)^T$. It is a routine calculation to show that $n(\tau(\gamma)p_{k}) = \tau(\gamma)n(p_{k})$ for all surface-preserving isometries of $\R^{3}$.
We refer the reader to \cite{Wall} for the details. Thus, we have the following fact.

\begin{lem}\label{image of normal}
For any surface-preserving symmetry group $\tau(\Gamma)$ of $\Y$,
we have $n(\tau(\gamma)p_{k}) = \tau(\gamma)n(p_{k})$ for all $\gamma\in \Gamma$. 
\end{lem}

Using Lemma~\ref{image of normal}, we may establish the following fundamental result.

\begin{prop}\label{Rextu=intz}
Let $(G,p)$ be a $\tau(\Gamma)$-symmetric framework on $\Y$.
If $R_{\Y}(G,p)u = z$, then for all $\gamma \in \Gamma$, we have $$R_{\Y}(G,p)(\tau\otimes P_V)(\gamma)u = \Tilde{P}_E(\gamma)z.$$
\end{prop}

\begin{proof}
Suppose $R_{\Y}(G,p)u = z$.
Fix $\gamma \in \Gamma$ and let $\tau(\gamma)$ be the orthogonal matrix representing $\gamma$ with respect to the canonical basis of $\R^{3}$.
We enumerate the rows of $R_{\Y}(G,p)$ by the set $\{a_{1}, \dots, a_{|E|}, b_{1}, \dots, b_{|V|}\}$.
By \cite{BS2}, we know that $(R_{\Y}(G,p)(\tau\otimes P_V)(\gamma)u)_{a_{i}} = (\Tilde{P}_E(\gamma)z)_{a_{i}}$, for all $i \in [|E|]$. We are left to show the result holds for the rows of $R_{\Y}(G,p)$ which represent the normal vectors of the vertices on the surface.

Write $u \in \R^{3|V|}$ as $u = (u_{1}, \dots, u_{|V|})$, where $u_{i} \in \mathbb{R}^3$ for all $i$, and let $\Phi(\gamma)(v_{i}) = v_{k}$.
We first see that $(\Tilde{P}_E(\gamma)z)_{b_{k}} = z_{b_{i}}$
by the definition of $P_{V}(\gamma)$.
From $R_{\Y}(G,p)u = z$, we also get that $z_{b_{i}} = n(p_{i}) \cdot u_{i}$.
Then $(\tau\otimes P_V)(\gamma)u = (\Bar{u}_{1}, \dots, \Bar{u}_{|V|})$, with $\Bar{u_{l}} = \tau(\gamma)u_{j}$ when $\Phi(\gamma)(v_{j}) = v_{l}$.
Therefore,
\begin{align*}
    (R_{\Y}(G,p)(\tau\otimes P_V)(\gamma)u)_{b_{k}}
    &= n_{1}(p_{k}) \cdot (\tau(\gamma)u_{i})_{1} + n_{2}(p_{k}) \cdot (\tau(\gamma)u_{i})_{2} + n_{3}(v_{k}) \cdot (\tau(\gamma)u_{i})_{3}\\
    &= n(p_{k}) \cdot (\tau(\gamma)u_{i})\\
    &= n(\tau(\gamma)p_{i}) \cdot (\tau(\gamma)u_{i}).
\end{align*}

Finally, using Lemma \ref{image of normal} plus the fact that the canonical inner product on $\R^{3}$ is invariant under the orthogonal transformation $\tau(\gamma) \in O(\mathbb{R}^3)$ gives that $n(\tau(\gamma)p_{i}) \cdot (\tau(\gamma)u_{i}) = \tau(\gamma)n(p_{i}) \cdot (\tau(\gamma)u_{i}) = n(p_{i}) \cdot u_{i} = z_{b_{i}},$
finishing the proof.
\end{proof}

The following is an immediate corollary of Schur's lemma (see e.g. \cite{Serre}) and the proposition above.

\begin{cor}
Let $(G,p)$ be a $\tau(\Gamma)$-symmetric framework on $\Y$ and let $I_{1}, \dots, I_{r}$  be the  pairwise non-equivalent irreducible linear representations of $\tau(\Gamma)$.
Then there exist matrices $A,B$ such that the matrices $B^{-1}R_{\Y}(G,p)A$ and $A^{-1}R_{\Y}(G,p)^{T}B$ are block-diagonalised and of the form
\[
\left(\begin{tabular}{ c c c c c }
        $R_{1}$ & & & & $\mathbf{0}$ \\
         & $R_{2}$ & & & \\
         & & $\ddots$ & & \\
         & & & & \\
        $\mathbf{0}$ & & & & $R_{r}$ \\
    \end{tabular} \right)
\]
where the submatrix $R_{i}$ corresponds to the irreducible representation $I_{i}$.
\end{cor}

This block decomposition corresponds to $\R^{3|V|} = X_{1} \oplus \dots \oplus X_{r}$ and $\R^{|E|+|V|} = Y_{1} \oplus \dots \oplus Y_{r}$.
The space $X_{i}$ is the $(\tau \otimes P_{V})$-invariant subspace of $\mathbb{R}^{3|V|}$ corresponding to $I_i$, and the space $Y_{i}$ is the $\Tilde{P}_{E}$-invariant subspace of $\mathbb{R}^{|E|}$ corresponding to $I_i$.
Then, the submatrix $R_{i}$ has size $(\dim(Y_{i})) \times (\dim(X_{i}))$.

\subsection{Additional necessary conditions}

Using the block-decomposition of the rigidity matrix, we may follow the basic approach described in \cite{FGsymmax,BS2} to derive added necessary conditions for a symmetric framework on $\Y$ to be isostatic. We first need the following result.

\begin{thm} \label{thm:trivialinv}
The space of trivial motions of an affinely spanning $\tau(\Gamma)$-symmetric framework $(G,p)$ on $\Y$, written $\T(G,p)$, is a $(\tau \otimes P_{V})$-invariant subspace of $\R^{3|V|}$.
Furthermore, the space of translational motions and the space of rotational motions of $(G,p)$ are also $(\tau \otimes P_{V})$-invariant subspaces of $\R^{3|V|}$.
\end{thm}

\begin{proof}
Suppose that $(G,p)$ affinely spans $\mathbb{R}^3$, so that the trivial motion space of $(G,p)$ on $\Y$ is $2$-dimensional. We first show that $N = \ker (R_{\Y}(K_{|V|},p))$ is $(\tau\otimes P_V)$-invariant.
By Lemma \ref{Rextu=intz}, if $R_{\Y}(K_{|V|},p)u = z$ then $R_{\Y}(K_{|V|},p)(\tau\otimes P_V)(\gamma)u = \Tilde{P}_E(\gamma)z$. Let $u\in N$, then $R_{\Y}(K_{|V|},p)u = 0$, so
\begin{align*}
    \Tilde{P}_E(\gamma)R_{\Y}(K_{|V|},p)u & = \Tilde{P}_E(\gamma)z\\
    & = \Tilde{P}_E(\gamma)0 = 0.
\end{align*}
Thus $R_{\Y}(K_{|V|},p)(\tau\otimes P_V)(\gamma)u = \Tilde{P}_E(\gamma)R_{\Y}(K_{|V|},p)u = 0$, giving $(\tau\otimes P_V)(\gamma)u \in \ker (R(K_{|V|},p))$. Hence $N$ is $(\tau\otimes P_V)$-invariant, as required for the first part of the theorem.

To show that the space of translational  motions is $(\tau\otimes P_V)$-invariant, first note that for $\Y$,
this space is generated by the vector $t=(0,0,1,0,0,1,\dots,0,0,1)^T$.
We need to show that for each $\gamma\in\Gamma$, we have $(\tau\otimes P_V)(\gamma)t = \alpha t$ for some $\alpha \in \R$.
By the definition of $\tau\otimes P_V$ this holds if $\tau(\gamma)(0,0,1)^T = \alpha (0,0,1)^T$ for all $\gamma\in\Gamma$.
Since $\tau(\Gamma)$ preserves $\Y$, such an $\alpha$ does exist for each $\gamma$ (specifically $\alpha = \pm{1}$).

Finally we look at the space of rotational  motions.
For $\Y$, this space is generated by the vector $r=(r_1,\ldots , r_{|V|})\in \mathbb{R}^{3|V|}$ defined as $r_k = (p_{k})_{1}e_{2} - (p_{k})_{2}e_{1} \in\mathbb{R}^3$, for all $k \in V$, where $e_1$ and $e_2$ are the standard basis vectors of $\mathbb{R}^3$ with  $1$ as the first and second coordinate, respectively. Note that $r$ is perpendicular to $t$.
Since for all $\gamma \in \Gamma$, $(\tau\otimes P_V)(\gamma)$ is an orthogonal matrix, $(\tau\otimes P_V)$ is a unitary representation (with respect to the canonical inner product on $R^{3|V|}$).
Therefore the subrepresentation ${H'_{e}}^{(N)}$ of $H'_{e}$ with representation space $N$ is also unitary  (with respect to the inner product obtained by restricting the canonical inner product on $R^{3|V|}$ to N).
It follows that the space $\langle r \rangle$ is $(\tau\otimes P_V)$-invariant since it is the orthogonal complement to $\langle t \rangle$ in $N$. 
\end{proof}

Let $(\tau \otimes P_{V})^{(\T)}$ be the subrepresentation of $(\tau \otimes P_{V})$ with representation space $\T(G,p)$.
Then $\T = T_{1} \oplus \dots \oplus T_{r}$ where $T_{i}$ is the $(\tau \otimes P_{V})$-invariant subspace corresponding to the irreducible representation $I_i$.

If $A = (a_{ij})$ is a square matrix then the trace of $A$ is given by $\tr(A) = \sum_{i}a_{ii}$.
For a linear representation $\rho$ of a group $\Gamma$ and a fixed ordering $\gamma_1,\ldots, \gamma_{|\Gamma|}$ of the elements of $\Gamma$, the character of $\rho$ is the $|\Gamma|$-dimensional vector $\chi(\rho)$ whose $i$th entry is $\tr(\rho(\gamma_i))$. 

\begin{thm}\label{thm: character count}
Let $(G,p)$ be a $\tau(\Gamma)$-symmetric framework on $\Y$.
If $(G,p)$ is isostatic, then
$$\chi(\Tilde{P}_{E}) = \chi( \tau \otimes P_{V}) - \chi((\tau \otimes P_{V})^{(\T)}).$$
\end{thm}

\begin{proof}
By Maschke's Theorem, for the subrepresentation $(\tau \otimes P_{V})^{(\T)} \subseteq (\tau \otimes P_{V})$, there exists a subrepresentation $(\tau \otimes P_{V})^{(Q)} \subseteq (\tau \otimes P_{V})$ with $(\tau \otimes P_{V})^{(\T)} \oplus (\tau \otimes P_{V})^{(Q)} = \tau \otimes P_{V}$.
Further,
since $\tau\otimes P_{V}$ is unitary,
we know that $Q(G,p)$ is the $(\tau \otimes P_{V})$-invariant subspace of $\R^{3|V|}$ which is orthogonal to $\T(G,p)$.

Since $(G,p)$ is isostatic, the restriction of the 
linear map given by the rigidity matrix to $Q(G,p)$ is an isomorphism onto $R^{|E|+|V|}$.
Moreover if $R'_{\Y}(G,p)$ is the matrix corresponding to this linear map restricted to $Q(G,p)$, then, the statement for $R_{\Y}(G,p)$ in Proposition~\ref{Rextu=intz} also holds for $R'_{\Y}(G,p)$ and hence we have $$R'_{\Y}(G,p)(\tau\otimes P_V)(\gamma)(R'_{\Y}(G,p))^{-1}= \Tilde{P}_E(\gamma) \quad \textrm{ for all } \gamma\in \Gamma.$$ Thus, $(\tau \otimes P_{v})^{(Q)}$ and $\Tilde{P}_{E}$ are isomorphic representations of $\Gamma$.
Therefore, we have 
$$\chi(\Tilde{P}_{E}) = \chi((\tau \otimes P_{v})^{(Q)}) = \chi( \tau \otimes P_{v}) - \chi((\tau \otimes P_{v})^{(\T)}).$$
\end{proof}

\subsection{Character table}

We now calculate the characters of the representations appearing in the statement of Theorem~\ref{thm: character count}.
The symmetry-preserving symmetry operations for $\Y$ are
rotations $c_n$, $n\in \mathbb{N}$, around the $z$-axis, reflections in a plane containing the $z$-axis, denoted by $\sigma$,  reflection in the $xy$-plane, denoted by $\sigma'$, half-turn  in an axis that is perpendicular to the $z$-axis (and goes through the origin), denoted by ${c_2}'$  and improper rotations around the $z$-axis, denoted by $s_n$, $n\geq 2$.
Recall that for $n=2$, $s_n$ is the inversion $\varphi$ in the origin.
The values of the traces of the matrices for $\Tilde{P}_{E}$ and $\tau \otimes P_{V}$ for each group element follow immediately from the definition. The following lemma provides the traces of the matrices for $(\tau \otimes P_{V})^{(\T)}$.

\begin{lem}\label{cyl-triv-mot}
Let $(G,p)$ be a $\tau(\Gamma)$-symmetric framework on $\Y$. Then  the character $\chi((\tau \otimes P_{V})^{(\T)})$ can be computed using Table~\ref{table:3.6}.
\begin{center}
\begin{table}[ht]
    \begin{tabular}{ |c|c|c|c|c|c|c|c| } 
\hline
 & $\textrm{id}$ & $c_{n}$ & $c_{2}'$ & $\sigma$ & $\sigma'$ & $s_{n}$ & $\varphi$ \\
\hline
$\chi((\tau \otimes P_{v})^{(\T)})$ & 2 & 2 & -2 & 0 & 0 & 0 & 0\\
\hline
\end{tabular}
\caption{Character of the external representation restricted to trivial infintesimal motions for symmetry operations of the cylinder.}
\label{table:3.6}
\end{table}
\end{center}
\end{lem}

\begin{proof}

Recall from the proof of Theorem~\ref{thm:trivialinv} that a basis for the space of trivial translational and rotational motions of a non-trivial framework on $\Y$ is given by the vectors $t$ and $r$, respectively.

It is easy to see that each of the matrices $(\tau \otimes P_{V})(\textrm{id})$ and $(\tau \otimes P_{V})(c_n)$, $n\in \mathbb{N}$, map both $t$ and $r$ to themselves. Thus, $$\tr((\tau \otimes P_{V})^{(\T)})(\textrm{id}))=\tr((\tau \otimes P_{V})^{(\T)})(c_n)) = 1 +1 = 2.$$ Further straightforward calculations using the definition of $(\tau \otimes P_{V})$ show that:
\begin{itemize}
\item $(\tau \otimes P_{v})(c_2')t=-t$ and $(\tau \otimes P_{v})(c_2')r=-r$. So $\tr((\tau \otimes P_{V})^{(\T)})(c_2')) = -1 -1 = -2$.
\item $(\tau \otimes P_{v})(\sigma)t=t$ and $(\tau \otimes P_{v})(\sigma)r=-r$. So $\tr((\tau \otimes P_{V})^{(\T)})(\sigma)) = 1 -1 = 0$.
\item $(\tau \otimes P_{v})(\sigma')t=-t$ and $(\tau \otimes P_{v})(\sigma')r=r$. So $\tr((\tau \otimes P_{V})^{(\T)})(\sigma')) = -1 +1 = 0$.
\item $(\tau \otimes P_{v})(s_n)t=-t$ and $(\tau \otimes P_{v})(s_n)r=r$. So $\tr((\tau \otimes P_{V})^{(\T)})(s_n)) = -1 +1 = 0$.
\end{itemize}
We refer the reader to \cite{Wall} for the detailed calculations.
\end{proof}

We are now able to  give the full character table for $\tau(\Gamma)$-symmetric isostatic frameworks on $\Y$ (see Table \ref{tab:CTCyl}).
We give these without calculation as they can be seen directly from the matrix representations of $\tau \otimes P_{V}$ and $\Tilde{P}_{E}$.

For a  $\Gamma$-symmetric graph $G=(V,E)$ with respect to $\phi:\Gamma \to \textrm{Aut}(G)$, we say that a vertex $v\in V$ is \emph{fixed} by $\gamma\in\Gamma$ if $\phi(\gamma)(v)=v$. Similarly, an edge $uv\in E$ is \emph{fixed} by $\gamma\in\Gamma$ if both $u$ and $v$ are fixed by $\gamma$ or if  $\phi(\gamma)(u)=v$ and $\phi(\gamma)(v)=u$.
For groups of order two, we will often just say that a vertex or edge is fixed if it is fixed by the non-trivial group element.

Note that if $(G,p)$ is a $\Gamma$-symmetric framework on $\Y$ with respect to $\tau$ and $\phi$, then there is no vertex fixed by an element of $\Gamma$ corresponding to a rotation $c_n$ about the $z$-axis or the inversion $\varphi$. The number of vertices that are fixed by the element in $\Gamma$ corresponding to the half-turn $c_{2}'$, or the reflections $\sigma$ and $\sigma'$ are denoted by $v_{2'}$, $v_\sigma$ and $v_{\sigma'}$, respectively. An edge of $G$ cannot be fixed by an element of $\Gamma$ that corresponds to a rotation $c_n$, $n\geq 3$, or an improper rotation $s_n$, $n\geq 3$. Hence we have separate columns for $c_2$ and $\varphi=s_2$ below. The number of edges that are fixed by the element in $\Gamma$ corresponding to the half-turns $c_2$ and $c_{2}'$, the reflections $\sigma$ and $\sigma'$ and the inversion $\varphi$ are denoted by $e_2$, $e_{2'}$, $e_\sigma$, $e_{\sigma'}$ and $e_{\varphi}$, respectively.

\begin{table}[ht]
    \centering
    \begin{tabular}{ |c|c|c|c|c|c|c|c|c| } 
\hline
$\Y$ & $\textrm{id}$ & $c_{n\geq3}$ & $c_{2}$ & $c_{2}'$ & $\sigma$ & $\sigma'$ & $s_{n\geq3}$ & $\varphi$ \\
\hline
$\chi(\Tilde{P}_{E})$ & $|E|+|V|$ & 0 & $e_{2}$ & $e_{2'}+v_{2'}$ & $e_{\sigma}+v_{\sigma}$ & $e_{\sigma'}+v_{\sigma'}$ & 0 & $e_{\varphi}$\\
\hline
$\chi(\tau \otimes P_{V})$ & $3|V|$ & 0 & 0 & $-v_{2'}$ & $v_{\sigma}$ & $v_{\sigma'}$ & 0 & 0\\
\hline
$\chi((\tau \otimes P_{v})^{(\T)})$ & 2 & 2 & 2 & -2 & 0 & 0 & 0 & 0\\
\hline
\end{tabular}
    \caption{Character table for symmetry operations of the cylinder.}
    \label{tab:CTCyl}
\end{table}

In the following proofs we shall use Theorem \ref{thm: character count} to draw conclusions from Table~\ref{tab:CTCyl}.

\begin{cor}\label{element count on cylinder}
If $(G,p)$ is a $\tau(\Gamma)$-symmetric isostatic framework on $\Y$, then $c_n\notin \tau(\Gamma)$ for any $n\geq 2$,  and $s_n\notin \tau(\Gamma)$ for any $n\geq 3$. Moreover,
\begin{itemize}
    \item if $c_{2}'\in \tau(\Gamma)$ then $e_{2'}=2$ and $v_{2'}=0$, or $e_{2'}=0$ and $v_{2'}=1$;
    \item if $\sigma\in \tau(\Gamma)$ or $\sigma'\in \tau(\Gamma)$ then $e_\sigma=0$ and $e_{\sigma'}=0$;
    \item if $\varphi\in \tau(\Gamma)$  then $e_{\varphi}=0$.
\end{itemize}
\end{cor}

\begin{proof}
Recall from Theorem~\ref{thm: character count} that if 
 $(G,p)$ is a $\tau(\Gamma)$-symmetric isostatic framework on $\Y$, then
$$\chi(\Tilde{P}_{E}) = \chi( \tau \otimes P_{V}) - \chi((\tau \otimes P_{V})^{(\T)}).$$

Clearly, by Table~\ref{tab:CTCyl}, this equation does not hold if $c_{n}\in \tau(\Gamma)$ for $n\geq 2$.
Further, since any $s_{n}$ symmetry  with $n\geq 3$ would also imply a $c_k$ symmetry for some $k \geq 2$,  $s_{n}\notin \tau(\Gamma)$ for $n\geq 3$.

Reading from Table~\ref{tab:CTCyl} we then draw the following conclusions.
If $c'_2\in\tau(\Gamma)$, then $e_{2'} + 2v_{2'} = 2$, so  either $e_{2'}=2$ and $v_{2'}=0$ or $e_{2'}=0$ and $v_{2'}=1$.
For both reflections, there is no restriction on the number of fixed vertices, but there cannot be an edge that is fixed by the reflection.
Finally for inversion, the table gives $e_{\varphi} = 0$.
\end{proof}

Note that the symmetry groups we can construct from these symmetry operations are the following \cite{altherz,Wall}:
\[
\tau(\Gamma) = \left\{
\begin{matrix*}[l]
     C_i = \{\textrm{id}, \varphi\}; \\
     C_s = \{\textrm{id}, \sigma\} \text{ or } \{\textrm{id}, \sigma'\}; \\ 
     C_2 = \{\textrm{id}, c'_2\}; \\ 
     C_{2v} = \{\textrm{id}, \sigma, \sigma', c'_2\}; \\
     C_{2h} = \{\textrm{id}, \sigma, c'_2, \varphi\}.
\end{matrix*}\right.\]

We are now able to use Corollary \ref{element count on cylinder}  to summarize the conclusions about $\tau(\Gamma)$-symmetric isostatic frameworks on $\Y$ for each possible symmetry group $\tau(\Gamma)$.

\begin{thm}\label{fixed e&v on cy}
Let $(G,p)$ be an isostatic $\tau(\Gamma)$-symmetric framework on $\Y$. Then $G$ is $\Gamma$-symmetric, $(2,2)$-tight and will satisfy the constraints in Table~\ref{table:fixed}.

\begin{center}
\begin{table}[ht]
    \begin{tabular}{|c|c|}
    \hline
    $\tau(\Gamma)$ & $\text{Number of edges and vertices fixed by symmetry operations}$ \\
    \hline
    $C_i$ & $e_{\varphi} = 0$ \\
    $C_s$ & $e_{\sigma} = 0$ \\ 
    $C_2$ & $e_{2'} = 2, v_{2'} = 0$ or $e_{2'} = 0, v_{2'} = 1$ \\ 
    $C_{2v}$ & $e_{\sigma} = e_{\sigma'} = 0, (e_{2'} = 2, v_{2'} = 0$ or $e_{2'} = 0, v_{2'} = 1)$ \\
    $C_{2h}$ & $e_{\sigma} = 0, e_{\varphi} = 0, e_{2'} = 2, v_{2'} = 0$ \\
    \hline
    \end{tabular}
    \caption{Fixed edge/vertex counts for symmetry operations on the cylinder.}
    \label{table:fixed}
    \end{table}
\end{center}
\end{thm}

\begin{proof} The graph $G$ must clearly be $\Gamma$-symmetric and $(2,2)$-tight (by \cite{NOP}). The statements for $C_{i}$, $C_{s}$, $C_{2}$ and $\mathcal{C}_{2v}$ follow immediately from Corollary \ref{element count on cylinder}.
For  a $C_{2h}$-symmetric framework $(G,p)$, note that if $v$ is a vertex of $G$ that is fixed by  $c'_2$, then $\sigma(v) = \varphi(v) \neq v$ will also be fixed by $c'_2$, so we cannot have $v_{2'}=1$. Thus, we must have $e_{2'} = 2, v_{2'} = 0$.
\end{proof}

\section{Rigidity preserving operations}\label{sec:ops}

Given a $\tau(\Gamma)$-symmetric isostatic framework on $\Y$, in this section we will construct larger $\tau(\Gamma)$-symmetric isostatic frameworks on $\Y$. To do this we introduce symmetry-adapted Henneberg-type graph operations. These operations are depicted in Figures \ref{fig:01}, \ref{fig:vk4} and \ref{fig:4cycle}.

Where it is reasonable to do so, we will work with a general group $\Gamma = \{\textrm{id} = \gamma_{0}, \gamma_{1}, \dots, \gamma_{t-1}\}$ and we will write $\gamma_{k}v$ instead of $\phi(\gamma_{k})(v)$ and often $\gamma_{k}(x,y,z)$ or $(x^{(k)},y^{(k)},z^{(k)})$ for $\tau(\gamma_{k})(p(v))$ where $p(v) = (x,y,z)$.
For a group of order two, it will be common to write $v' = \gamma(v)$ for $\gamma \in \Gamma \setminus\{\textrm{id}\}$.

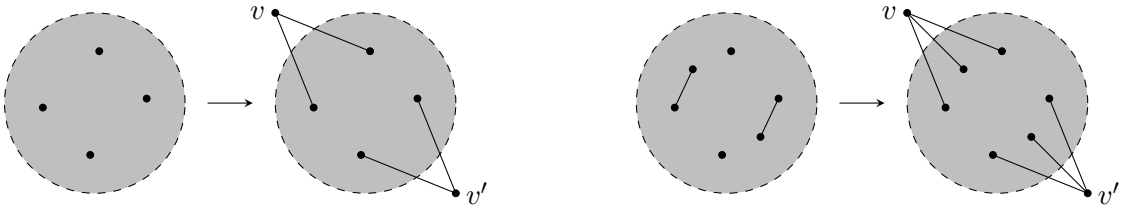
\begin{figure}[ht]
    \begin{center}
    \begin{tikzpicture}
  [scale=.3]
  
  \coordinate (n1) at (0,8);
  \coordinate (n2) at (1.7,3.8);
  \coordinate (n3) at (2.5,5.5);
  \coordinate (n4) at (4.2,6.3);
  \coordinate (n5) at (8,0);
  \coordinate (n6) at (3.8,1.7);
  \coordinate (n7) at (5.5,2.5);
  \coordinate (n8) at (6.3,4.2);
  \coordinate (n9) at (4,4);
  
 \draw[dashed, fill=lightgray] (n9) circle (4cm); 
 \draw[fill=black] (n2) circle (0.15cm);
 \draw[fill=black] (n4) circle (0.15cm);
 \draw[fill=black] (n6) circle (0.15cm);
 \draw[fill=black] (n8) circle (0.15cm);

  \coordinate (n11) at (12,8);
  \coordinate (n12) at (13.7,3.8);
  \coordinate (n13) at (14.5,5.5);
  \coordinate (n14) at (16.2,6.3);
  \coordinate (n15) at (20,0);
  \coordinate (n16) at (15.8,1.7);
  \coordinate (n17) at (17.5,2.5);
  \coordinate (n18) at (18.3,4.2);
  \coordinate (n19) at (16,4);
  
 \draw[dashed, fill=lightgray] (n19) circle (4cm); 
 \draw[fill=black] (n11) circle (0.15cm)
	    node[left] {$v$};
 \draw[fill=black] (n12) circle (0.15cm);
 \draw[fill=black] (n14) circle (0.15cm);
 \draw[fill=black] (n15) circle (0.15cm)
	    node[right] {$v'$};
 \draw[fill=black] (n16) circle (0.15cm);
 \draw[fill=black] (n18) circle (0.15cm);

  \foreach \from/\to in {n11/n12,n11/n14,n15/n16,n15/n18} 
    \draw (\from) -- (\to);
	    
  \coordinate (n21) at (28,8);
  \coordinate (n22) at (29.7,3.8);
  \coordinate (n23) at (30.5,5.5);
  \coordinate (n24) at (32.2,6.3);
  \coordinate (n25) at (36,0);
  \coordinate (n26) at (31.8,1.7);
  \coordinate (n27) at (33.5,2.5);
  \coordinate (n28) at (34.3,4.2);
  \coordinate (n29) at (32,4);
  
 \draw[dashed, fill=lightgray] (n29) circle (4cm); 
 \draw[fill=black] (n22) circle (0.15cm);
 \draw[fill=black] (n23) circle (0.15cm);
 \draw[fill=black] (n24) circle (0.15cm);
 \draw[fill=black] (n26) circle (0.15cm);
 \draw[fill=black] (n27) circle (0.15cm);
 \draw[fill=black] (n28) circle (0.15cm);

  \foreach \from/\to in {n22/n23,n28/n27} 
    \draw (\from) -- (\to);

  \coordinate (n31) at (40,8);
  \coordinate (n32) at (41.7,3.8);
  \coordinate (n33) at (42.5,5.5);
  \coordinate (n34) at (44.2,6.3);
  \coordinate (n35) at (48,0);
  \coordinate (n36) at (43.8,1.7);
  \coordinate (n37) at (45.5,2.5);
  \coordinate (n38) at (46.3,4.2);
  \coordinate (n39) at (44,4);

 \draw[dashed, fill=lightgray] (n39) circle (4cm); 
 \draw[fill=black] (n31) circle (0.15cm)
	    node[left] {$v$};
 \draw[fill=black] (n32) circle (0.15cm);
 \draw[fill=black] (n33) circle (0.15cm);
 \draw[fill=black] (n34) circle (0.15cm);
 \draw[fill=black] (n35) circle (0.15cm)
	    node[right] {$v'$};
 \draw[fill=black] (n36) circle (0.15cm);
 \draw[fill=black] (n37) circle (0.15cm);
 \draw[fill=black] (n38) circle (0.15cm);
	    
  \foreach \from/\to in {n31/n32,n31/n33,n31/n34,n35/n36,n35/n37,n35/n38} 
    \draw (\from) -- (\to);
    
\draw[-stealth] (9,4) -- (11,4);
\draw[-stealth] (37,4) -- (39,4);

\end{tikzpicture}
    \caption{Symmetrised 0- and 1-extensions adding new vertices $v$ and $v'$ in each case.}
    \label{fig:01}
    \end{center}
\end{figure}

\begin{figure}[ht]
\begin{tikzpicture}
  [scale=.28]
  
  \coordinate (n1) at (5,9);
  \coordinate (n2) at (4,5);
  \coordinate (n3) at (3,12);
  \coordinate (n4) at (7,13);
  \coordinate (n5) at (9,5);
  \coordinate (n6) at (10,9);
  \coordinate (n7) at (11,2);
  \coordinate (n8) at (7,1);
  \coordinate (n9) at (7,7);
  
 \draw[dashed, fill=lightgray] (n9) circle (7cm); 
 \draw[fill=black] (n1) circle (0.15cm)
	    node[right] {$v$};
 \draw[fill=black] (n2) circle (0.15cm);
 \draw[fill=black] (n3) circle (0.15cm);
 \draw[fill=black] (n4) circle (0.15cm);
 \draw[fill=black] (n5) circle (0.15cm)
 	    node[left] {$v'$};
 \draw[fill=black] (n6) circle (0.15cm);
 \draw[fill=black] (n7) circle (0.15cm);
 \draw[fill=black] (n8) circle (0.15cm);

  \coordinate (n12) at (24,5);
  \coordinate (n13) at (23,12);
  \coordinate (n14) at (27,13);
  \coordinate (n16) at (30,9);
  \coordinate (n17) at (31,2);
  \coordinate (n18) at (27,1);
  \coordinate (n19) at (27,7);
  \coordinate (n20) at (24,8);
  \coordinate (n21) at (24,10);
  \coordinate (n22) at (26,8);
  \coordinate (n23) at (26,10);
  \coordinate (n24) at (28,4);
  \coordinate (n25) at (28,6);
  \coordinate (n26) at (30,4);
  \coordinate (n27) at (30,6);
  
 \draw[dashed, fill=lightgray] (n19) circle (7cm); 
 \draw[fill=black] (n12) circle (0.15cm);
 \draw[fill=black] (n13) circle (0.15cm);
 \draw[fill=black] (n14) circle (0.15cm);
 \draw[fill=black] (n16) circle (0.15cm);
 \draw[fill=black] (n17) circle (0.15cm);
 \draw[fill=black] (n18) circle (0.15cm);
 \draw[fill=black] (n20) circle (0.15cm);
 \draw[fill=black] (n21) circle (0.15cm);
 \draw[fill=black] (n22) circle (0.15cm);
 \draw[fill=black] (n23) circle (0.15cm);
 \draw[fill=black] (n24) circle (0.15cm);
 \draw[fill=black] (n25) circle (0.15cm);
 \draw[fill=black] (n26) circle (0.15cm);
 \draw[fill=black] (n27) circle (0.15cm);
	  
\draw[-stealth] (15,7) -- (19,7);

  \foreach \from/\to in {n1/n2,n1/n3,n1/n4,n5/n6,n5/n7,n5/n8} 
    \draw (\from) -- (\to);
  \foreach \from/\to in {n22/n12,n21/n13,n23/n14,n25/n16,n26/n17,n24/n18,
  n20/n21,n20/n22,n20/n23,n21/n22,n21/n23,n22/n23,n24/n25,n24/n26,n24/n27,n25/n26,n25/n27,n26/n27} 
    \draw (\from) -- (\to);

\end{tikzpicture}
\caption{The symmetrised vertex-to-$K_4$ operation (in this case expanding the degree 3 vertices $v$ and $v'$).}
\label{fig:vk4}
\end{figure}
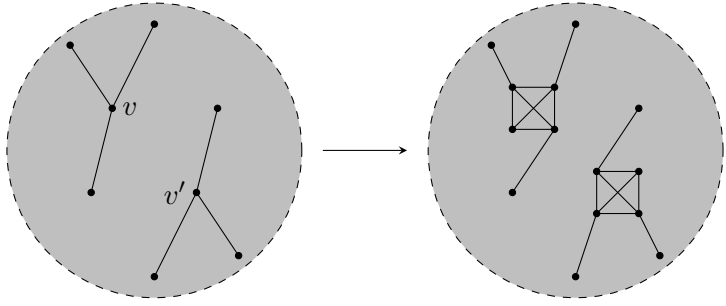

\begin{center}
\begin{figure}[ht]
\centering
\begin{tikzpicture}
  [scale=.28]
  
  \coordinate (n1) at (5,9);
  \coordinate (n2) at (3,5);
  \coordinate (n3) at (3.5,9);
  \coordinate (n31) at (5,10.5);
  \coordinate (n32) at (6.5,9);
  \coordinate (n33) at (5,7.5);
  \coordinate (n4) at (8,12);
  \coordinate (n5) at (9,5);
  \coordinate (n6) at (11,9);
  \coordinate (n7) at (10.5,5);
  \coordinate (n71) at (9,3.5);
  \coordinate (n72) at (7.5,5);
  \coordinate (n73) at (9,6.5);
  \coordinate (n8) at (6,2);
  \coordinate (n9) at (7,7);
  
 \draw[dashed, fill=lightgray] (n9) circle (7cm); 
 \draw[fill=black] (n1) circle (0.15cm);
 \draw[fill=black] (n2) circle (0.15cm);
 \draw[fill=black] (n4) circle (0.15cm);
 \draw[fill=black] (n5) circle (0.15cm);
 \draw[fill=black] (n6) circle (0.15cm);
 \draw[fill=black] (n8) circle (0.15cm);
	    
  \coordinate (n12) at (23,5);
  \coordinate (n13) at (22.5,10);
  \coordinate (n131) at (24,11.5);
  \coordinate (n132) at (27,8.5);
  \coordinate (n133) at (25.5,7);
  \coordinate (n14) at (28,12);
  \coordinate (n16) at (31,9);
  \coordinate (n17) at (27,5.5);
  \coordinate (n171) at (28.5,7);
  \coordinate (n172) at (31.5,4);
  \coordinate (n173) at (30,2.5);
  \coordinate (n18) at (26,2);
  \coordinate (n19) at (27,7);
  \coordinate (n21) at (24,10);
  \coordinate (n22) at (25.5,8.5);
  \coordinate (n25) at (28.5,5.5);
  \coordinate (n26) at (30,4);
  
 \draw[dashed, fill=lightgray] (n19) circle (7cm); 
 \draw[fill=black] (n12) circle (0.15cm);
 \draw[fill=black] (n14) circle (0.15cm);
 \draw[fill=black] (n16) circle (0.15cm);
 \draw[fill=black] (n18) circle (0.15cm);
 \draw[fill=black] (n21) circle (0.15cm);
 \draw[fill=black] (n22) circle (0.15cm);
 \draw[fill=black] (n25) circle (0.15cm);
 \draw[fill=black] (n26) circle (0.15cm);
	  
\draw[-stealth] (15,7) -- (19,7);

  \foreach \from/\to in {n1/n2,n1/n3,n1/n31,n1/n32,n1/n33,n1/n4,n5/n6,n5/n7,n5/n71,n5/n72,n5/n73,n5/n8} 
    \draw (\from) -- (\to);
  \foreach \from/\to in {n12/n21,n12/n22,n14/n21,n14/n22,n21/n13,n21/n131,n22/n132,n22/n133,
  n16/n25,n16/n26,n18/n25,n18/n26,n25/n17,n25/n171,n26/n172,n26/n173} 
    \draw (\from) -- (\to);

\end{tikzpicture}
\caption{The symmetrised vertex-to-$C_4$ operation. In this example each of the split vertices had degree 6 and the corresponding two new vertices have degree 4 each.}
\label{fig:4cycle}
\end{figure}
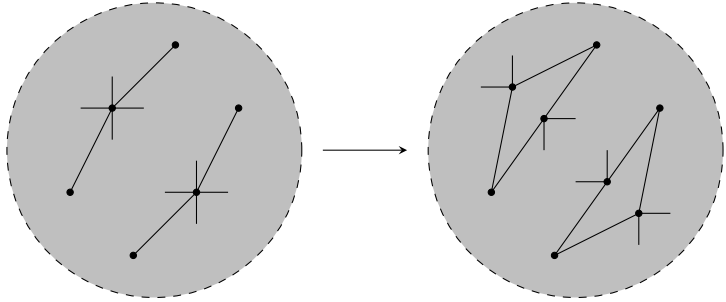
\end{center}

In each of the following operations we have a $\Gamma$-symmetric graph $(G,\phi)$ for a group $\Gamma$ of order $t$ and define a new  $\Gamma$-symmetric graph $(G^+,\phi^+)$.
We write $G=(V,E)$ and $G^+ = (V^+,E^+)$.
For all $\gamma \in \Gamma$ and $v \in V$, $\phi^+(\gamma) v = \phi(\gamma)v$.
A \emph{symmetrised 0-extension} creates a new $\Gamma$-symmetric graph $G^+$ by adding the $t$ vertices $\{v, \gamma v, \dots, \gamma_{t-1}v\}$ with $v$ adjacent to $v_{i},v_{j}$, and for each $k \in \{1,\dots,t-1\}$, $\gamma_{k}v$ adjacent to $\gamma_{k}v_{i},\gamma_{k}v_{j}$.
Let $e_i=x_{i}y_{i}$, $i=0\leq i \leq t-1$ be an edge orbit of $G$ of size $t$ under the action of $\Gamma$. Further let $z_0\neq x_0,y_0$ and let $z_i=\gamma_i z_0$ for $i=1,\ldots, t-1$. A \emph{symmetrised 1-extension} creates a new $\Gamma$-symmetric graph by deleting all the edges $e_i$ from $G$ and adding $t$ vertices $\{v, \gamma v, \dots, \gamma_{t-1}v\}$ with $v$ adjacent to $x_0,y_0$ and $z_0$, and $\gamma_iv$ adjacent to $x_{i},y_{i}$ and $z_{i}$ for $i=1,\ldots, t-1$. 
A \emph{symmetrised vertex-to-$C_{4}$} operation at the vertices $w, \gamma_{1}w, \dots, \gamma_{t-1}w$, creates a new $\Gamma$-symmetric graph $G^+=(V^+,E^+)$ where $V^+ = V \cup \{u,\dots, \gamma_{t-1}u\}$. 
The edge set changes such that if $w$ is adjacent to $v_1,\dots, v_r$ in $G$, $v_{1}, v_{2}$ are adjacent to both $w$ and the new vertex $u$, with $v_3,\dots,v_r$ adjacent to one of $w$ or $u$ in $E^+$.
Similarly $v_{1}^{(k)}, v_{2}^{(k)}$ are adjacent to both $w^{(k)}$ and $u^{(k)}$ and $v_3^{(k)},\dots,v_r^{(k)}$ are adjacent to one of $w^{(k)}$ or $u^{(k)}$ in $G^{+}$.
A \emph{symmetrised vertex-to-$K_{4}$} operation at the vertices $w, \gamma_{1}w, \dots, \gamma_{t-1}w$, creates a new $\Gamma$-symmetric graph $G^+$ with $V^+=V \cup \{a_0,b_0,c_0,\dots,a_{t-1}, b_{t-1},c_{t-1}\}$, where for each $1\leq i \leq t-1$, $\gamma_ia_0 = a_i,\gamma_ib_0 = b_i,\gamma_ic_0 = c_i$. 
If in $G$ the vertex $w$ is adjacent to $v_1,\dots, v_r$,  then $v_{i}$ is adjacent to some $d_i \in \{w,a,b,c\}$ in $G^{+}$ for each $i$. Similarly $v_{i}^{(k)}$ is adjacent to $d_i^{(k)}$ for all $k$.
Finally, we let $G^+[w,a_0,b_0,c_0] \cong K_{4}$ and $G^+[\gamma_iw,a_i,b_i,c_1] \cong K_{4}$ for all $i$.

For $\Gamma=\mathbb{Z}_2$, we introduce special cases of symmetrised extensions above.
A \emph{symmetrised fixed-vertex $0$-extension}, adds a single degree two vertex $v$ that is fixed.
The neighbours of the new vertex are not fixed, but are images of each other under the non-trivial group element.
A \emph{symmetrised fixed-vertex-to-$C_4$} operation at the fixed vertex $w$ creates a new graph $G^+ = G+u$, where $u$ is also a fixed vertex.
The edge set changes such that if $w$ is adjacent to $v_1,\dots, v_r$ in $G$, $v_{1}, v_{2}$ are adjacent to both $w$ and the new vertex $u$, with $v_3,\dots,v_r$ adjacent to one of $w$ or $u$ in $E^+$.

\subsection{Henneberg extensions}

To make the geometric statements in this section as general as possible, we sometimes show that the graph operations preserve $\tau(\Gamma)$-independence and sometimes $\tau(\Gamma)$-rigidity depending on the proof strategy. Note that for some symmetry groups $\tau(\Gamma)$, there are no $\tau(\Gamma)$-isostatic graphs and hence this distinction is important.

\begin{lem}\label{0-ext rigid}
Suppose $(G,\phi)$ is $\Gamma$-symmetric. Let $(G^{+},\phi^+)$ be obtained from $(G,\phi)$ by a symmetrised $0$-extension such that $v_i$ and $v_j$ are not fixed vertices and $v_i\neq \gamma_k v_j$ for any $k$. If $G$ is $\tau(\Gamma)$-independent on $\Y$, then $G^{+}$ is $\tau(\Gamma)$-independent on $\Y$.
\end{lem}

\begin{proof}
Write $G^{+} = G + \{v,\dots, \gamma_{t-1}v\}$, and let $v \in V^{+}$ be adjacent to $v_{i},v_{j}$, and for each $k \in \{1,\dots,t-1\}$, $\gamma_{k}v$ adjacent to $\gamma_{k}v_{i},\gamma_{k}v_{j}$. Since $G$ is $\tau(\Gamma)$-independent on $\Y$ we may choose $p$ so that $R_{\mathcal{Y}}(G,p)$ has linearly independent rows.
Define $p^{+}: V^{+} \to \R^{3}$ by $p^{+}(w)=p(w)$ for all $w\in V$, $p^{+}(v) = (x,y,z)$, and $p^{+}(\gamma_{k}v) = (x^{(k)},y^{(k)},z^{(k)})$. Write $p(v_{i}) = (x_{i},y_{i},z_{i})$, $p(v_{j}) = (x_{j},y_{j},z_{j})$. Then, $R_{\Y}(G^{+},p^{+}) =$\\
\[
\left[ \begin{tabular}{ c c c c c c c c c }
    $R_{\mathcal{Y}}(G,p)$ & & & & & & & & \\
    & $x - x_{i}$ & $y - y_{i}$ & $z - z_{i}$ & & & & &  \\
    *& $x - x_{j}$ & $y - y_{j}$ & $z - z_{j}$ & & & $\mathbf{0}$ & &  \\
    & $x$ & $y$ & 0 & & & & &  \\
    & & & & $\ddots$ & & & & \\
    & & & & &  $x^{(k)} - x_{i}^{(k)}$ & $y^{(k)} - y_{i}^{(k)}$ & $z^{(k)} - z_{i}^{(k)}$ & \\
    * & & $\mathbf{0}$ & & & $x^{(k)} - x_{j}^{(k)}$ & $y^{(k)} - y_{j}^{(k)}$ & $z^{(k)} - z_{j}^{(k)}$ & \\
    & & & & & $x^{(k)}$ & $y^{(k)}$ & 0 & \\
    & & & & & & & & $\ddots$
    \end{tabular} \right], \]
and hence the fact that $R_{\mathcal{Y}}(G^{+},p^{+})$ has linearly independent rows will follow once each $3 \times 3$ submatrix indicated above is shown to be invertible.
For the first such submatrix, one can see that is the case unless $p(v_j)$ lies on the intersection between the cylinder and the plane $A = \{(x,y,z)+a_1 (x,y,0) + a_2 (x-x_i,y-y_i,z-z_i)\}$.
Note that the hypotheses of the lemma guarantee that $p^+$ can be chosen in this way.
Since each $\tau(\gamma_{k})$ is an isometry, all of the other $t-1$ remaining submatrices are also invertible, and so $\textrm{rank } R_{\mathcal{Y}}(G^{+},p^{+}) = \textrm{rank }R_{\mathcal{Y}}(G,p) + 3t$.
Hence, if $G$ is $\tau(\Gamma)$-independent on the cylinder, so is $G^{+}$.
\end{proof}

We note that $p(v_j)$ could belong to the plane $A$ in the above proof when $v,v_i,v_j$ are in special positions.
Hence when some of $v,v_i,v_j$ are fixed by the symmetry or are images of one another under the symmetry, a symmetrised 0-extension may not preserve rigidity. 
In the following remark we note two cases when such symmetry exists but $R_{\mathcal{Y}}(G^{+},p^{+})$ has full rank. 

\begin{rem}\label{rem: Cs fixed 0-ext}
For a $\mathbb{Z}_2$-symmetric graph $G$ and symmetry group $\tau(\Gamma) = C_s$, let $G^+$ be defined as in either of the following way:
\begin{itemize}
    \item let $G^+ = G + \{v\}$ be obtained by a symmetrised fixed-vertex $0$-extension,
    \item let $G^+ = G + \{v,v'\}$ be obtained by a symmetrised $0$-extension, where $N(v) = \{v_i,v_j\} = N(v')$.
\end{itemize}

If $G$ is $C_s$-independent on $\Y$ then $G^{+}$ is $C_s$-independent on $\Y$.
\end{rem}

\begin{lem}\label{1-ext rigid}
Let $(G,\phi)$ be a $\Gamma$-symmetric graph, and $(G^{+},\phi^+)$ be obtained from $(G,\phi)$ by a symmetrised $1$-extension.
If $G$ is $\tau(\Gamma)$-rigid on $\Y$, then $G^{+}$ is $\tau(\Gamma)$-rigid on $\Y$. 
\end{lem}

\begin{proof}
Let $G^+$ be obtained from a symmetrised 1-extension on $G$, that is by deleting the edges $\{v_{1}v_{2},\dots, \gamma_{t-1}(v_{1}v_{2})\}$, and adding the vertices $\{v_{0},\dots, \gamma_{t-1}v_{0}\}$ where $v_0$ is adjacent to $v_{1}, v_{2}, v_{3}$ and each $\gamma_{i}v_{0}$ is adjacent to $\gamma_iv_{1}, \gamma_iv_{2}, \gamma_iv_{3}$.
Let $(G,p)$ be completely $\Gamma$-regular on $\Y$ and define $p^{+} = (p_{0},$ $p_{-1} = \gamma_1(p_0), \dots, p_{-t+1} =\gamma_{t-1}(p_{0}), p)$, where $(G^{+},p^{+})$ is completely $\Gamma$-regular. 
Suppose for a contradiction, $(G^{+}, p^{+})$ is not infinitesimally rigid on $\Y$.
Then any $\tau(\Gamma)$-symmetric framework of $G^{+}$ on $\Y$ will be infinitesimally flexible. We will use a sequence of $\tau(\Gamma)$-symmetric frameworks, moving only the points $\{p_{0},\dots, p_{-t+1}\}$.
First let $a,b$ be tangent vectors at $p_{1}$, with $b$ orthogonal to $p_{1} - p_{2}$ and $a$ orthogonal to $b$. 
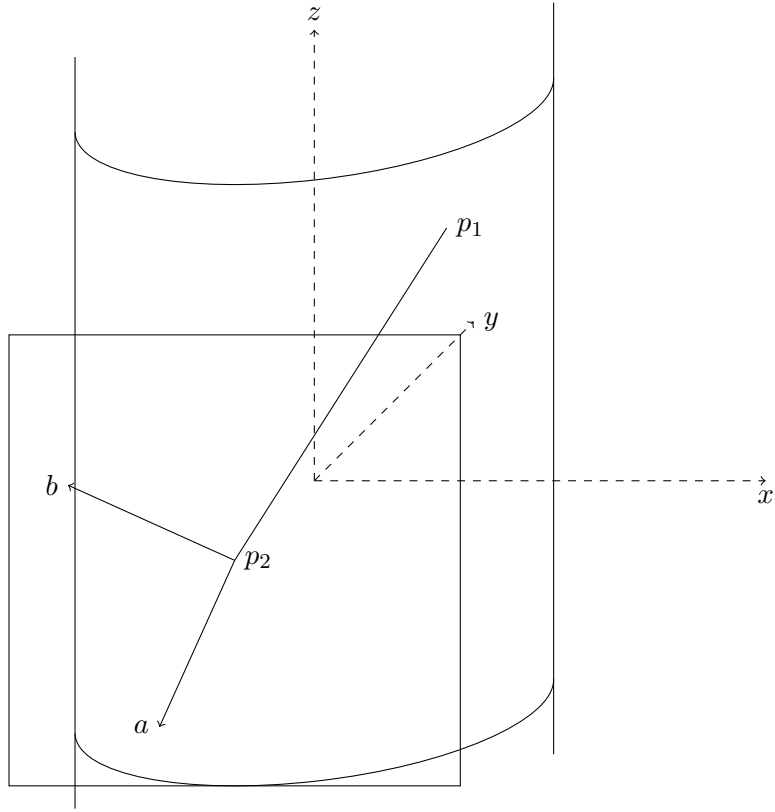
\begin{figure}
    \centering
    \begin{tikzpicture}[math3d]
    \newcommand{\h}{6}
    \newcommand{\rl}{6}
    \newcommand{\rh}{3}
    \begin{scope}[dashed,->,opacity=0.4]
        \draw (0,0,0) -- (\rl,0,0) node[below] {$x$};
        \draw (0,0,0) -- (0,\rl,0) node[right] {$y$};
        \draw (0,0,0) -- (0,0,\h) node[above] {$z$};
    \end{scope}
    
    \draw[opacity=0.4] (-2.81907786236,-1.02606042998,5) arc [start angle=-160,end angle=20,x radius=3,y radius=3];
    \draw[opacity=0.4] (-2.81907786236,-1.02606042998,-3) arc [start angle=-160,end angle=20,x radius=3,y radius=3];
    \draw[opacity=0.4] (2.81907786236,1.02606042998,-4) -- (2.81907786236,1.02606042998,6);
    \draw[opacity=0.4] (-2.81907786236,-1.02606042998,-4) -- (-2.81907786236,-1.02606042998,6);
    
    \draw[opacity=0.4] (3,-3,3) -- (3,-3,-3) -- (-3,-3,-3) -- (-3,-3,3) -- (3,-3,3);
    
    \draw (2.39590653014,-1.80544506946,4) node[right] {$p_{1}$} -- (0,-3,0) node[right] {$p_{2}$};
    \draw[->] (0,-3,0) -- (-2.21552018816,-3,1) node[left] {$b$};
    \draw[->] (0,-3,0) -- (-1,-3,-2.21552018816) node[left] {$a$};
\end{tikzpicture}
    \caption{Tangent vectors $a$, $b$ in relation to the edge $v_{1}v_{2}$.}
\end{figure}
Let $((G^{+}, p^{j}))_{j=0}^{\infty}$ where $p^{j} = (p_{0}^{j},\dots, p_{-t+1}^{j} = \gamma_{t-1}(p_{0}^{j}), p)$ is so that
$$\frac{\gamma_{i}(p_{1})-\gamma_{i}(p_{0}^{j})}{||\gamma_{i}(p_{1})-\gamma_{i}(p_{0}^{j})||} \to \gamma_{i}a$$
as $j \to \infty$, for each $i \in 0,\dots,t-1$. 
The frameworks $(G^{+}, p^{j})$ have a unit norm infinitesimal motion $u^{j}$ which is orthogonal to the space of trivial motions. 
By the Bolzano-Weierstrass theorem there is a subsequence of $(u^{j})$ which converges to a vector, $u^{\infty}$ say, also of unit norm.
We can discard and relabel parts of the sequence to assume this holds for the original sequence.
Looking at the limit $(G^{+}, p^{\infty})$, write $u^{\infty} = (u_{0}^{\infty},\dots, u_{-t+1}^{\infty}, u_{1}, u_{2}, \dots, u_{n})$, $p^{\infty} = (p_{0}^{\infty},\dots, p_{-t+1}^{\infty}, p_{1}, p_{2}, \dots, p_{n})$ with $\gamma_{i}(p_{0}^{\infty}) = \gamma_{i}(p_{1})$ for each $i$.

We show that $(u_{1}, u_{2})$ is an infinitesimal motion of the bar joining $p(v_1)$ and $p(v_2)$.
Since $p_{0}^{j}$ converges to $p_{1}$ in the $a$ direction, the velocities $u_{1}$ and $u_{0}^{\infty}$ have the same component in this direction, so $(u_{1}-u_{0}^{\infty})\cdot a = 0$.
Then $u_{1}-u_{0}^{\infty}$ is tangential to $\Y$ at $p_{1}$, and orthogonal to $a$, so it is orthogonal to $p_{1}-p_{0}^{\infty}$.
Also, $u_{2}-u_{0}^{\infty}$ is orthogonal to $p_{2}-p_{0}^{\infty}$.
Subtracting one from the other gives $u_{1} - u_{2}$ is orthogonal to $p_{1} - p_{2}$, which is the required condition for an infinitesimal motion.

Once again looking at $(G,p)$, we know the infinitesimal motion $u = (u_{1}, u_{2}, \dots, u_{n})$ is a trivial motion. 
In order to preserve the distances $d(p_{0}^{\infty}, p_{2})$ and $d(p_{0}^{\infty}, p_{3})$, $u_{0}^{\infty}$ is determined by $u_{2}$ and $u_{3}$.
Similarly $u_{-i}^{\infty}$ is determined by the motion vectors of $u$ which are present on the neighbours of $\gamma_i p_0$, for all $1 \leq i \leq t-1$.
We now see that $u_{0}^{\infty}$ agrees with $u_{1}$ and so $u^{\infty}$ is a trivial motion for $(G^{+}, p^{\infty})$.
However, since $u^{\infty}$ is a unit norm infinitesimal motion and orthogonal to the space of trivial motions, we have reached a contradiction. 
\end{proof}

\subsection{Further operations}

For a graph $G$ and pairwise vertex disjoint subgraphs $H_{1}, \dots, H_{k}$ of $G$, write $G//\{H_{i}\}_{i=1}^{k}$ for the graph derived from $G$ by contracting each of the subgraphs $H_{1}, \dots, H_{k}$ to their own single vertex. The resultant graph $G//\{H_{i}\}_{i=1}^{k}$ will have $|V(G)|-\sum_{i=1}^{k}(|V(H_{i})|-1)$ vertices and $|E(G)|-\sum_{i=1}^{k}|E(H_{i})|$ edges. When $k=1$ we will sometimes use the more common notation $G/H_1$.

\begin{lem}\label{minrigidK_4}
Suppose $(G,\phi)$ is $\Gamma$-symmetric and $H \leq G$ is a copy of $K_{4}$. Further, suppose for all $\gamma \in \Gamma \setminus \{\textrm{id}\}$, we have that $V(H) \cap V(\gamma H) = \emptyset$.
If $G//\{\gamma_{i}(H)\}_{i = 0}^{t-1}$ is $\tau(\Gamma)$-isostatic on $\mathcal{Y}$, then $G$ is $\tau(\Gamma)$-isostatic on $\mathcal{Y}$.
\end{lem}

\begin{proof}
 Let $|V| = n$ and $(G,p)$ be a $\tau(\Gamma)$-symmetric framework on $\mathcal{Y}$ which is completely $\Gamma$-regular. Further, let the vertices of $H$ be $x,y,z,w$. Suppose $p = (p(v_{1}), \dots, p(v_{n}))$, labelling so that $$V(\gamma_{i}(H)) = \{\gamma_{i}x = v_{4i+1}, \gamma_{i}y = v_{4i+2}, \gamma_{i}z = v_{4i+3}, \gamma_{i}w = v_{4i+4}\}$$
 for each $i=1,\ldots t-1$. 
Define a set of graphs $\{G_{j}\}_{j=0}^{t}$ by 
\[
G_{j} = \left\{
\begin{matrix*}[l]
     G//\{\gamma_{i}(H)\}_{i = j}^{t-1} & \text{if } j = 0, \dots, t-1; \\
     G & \text{if } j=t.
\end{matrix*}\right\}\]
where $\gamma_0=\textrm{id}$.
We want to show by induction that for $0\leq j \leq t-1$, if $G_{j}$ is isostatic on $\mathcal{Y}$, then $G_{j+1}$ is isostatic on $\mathcal{Y}$.
Then repeating this method, we show $G_{t} := G$ will be isostatic and $\tau(\Gamma)$-symmetric on $\mathcal{Y}$.
For each $0\leq i \leq t-1$, let the vertices $v_{4i+1},v_{4i+2},v_{4i+3},v_{4i+4}$ in $G$ contract to $v_{4i+1}$ in $\{G_{0},\dots,G_{i-1}\}$.\footnote{In the graph $G_{j}$, $j$ can be seen as a count on the number of $K_{4}$ copies of $H$ that are not contracted.}
We start by writing
\[R_{\mathcal{Y}}(G_{1},p|_{G_{1}}) = 
\left( \begin{tabular}{ c c }
        $R_{\mathcal{Y}}(\gamma_{0}(H),p|_{\gamma_{0}(H)})$ & $0$ \\
        $M_{1}(p)$ & $M_{2}(p)$ \\
    \end{tabular} \right) \]
where $M_{2}(p)$ is a $3(n-3t-1)$ square matrix, since
$|V(G_{1})| = n - 3(t-1)$ and so $M_{2}(p)$ has $3(n - 3(t-1)) - 12$ columns, and
$|E(G_{1})| = |E| - 6(t-1)= 2n-6t+4$ so $M_{2}(p)$ has $2n-6t+4+(n-3(t-1))-(6+4)$ rows. 
For a contradiction, suppose that $G_{1}$ is not $\tau(\Gamma)$-isostatic.
Then there exists a non-trivial infinitesimal motion $m$ of $(G_{1},p|_{G_{1}})$.
Since $(H,p|_{H})$ is infinitesimally rigid on $\Y$, we may suppose that 
$$m = (0,0,0,0, m_{5}, m_{9}, \dots, m_{4t+1}, m_{4t+2},\dots, m_{n}).$$
Consider the realisation $(G_{1},\hat{p})$  such that 
$$\hat{p} = (p(v_{1}), p(v_{1}), p(v_{1}), p(v_{1}), p(v_{5}), p(v_{9}), \dots, p(v_{4t+1}), p(v_{4t+2}), \dots, p(v_{n}))$$ 
and define $(G_{0}, p^{*})$ by letting $$p^{*} = (p(v_{1}), p(v_{5}), p(v_{9}), \dots, p(v_{4t+1}), p(v_{4t+2}), \dots, p(v_{n})).$$
By construction $(G_{0},p^{*})$ is completely $\Gamma$-regular, so it is $\tau(\Gamma)$-isostatic on $\mathcal{Y}$.
Now, $M_{2}(p)$ is square with the nonzero vector $(m_{5},m_{9},\dots,m_{4t+1},m_{4t+2},\dots, m_{n}) \in \ker M_{2}(p)$.
Hence $\textrm{rank} M_{2}(p) < 3(n-3t-1)$.
Since $(G,p)$ is completely $\Gamma$-regular, we also have $\textrm{rank} M_{2}(\hat{p}) < 3(n - 3t-1)$ and hence there exists a nonzero vector $\hat{m} \in \ker M_{2}(\hat{p})$.
Therefore we have 
\[R_{\mathcal{Y}}(G_{0},p^{*})
\left(\begin{tabular}{ c }
        $0$\\
        $\hat{m}$ \\
    \end{tabular} \right)
=
\left( \begin{tabular}{ c c }
        $p(v_{1})$ & $0$ \\
        * & $M_{2}(\hat{p})$ \\
    \end{tabular} \right)
\left(\begin{tabular}{ c }
        $0$\\
        $\hat{m}$ \\
    \end{tabular} \right)
= 0,\]
contradicting the infinitesimal rigidity of $(G_{0},p^{*})$.
We continue the above process inductively, writing $R_{\mathcal{Y}}(G_{j},p)$ as 
\[\left( \begin{tabular}{ c c }
        $R_{\mathcal{Y}}(\gamma_{j-1}(H),p|_{\gamma_{j-1}(H)})$ & $0$ \\
        $L_{1}(p)$ & $L_{2}(p)$ \\
    \end{tabular} \right) \]
where $L_{2}(p)$ is a $3(n-3(t-j)-4)$ square matrix.
From the same contradiction argument as before, we have that $(G_{j},p)$ is isostatic, and by noting that $G_{t}$ will be $\tau(\Gamma)$-symmetric, we finish the proof.
\end{proof}

The proof of the following lemma works with a similar strategy as is applied in Lemma \ref{minrigidK_4}.
For the first bullet point of the lemma, for a $C_i$-symmetric graph, we additionally need to perform an inverse (non-symmetric) $0$-extension on the vertex resulting from the contraction of $G_1$.

\begin{lem}\label{2addition rigid}
Suppose $(G_1,\phi_1)$ and $(G_2,\phi_2)$ are $\Gamma$-symmetric graphs with $G_{1} = (V_1, E_1)$ and $G_{2} = (V_2, E_2)$.
\begin{itemize}
    \item For $\tau(\Gamma) = C_i$, let $(G,\phi)$ be the $\Gamma$-symmetric graph with $V(G) = V_{1} \cup V_{2}$ and $E(G) = E_{1} \cup E_{2} \cup \{e_{1}, e_{2}\}$, and $\phi$ defined so that $\phi(\gamma)|_{V_i} = \phi_i(\gamma)$ for $i = 1,2$ and all $\gamma \in \Gamma$; additionally $e_{1} = xy, e_{2} = x'y'$ for any $x\in V_{1}, y\in V_{2}$.
    \item For $\tau(\Gamma) \in \{C_2,C_s\}$, suppose $G_2$ has a fixed vertex $v$ with neighbours $x_1, x_1', \dots, x_k, x_k'$.
    Define $(G,\phi)$ to be the $\Gamma$-symmetric graph with vertex set $V = V_1 \cup V_2 \setminus \{v\}$, and edge set $E$ obtained from $E_1 \cup E_2$ by deleting the edges $vx_1, vx_1', \dots, vx_k, vx_k'$ and replacing them with the edges $x_1y_1, x_1'y_1', \dots, x_ky_k, x_k'y_k'$ for some not necessarily distinct $y_1, y_1', \dots, y_k, y_k' \in V_1$, and $\phi$ being induced by $\phi_1,\phi_2$, similar to the above.
\end{itemize}
If $G_1$ and $G_{2}$ are $\tau(\Gamma)$-rigid on $\Y$, then $G$ is $\tau(\Gamma)$-rigid on $\Y$.
\end{lem}

\begin{proof}
We prove the two statements simultaneously. Let $|V| = n$ and $(G,p)$ be a completely $\tau(\Gamma)$-regular framework on $\mathcal{Y}$. Put $p = (p(v_{1}), \dots, p(v_{n}))$ labelling so that $V_{1} = \{v_{1},\dots, v_{r}\}$ and $V_{2} = \{v_{r+1}, \dots, v_{n}\}$.
As in Lemma \ref{minrigidK_4}, we write
\[R_{\mathcal{Y}}(G,p) = 
\left( \begin{tabular}{ c c }
        $R_{\mathcal{Y}}(G_{1},p|_{G_{1}})$ & $0$ \\
        $M_{1}(p)$ & $M_{2}(p)$ \\
    \end{tabular} \right) \]
where $M_{2}(p)$ is a $3(n-r)$ square matrix.
We repeat the same arguments as before to show $G$ is rigid.
For a contradiction, suppose that $G$ is not rigid.
Then there exists some non-trivial infinitesimal motion $m$ of $(G,p)$.
Since $(G_{1},p|_{G_{1}})$ is $\tau(\Gamma)$-rigid on $\Y$, we may suppose that $m = (0,\dots,0, m_{r+1},\dots, m_{n})$. 
Consider the realisation $(G,\hat{p})$  such that $\hat{p} = (p(v_{1}),\dots, p(v_{1}), p(v_{r+1}), \dots, p(v_{n}))$ and define $(G/G_{1}, p^{*})$ by letting $p^{*} = (p(v_{1}), p(v_{r+1}), \dots, p(v_{n}))$.
By construction $(G/G_1,p^{*})$ is completely regular, so $(G/G_{1}, p^{*})$ is independent on $\mathcal{Y}$.

Now, $M_{2}(p)$ is square with the nonzero vector $(m_{1},m_{r+1},\dots, m_{n}) \in \ker M_{2}(p)$.
Hence $\textrm{rank} M_{2}(p) < 3(n-r)$.
Since $(G/G_1,p^{*})$ is completely $\tau(\Gamma)$-regular, we also have $\textrm{rank} M_{2}(\hat{p}) < 3(n-r)$ and hence there exists a nonzero vector $\hat{m} \in \ker M_{2}(\hat{p})$.
Therefore we have 
\[(R_{\mathcal{Y}}(G/G_{1},p^{*}))
\left(\begin{tabular}{ c }
        $0$\\
        $\hat{m}$ \\
    \end{tabular} \right)
=
\left( \begin{tabular}{ c c }
        $p(v_{1})$ & $0$ \\
        * & $M_{2}(\hat{p})$ \\
    \end{tabular} \right)
\left(\begin{tabular}{ c }
        $0$\\
        $\hat{m}$ \\
    \end{tabular} \right)
= 0,\]
contradicting the rigidity of $(G/G_{1},p^{*})$.
Note that in the $C_i$-symmetric case, $G/G_{1}$ is the graph obtained from $G_{2}$ by a (non-symmetrised) 0-extension.
Hence, we know that if $G_{1}$ and $G_{2}$ are $\tau(\Gamma)$-rigid on $\Y$, then $G/G_{1}$ is rigid and so $G$ is $\tau(\Gamma)$-rigid.
\end{proof}

\begin{lem}\label{cycle rigid}
Let $(G,p)$ be a $\tau(\Gamma)$-symmetric and independent framework.
Let $w \in V$ be adjacent to $v_{1}, \dots, v_{r}$.
Suppose that $p(w) - p(v_{1})$, $p(w) - p(v_{2})$, and $n(w)$ are linearly independent.
Let $(G^{+},\phi^+)$ be obtained by performing a symmetrised vertex-to-$C_{4}$ operation at the vertices $w, \gamma_{1}w, \dots, \gamma_{t-1}w$.
Let $p^{+}(v) = p(v)$ for all $v \in V \setminus \{\gamma_{k} w | k \in \{0,\dots , t-1\}\}$, and $p^{+}(\gamma_{k} w) = p^{+}(\gamma_{k} u) = p(\gamma_{k} w)$ for all $k$. 
Then $(G^+,p^+)$ is independent.
\end{lem}

\begin{proof}
We will construct $R_{\mathcal{Y}}(G^{+}, p^{+})$ from $R_{\mathcal{Y}}(G, p)$ by a series of matrix row operations. 
We first add $3t$ zero columns to $R_{\mathcal{Y}}(G, p)$ for the new vertices $\{\gamma_{k} u\}$.
Then add $3t$ rows to this matrix, for the edges $\gamma_{k} u \gamma_{k}v_{1}, \gamma_{k} u \gamma_{k}v_{2}$, and the normal vectors to the surface at the points $p(\gamma_{k} u)$.
Since $p(w) - p(v_{1})$, $p(w) - p(v_{2})$, $n(w)$ are linearly independent (and, hence, so are each of the $p(\gamma_{k}w) - p(\gamma_{k}v_{1})$, $p(\gamma_{k}w) - p(\gamma_{k}v_{2})$, $n(\gamma_{k}w)$), $\textrm{rank} R_{\mathcal{Y}}(G^{+}, p^{+}) = \textrm{rank} R_{\mathcal{Y}}(G, p) + 3t$.
This gives the matrix $M$ of the form:
\[M = 
\left[ \begin{tabular}{ c c c }
        * & $p(w)-p(v_{1})$ & $0$ \\
        * & $p(w)-p(v_{2})$ & $0$ \\
         & $\vdots$ & \\
        * & $p(w)-p(v_{i})$ & $0$\\
         & $\vdots$ & \\
        * & $0$ & $p(u)-p(v_{1})$\\
        * & $0$ & $p(u)-p(v_{2})$\\
         & $\vdots$ & \\
         \hline
         * & $n(w)$ & $0$ \\
         * & $0$ & $n(u)$\\
           & $\vdots$ & 
    \end{tabular} \right], \]
where the columns given are for the vertices $w$ and $u$, and rows given for the edges $wv_{1}, wv_{2}, wv_{i}, uv_{1}, uv_{2}$ and normal vectors to the surface at $w$ and $u$.
There would be similar columns for each pair $\gamma_{k}w$ and $\gamma_{k}u$.
This is the rigidity matrix for a graph generated from $G$ by a $\tau(\Gamma)$-symmetric vertex-to-$C_{4}$ operation where $v_{i}w$ is an edge for all $3 \leq i \leq r$.
We wish to show that removing the edges $\{\gamma_{k}w\gamma_{k}v_{i} : k = 0,\dots, t-1\}$ and replacing them with the edges $\{\gamma_{k}u\gamma_{k}v_{i}  : k = 0,\dots, t-1\}$ preserves $\tau(\Gamma)$-independence.

Since $p(w) - p(v_{1})$, $p(w) - p(v_{2})$, and $n(w)$ are linearly independent and span $\R^{3}$, there exists $\alpha, \beta, \gamma \in \R$ such that 
$$p(w)-p(v_{i}) = \alpha(p(w) - p(v_{1})) + \beta (p(w) - p(v_{2})) + \gamma n(w).$$ 
Hence we perform row operations as follows.
From the row of $wv_{i}$, subtract $\alpha$ multiples of the row of $wv_{1}$, $\beta$ multiples of the row of $wv_{2}$, $\gamma$ multiples of the row for the normal vector of $w$.
Then to the row of $wv_{i}$, add $\alpha$ multiples of the row of $uv_{1}$, $\beta$ multiples of the row of $uv_{2}$, $\gamma$ multiples of the row for the normal vector of $u$.
Since $p(w) = p(u)$, when we do this to every neighbour $v_{i}$ of $u$, and similarly $\gamma_{k}v_{i}$ of $\gamma_{k}u$ (since all $\tau(\gamma_{k})$ are isometries of $\R^{3}$ that preserve the cylinder, the same $\alpha, \beta, \gamma$ work for the symmetric copies) in $G^{+}$, we obtain $R_{\Y}(G^{+}, p^{+})$.
The row operations preserve $\tau(\Gamma)$-independence, giving the desired result.
\end{proof}

When considering $C_s$-symmetric frameworks, we will use a special case of Lemma~\ref{cycle rigid} which we record in the following remark.

\begin{rem}\label{rem: C4 mirror}
Let $(G,p)$ be a $C_s$-symmetric and independent framework with $w\in V$ be fixed by $\sigma$ and adjacent to $v_1, \dots, v_r$.
Suppose that $p(w) - p(v_1)$, $p(w) - p(v_{1}')$, and $n(w)$ are linearly independent.
Let $G^{+}$ be obtained by performing a symmetrised fixed-vertex-to-$C_{4}$ operation at $w$, so that $v_{1}, v_{1}'$ are adjacent to both $w$ and the new vertex $u$ also fixed by $\sigma$ in $G^{+}$.
Let $p^{+}(v) = p(v)$ for all $v \in V$, and $p^{+}(u) = p(w)$.
Then $(G^+,p^+)$ is independent.
\end{rem}

For the case when the group is $C_2$, we will need one more operation.
A \textit{double 1-extension} on a $\mathbb{Z}_2$-symmetric graph $G$ is the combination of two non-symmetric 1-extensions: the first creates a new graph $G^+$ by removing a fixed edge $e=vv'$ of $G$, adding a new vertex, say $w$, of degree three adjacent to $v,v'$ and some other vertex $y$; followed by another non-symmetric 1-extension on $G^+$, namely removing $wv'$ and adding a new vertex $w'$ with 3 incident edges chosen so that $v' = \varphi(v)$.
See Figure \ref{double 1-ext}.

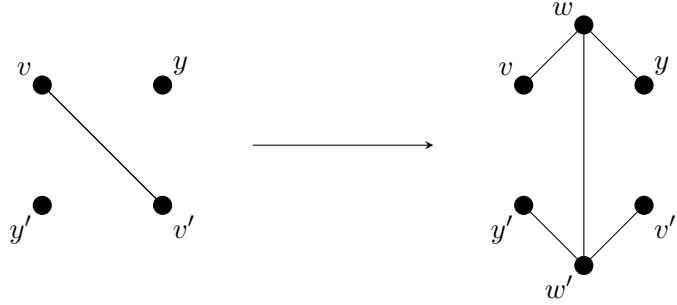
\begin{figure}
    \centering
    \begin{tikzpicture}
  [scale=.8,auto=left]
  
  \coordinate (n1) at (9,4);
  \coordinate (n2) at (8,3);
  \coordinate (n3) at (10,3);
  \coordinate (n4) at (8,1);
  \coordinate (n5) at (10,1);
  \coordinate (n6) at (9,0);
  
 \draw[fill=black] (n1) circle (0.15cm)
    node[above left] {$w$};
 \draw[fill=black] (n2) circle (0.15cm)
    node[above left] {$v$};
 \draw[fill=black] (n3) circle (0.15cm)
    node[above right] {$y$};
 \draw[fill=black] (n4) circle (0.15cm)
    node[below left] {$y'$};
 \draw[fill=black] (n5) circle (0.15cm)
    node[below right] {$v'$};
 \draw[fill=black] (n6) circle (0.15cm)
    node[below left] {$w'$};

  \foreach \from/\to in {n1/n6,n1/n2,n1/n3,n6/n4,n6/n5} 
    \draw (\from) -- (\to);

  \coordinate (n22) at (0,3);
  \coordinate (n23) at (2,3);
  \coordinate (n24) at (0,1);
  \coordinate (n25) at (2,1);
  
 \draw[fill=black] (n22) circle (0.15cm)
    node[above left] {$v$};
 \draw[fill=black] (n23) circle (0.15cm)
    node[above right] {$y$};
 \draw[fill=black] (n24) circle (0.15cm)
    node[below left] {$y'$};
 \draw[fill=black] (n25) circle (0.15cm)
    node[below right] {$v'$};

  \foreach \from/\to in {n25/n22} 
    \draw (\from) -- (\to);

\foreach \from/\to in {n25/n22} 
    \draw (\from) -- (\to);

  \coordinate (n31) at (3.5,2);
  \coordinate (n32) at (6.5,2);
\draw [-stealth] (n31) -- (n32);

\end{tikzpicture}
    \caption{A double 1-extension which deletes a fixed edge, and adds a new fixed edge between two degree 3 vertices.}
    \label{double 1-ext}
\end{figure}

\begin{lem}\label{double 1-ext rig}
Let $(G,\phi)$ be a $\Gamma$-symmetric graph (where $\Gamma=\mathbb{Z}_2$), with fixed edge $vv'$.
Let $(G^+,\phi^+)$ be the graph with vertex set $V^+ = V + \{w,w'\}$, and edge set $E^+ = E - vv' + \{wv,wy,w'v',w'y',ww'\}$, $\phi^+(\gamma)|_V = \phi(\gamma)$ for all $\gamma \in \mathbb{Z}_2$.
If $G$ is $C_2$-rigid on the cylinder then $G^+$ is too.
\end{lem}

\begin{proof}
Let $G^+$ be obtained from a double 1-extension on $G$, that is by deleting the edge $vv'$, and adding the vertices $w,w'$ where $w$ is a node adjacent to $v,y, w'$ and $w'$ is adjacent to $v', y', w$.
Let $c = \tau(\gamma)$ be the half-turn in $\tau(\Gamma)$ (recall that previously $c$ was called either $c_2$ or $c_2'$ depending on the position of the rotational axis relative to the cylinder). 
Let $p_0$ and $c(p_{0})$ be the positions of the vertex $w$ and its symmetric copy. Let $(G,p)$ be completely $\Gamma$-regular on $\Y$ and define $p^{+} = (p_0, p_{-1}, p)$, so that $(G^{+},p^{+})$ is completely $\Gamma$-regular. We let $p(v) = p_1$, $p(v') = p_2 = c(p_1)$, $p(y) = p_3$, and $p(y') = p_4=c(p_3)$. 

Suppose for a contradiction that $(G^{+}, p^{+})$ is not infinitesimally rigid on $\Y$.
Then any $\tau(\Gamma)$-symmetric framework of $G^{+}$ on $\Y$ will be infinitesimally flexible. We will use a sequence of $\tau(\Gamma)$-symmetric frameworks, moving only the points $\{p_{0},c(p_{0})\}$.
Let $\mathcal{T}$ denote the tangent plane to $\mathcal{Y}$ at $p_1$. Choose $a$ and $b$ to be orthogonal vectors in $\mathcal{T}$ such that
$b$ is orthogonal to $p_1 - p_2$. 
Let $((G^{+}, p^{j}))_{j=0}^{\infty}$ be a sequence of frameworks where $p^{j} = (p_{0}^{j}, c(p_{0}^{j}), p)$ is taken so that $$\frac{c^{i}(p_1)-c^{i}(p_{0}^{j})}{||c^{i}(p_1)-c^{i}(p_{0}^{j})||} \to c^{i}a$$
as $j \to \infty$, for each $i \in 0,1$. 
The frameworks $(G^{+}, p^{j})$ have a unit norm infinitesimal motion $u^{j}$ which is orthogonal to the space of trivial motions of $\mathcal{Y}$. 
By the Bolzano-Weierstrass theorem there is a subsequence of $(u^{j})$ which converges to a vector, $u^{\infty}$ say, also of unit norm.
We can discard and relabel parts of the sequence to assume this holds for the original sequence. For convenience, in an infinitesimal motion $u$, we will denote the instantaneous velocity at $c(p_{0})$ by $u_{-1}$.
Looking at the limit $(G^{+}, p^{\infty})$, write $u^{\infty} = (u_{0}^{\infty}, u_{-1}^{\infty}, u_{1}, u_{2}, \dots, u_{n})$, $p^{\infty} = (p_{0}^{\infty}, c(p_{0}^{\infty}), p_{1}, p_{2}, \dots, p_{n})$ with $p_{0}^{\infty} = p_{1}$ and $c(p_{0}^{\infty}) = p_2$.

We show that $(u_{1}, u_{2})$ is an infinitesimal motion of the bar joining $p(v)$ and $p(v')$.
Since $p_{0}^{j}$ converges to $p_{1}$ in the $a$ direction, the velocities $u_{1}$ and $u_{0}^{\infty}$ have the same component in this direction, so $(u_{1}-u_{0}^{\infty})\cdot a = 0$.
Then $u_{1}-u_{0}^{\infty}$ is tangential to $\Y$ at $p_{1}$, and orthogonal to $a$, so it is orthogonal to $p_{1}-p_{0}^{\infty}$.
Also, since there is a bar joining $p(w) = p_0^{\infty}$ and $p(w') = c(p_0^{\infty})$ as well as a bar joining $p(w')$ and $p(v') = p_2$, $u_{0}^{\infty}-u_{-1}^{\infty}$ is orthogonal to $p_{0}^{\infty}-c(p_{0}^{\infty})$ and $u_{2}-u_{-1}^{\infty}$ is orthogonal to $p_{2}-c(p_{0}^{\infty})$.
We may express this as
$$\langle u_{1}-u_{0}^{\infty}, p_{1}-p_{0}^{\infty} \rangle = \langle u_{2}-u_{-1}^{\infty}, p_{2}-c(p_{0}^{\infty}) \rangle = \langle u_{0}^{\infty}-u_{-1}^{\infty}, p_{0}^{\infty}-c(p_{0}^{\infty}) \rangle = 0.$$
It follows that
$$0=\langle u_{1}-u_{0}^{\infty}, p_{1}-p_{0}^{\infty} \rangle - \langle u_{2}-u_{-1}^{\infty}, p_{2}-c(p_{0}^{\infty}) \rangle + \langle u_{0}^{\infty}-u_{-1}^{\infty}, p_{0}^{\infty}-c(p_{0}^{\infty}) \rangle = \langle u_1 - u_2 , p_1 - p_2 \rangle.$$
which is the required condition for an infinitesimal motion.

Once again looking at $(G,p)$, we know the infinitesimal motion $u = (u_{1}, u_{2}, \dots, u_{n})$ is a trivial motion.
In order to preserve the distance $d(p_{0}^{\infty}, p_{3})$, $u_0^{\infty}$ takes one of two values, representing rotating or translating the bar between $p(w)$ and $p(y)$.
Additionally, $(u_{1}-u_{0}^{\infty})\cdot a = 0$ determines $u_{0}^{\infty}$.
Similarly, $u_{-1}^{\infty}$ is determined by $d(c(p_{0}^{\infty}), p_{4})$ and $(u_{2}-u_{-1}^{\infty})\cdot \gamma_2 a = 0$.
Finally, since $\langle u_{0}^{\infty}-u_{-1}^{\infty}, p_{0}^{\infty}-c(p_{0}^{\infty}) \rangle = 0$, $\langle u_1 - u_2 , p_1 - p_2 \rangle = 0$, and $p_{0}^{\infty} = p_1, c(p_{0}^{\infty}) = p_2$, we have that $u_{0}^{\infty}$ agrees with $u_{1}$ and $u_{-1}^{\infty}$ agrees with $u_{2}$, so $u^{\infty}$ is a trivial motion for $(G^{+}, p^{\infty})$.
This gives a contradiction since $u^{\infty}$ is a unit norm infinitesimal motion orthogonal to the space of trivial motions of $\mathcal{Y}$.
\end{proof}

\section{Symmetric isostatic graphs}
\label{sec:merged}

In the next four sections we prove our main results. These are combinatorial characterisations of when a symmetric graph is isostatic on $\Y$ for the symmetry groups $C_i = \{\textrm{id}, \varphi\}$, $C_2 = \{\textrm{id}, c_2'\}$ and $C_s = \{\textrm{id}, \sigma\}$.
These results give a precise converse to the necessary conditions developed in Section \ref{sec:necrep} and utilise the geometric operations of the previous section. 
In order to prove the results we need to develop some combinatorics. In this section we work as generally as possible among the three groups. Then the three subsequent sections specialise one by one to the specific symmetry groups.

\subsection{Graph theoretic preliminaries and base graphs.}
We will use standard graph theoretic terminology. 
For a graph $G=(V,E)$, $\delta(G)$ will denote the minimum degree of $G$, $N(v)$ and $N[v]$ will denote the open and closed neighbourhoods of a vertex $v\in V$ respectively (so $N[v]=N(v)\cup \{v\}$).
As is common, $W_k$ will denote the wheel over a cycle on $k-1$ vertices ($k \geq 4$) and $Wd(n,k)$ will denote the windmill, which is $k$ copies of $K_n$ all joined at a single vertex. 
The degree of a vertex $v$ is denoted $d_G(v)$.
For $X\subset V$ we will use $i_G(X)$ to denote the number of edges in the induced subgraph $G[X]$ and the set $X$ will be called \emph{k-critical}, for $k\in \mathbb{N}$, if $i_G(X)=2|X|-k$.
For $X,Y\subset V$, $d_G(X,Y)$ will denote the number of edges of the form $xy$ with $x\in X\setminus Y$ and $y\in Y\setminus X$.
We will often suppress subscripts when the graph is clear from the context and use $d(v)$, $i(X)$ and $d(X,Y)$.
We also say that a subset $X$ of $V$ is \emph{$\Gamma$-symmetric} if $(G[X],\phi)$ is a $\Gamma$-symmetric subgraph of the $\Gamma$-symmetric graph $(G,\phi)$.

Consider the inversion symmetry group $C_i$. It follows from Theorem \ref{fixed e&v on cy} that the graphs we need to understand are $C_i$-symmetric graphs which are $(2,2)$-tight and have no edges or vertices fixed by the inversion $\varphi$.
Henceforth we shall refer to such graphs as \emph{$(2,2)$-$C_i$-tight} graphs.
Similarly, graphs which are $(2,2)$-sparse and $C_i$-symmetric shall be referred to as \emph{$(2,2)$-$C_i$-sparse}.
Figure \ref{fig CI2} shows the two base graphs for the class of $(2,2)$-$C_i$-tight graphs; we will call the graph on six vertices $(F_1,\phi_1)$, and the graph on eight vertices $(F_2,\phi_2)$, where for $\gamma \in \mathbb Z_2 \setminus \{\textrm{id}\}$, $\phi_1(\gamma)$ and $\phi_2(\gamma)$ do not fix any vertices or edges of $F_1$ and $F_2$ respectively.

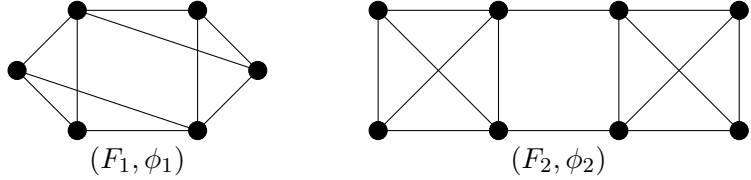
\begin{figure}[ht]
    \centering
    \begin{tikzpicture}
  [scale=.8,auto=left]
  
  \coordinate (n1) at (0,1);
  \coordinate (n2) at (1,0);
  \coordinate (n3) at (1,2);
  \coordinate (n4) at (3,0);
  \coordinate (n5) at (3,2);
  \coordinate (n6) at (4,1);
  
 \draw[fill=black] (n1) circle (0.15cm);
 \draw[fill=black] (n2) circle (0.15cm);
 \draw[fill=black] (n3) circle (0.15cm);
 \draw[fill=black] (n4) circle (0.15cm);
 \draw[fill=black] (n5) circle (0.15cm);
 \draw[fill=black] (n6) circle (0.15cm);

  \foreach \from/\to in {n1/n2,n1/n3,n1/n4,n2/n3,n2/n4,n3/n5,n3/n6,n4/n5,n4/n6,n5/n6} 
    \draw (\from) -- (\to);
    
  \coordinate (n11) at (6,0);
  \coordinate (n12) at (6,2);
  \coordinate (n13) at (8,0);
  \coordinate (n14) at (8,2);
  \coordinate (n15) at (10,0);
  \coordinate (n16) at (10,2);
  \coordinate (n17) at (12,0);
  \coordinate (n18) at (12,2);  
  
 \draw[fill=black] (n11) circle (0.15cm);
 \draw[fill=black] (n12) circle (0.15cm);
 \draw[fill=black] (n13) circle (0.15cm);
 \draw[fill=black] (n14) circle (0.15cm);
 \draw[fill=black] (n15) circle (0.15cm);
 \draw[fill=black] (n16) circle (0.15cm);
 \draw[fill=black] (n17) circle (0.15cm);
 \draw[fill=black] (n18) circle (0.15cm);

  \foreach \from/\to in {n11/n12,n11/n13,n11/n14,n12/n13,n12/n14,n13/n15,n13/n14,n14/n16,n15/n16,n15/n17,n15/n18,n16/n17,n16/n18,n17/n18} 
    \draw (\from) -- (\to);

\draw (2,-0.5) node {$(F_1,\phi_1)$};
\draw (9,-0.5) node {$(F_2,\phi_2)$};
    
\end{tikzpicture}
    \caption{The $C_i$-symmetric base graphs.}
    \label{fig CI2}
\end{figure}

Instead consider the half-turn symmetry group $C_2$.
By Theorem \ref{fixed e&v on cy}, a $C_2$-isostatic graph is $(2,2)$-tight and has two fixed edges and no fixed vertex, or no fixed edge and one fixed vertex. Hence a graph is called \emph{$(2,2)$-$C_2$-tight} if it is $(2,2)$-tight, $C_2$-symmetric and contains either two fixed edges and no fixed vertex, or no fixed edge and one fixed vertex.
Similarly, graphs which are $(2,2)$-sparse and $C_2$-symmetric shall be referred to as \emph{$(2,2)$-$C_2$-sparse}.
In Figure \ref{BaseGraphs,C_2,cylinder}, we show four small $C_2$-symmetric graphs that are $(2,2)$-tight.
These are, reading left to right, top to bottom: $(K_4, \phi_3)$ with two fixed edges and no fixed vertex, $(W_5,\phi_4)$ with one fixed vertex and no fixed edge, $(Wd(4,2),\phi_5)$ with one fixed vertex and no fixed edge, and $(F_2,\phi_2)$.
These will turn out to be the base graphs of our recursive construction.
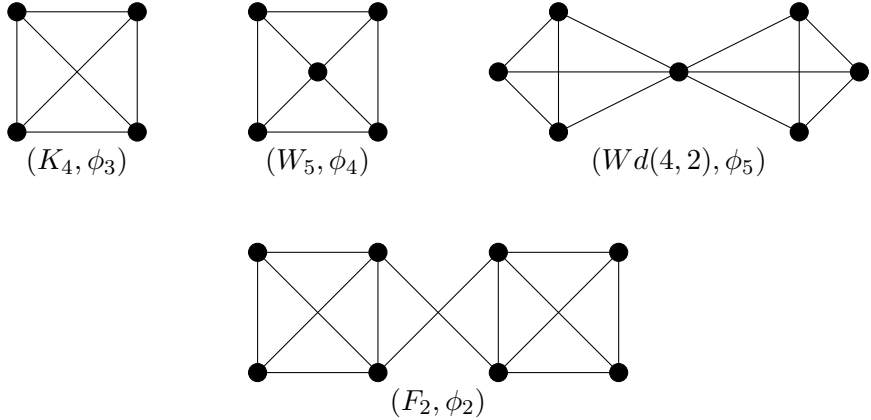
\begin{figure}[ht]
\centering
    \begin{tikzpicture}
  [scale=.8,auto=left]
  
  \coordinate (n21) at (0,4);
  \coordinate (n22) at (0,6);
  \coordinate (n23) at (2,4);
  \coordinate (n24) at (2,6);
 \draw[fill=black] (n21) circle (0.15cm);
 \draw[fill=black] (n22) circle (0.15cm);
 \draw[fill=black] (n23) circle (0.15cm);
 \draw[fill=black] (n24) circle (0.15cm);

  \foreach \from/\to in {n21/n22,n21/n23,n22/n24,n23/n24,n21/n24,n22/n23} 
    \draw (\from) -- (\to);
  
  \coordinate (n1) at (4,4);
  \coordinate (n2) at (4,6);
  \coordinate (n3) at (6,4);
  \coordinate (n4) at (6,6);
  \coordinate (n5) at (5,5);
 \draw[fill=black] (n1) circle (0.15cm);
 \draw[fill=black] (n2) circle (0.15cm);
 \draw[fill=black] (n3) circle (0.15cm);
 \draw[fill=black] (n4) circle (0.15cm);
 \draw[fill=black] (n5) circle (0.15cm);

  \foreach \from/\to in {n1/n2,n1/n3,n2/n4,n3/n4,n1/n5,n2/n5,n3/n5,n4/n5} 
    \draw (\from) -- (\to);
    
  \coordinate (n11) at (8,5);
  \coordinate (n12) at (9,4);
  \coordinate (n13) at (9,6);
  \coordinate (n14) at (13,4);
  \coordinate (n15) at (13,6);
  \coordinate (n16) at (14,5);
  \coordinate (n17) at (11,5);
 \draw[fill=black] (n11) circle (0.15cm);
 \draw[fill=black] (n12) circle (0.15cm);
 \draw[fill=black] (n13) circle (0.15cm);
 \draw[fill=black] (n14) circle (0.15cm);
 \draw[fill=black] (n15) circle (0.15cm);
 \draw[fill=black] (n16) circle (0.15cm);
 \draw[fill=black] (n17) circle (0.15cm);

  \foreach \from/\to in {n11/n12,n11/n13,n12/n13,n14/n15,n14/n16,n15/n16,n11/n17,n12/n17,n13/n17,n14/n17,n15/n17,n16/n17} 
    \draw (\from) -- (\to);

  \coordinate (n41) at (4,0);
  \coordinate (n42) at (4,2);
  \coordinate (n43) at (6,0);
  \coordinate (n44) at (6,2);
  \coordinate (n45) at (8,0);
  \coordinate (n46) at (8,2);
  \coordinate (n47) at (10,0);
  \coordinate (n48) at (10,2);  
  
 \draw[fill=black] (n41) circle (0.15cm);
 \draw[fill=black] (n42) circle (0.15cm);
 \draw[fill=black] (n43) circle (0.15cm);
 \draw[fill=black] (n44) circle (0.15cm);
 \draw[fill=black] (n45) circle (0.15cm);
 \draw[fill=black] (n46) circle (0.15cm);
 \draw[fill=black] (n47) circle (0.15cm);
 \draw[fill=black] (n48) circle (0.15cm);

  \foreach \from/\to in {n41/n42,n41/n43,n41/n44,n42/n43,n42/n44,n43/n46,n43/n44,n44/n45,n45/n46,n45/n47,n45/n48,n46/n47,n46/n48,n47/n48} 
    \draw (\from) -- (\to);

\draw (1,3.5) node {$(K_4,\phi_3)$};
\draw (5,3.5) node {$(W_5,\phi_4)$};
\draw (11,3.5) node {$(Wd(4,2),\phi_5)$};
\draw (7,-0.5) node {$(F_2,\phi_2)$};
        
\end{tikzpicture}
    \caption{The $C_2$-symmetric base graphs.}
    \label{BaseGraphs,C_2,cylinder}
\end{figure}

Finally consider the reflection symmetry group $C_s$.
By Theorem \ref{fixed e&v on cy}, a $C_s$-isostatic graph is $(2,2)$-tight and has no fixed edge and any number of fixed vertices. Hence a graph is called \emph{$(2,2)$-$C_s$-tight} if it is $(2,2)$-tight, $C_s$-symmetric and contains no fixed edge.
Similarly, graphs which are $(2,2)$-sparse, $C_s$-symmetric and have no fixed edge shall be referred to as \emph{$(2,2)$-$C_s$-sparse}.
In Figure \ref{BaseGraphs,C_s,cylinder}, we show six small $C_s$-symmetric graphs that are $(2,2)$-tight.
These are, reading left to right, top to bottom: $(F_2, \phi_2)$, $(W_5, \phi_4)$, $(Wd(4,2),\phi_5)$, $(F_1,\phi_1)$, $(F_1,\phi_6)$ with two fixed vertices and no fixed edge, and $(K_{3,4}, \phi_7)$ with three fixed vertices and no fixed edge.
These will be the base graphs of our recursive construction.

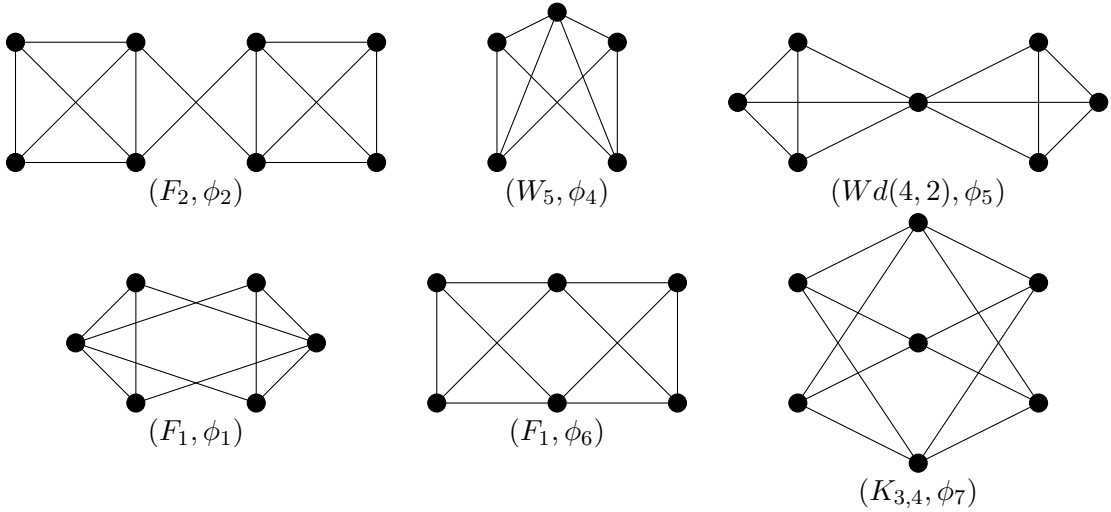
\begin{figure}[ht]
    \centering
    \begin{tikzpicture}
  [scale=.8,auto=left]

  \coordinate (n1) at (1,-3);
  \coordinate (n2) at (2,-4);
  \coordinate (n3) at (2,-2);
  \coordinate (n4) at (4,-4);
  \coordinate (n5) at (4,-2);
  \coordinate (n6) at (5,-3);
 \draw[fill=black] (n1) circle (0.15cm);
 \draw[fill=black] (n2) circle (0.15cm);
 \draw[fill=black] (n3) circle (0.15cm);
 \draw[fill=black] (n4) circle (0.15cm);
 \draw[fill=black] (n5) circle (0.15cm);
 \draw[fill=black] (n6) circle (0.15cm);

\draw (3,-4.5) node {$(F_1,\phi_1)$};

  \foreach \from/\to in {n1/n2,n1/n3,n1/n4,n1/n5,n2/n3,n2/n6,n3/n6,n4/n5,n4/n6,n5/n6} 
    \draw (\from) -- (\to);

  \coordinate (n11) at (0,0);
  \coordinate (n12) at (0,2);
  \coordinate (n13) at (2,0);
  \coordinate (n14) at (2,2);
  \coordinate (n15) at (4,0);
  \coordinate (n16) at (4,2);
  \coordinate (n17) at (6,0);
  \coordinate (n18) at (6,2);  
  
 \draw[fill=black] (n11) circle (0.15cm);
 \draw[fill=black] (n12) circle (0.15cm);
 \draw[fill=black] (n13) circle (0.15cm);
 \draw[fill=black] (n14) circle (0.15cm);
 \draw[fill=black] (n15) circle (0.15cm);
 \draw[fill=black] (n16) circle (0.15cm);
 \draw[fill=black] (n17) circle (0.15cm);
 \draw[fill=black] (n18) circle (0.15cm);

\draw (3,-0.5) node {$(F_2,\phi_2)$};

  \foreach \from/\to in {n11/n12,n11/n13,n12/n14,n13/n16,n15/n17,n16/n18,n17/n18,n11/n14,n12/n13,n13/n14,n14/n15,n15/n16,n15/n18,n16/n17} 
    \draw (\from) -- (\to);
    
  \coordinate (n21) at (12,1);
  \coordinate (n22) at (13,0);
  \coordinate (n23) at (13,2);
  \coordinate (n24) at (17,0);
  \coordinate (n25) at (17,2);
  \coordinate (n26) at (18,1);
  \coordinate (n27) at (15,1);
 \draw[fill=black] (n21) circle (0.15cm);
 \draw[fill=black] (n22) circle (0.15cm);
 \draw[fill=black] (n23) circle (0.15cm);
 \draw[fill=black] (n24) circle (0.15cm);
 \draw[fill=black] (n25) circle (0.15cm);
 \draw[fill=black] (n26) circle (0.15cm);
 \draw[fill=black] (n27) circle (0.15cm);

\draw (15,-0.5) node {$(Wd(4,2),\phi_5)$};

  \foreach \from/\to in {n21/n23,n22/n23,n24/n25,n24/n26,n22/n27,n25/n27,n21/n22,n25/n26,n21/n27,n23/n27,n24/n27,n26/n27} 
    \draw (\from) -- (\to);
    
  \coordinate (n31) at (13,-4);
  \coordinate (n32) at (13,-2);
  \coordinate (n33) at (15,-5);
  \coordinate (n34) at (15,-3);
  \coordinate (n35) at (15,-1);
  \coordinate (n36) at (17,-4);
  \coordinate (n37) at (17,-2);
 \draw[fill=black] (n31) circle (0.15cm);
 \draw[fill=black] (n32) circle (0.15cm);
 \draw[fill=black] (n33) circle (0.15cm);
 \draw[fill=black] (n34) circle (0.15cm);
 \draw[fill=black] (n35) circle (0.15cm);
 \draw[fill=black] (n36) circle (0.15cm);
 \draw[fill=black] (n37) circle (0.15cm);

\draw (15,-5.5) node {$(K_{3,4},\phi_7)$};

  \foreach \from/\to in {n31/n33,n31/n34,n31/n35,n32/n33,n32/n34,n32/n35,n34/n36,n34/n37,n35/n36,n35/n37,n33/n36,n33/n37} 
    \draw (\from) -- (\to);
    
  \coordinate (n41) at (7,-4);
  \coordinate (n42) at (7,-2);
  \coordinate (n43) at (9,-4);
  \coordinate (n44) at (9,-2);
  \coordinate (n45) at (11,-4);
  \coordinate (n46) at (11,-2);
 \draw[fill=black] (n41) circle (0.15cm);
 \draw[fill=black] (n42) circle (0.15cm);
 \draw[fill=black] (n43) circle (0.15cm);
 \draw[fill=black] (n44) circle (0.15cm);
 \draw[fill=black] (n45) circle (0.15cm);
 \draw[fill=black] (n46) circle (0.15cm);

\draw (9,-4.5) node {$(F_1,\phi_6)$};

  \foreach \from/\to in {n41/n42,n41/n43,n41/n44,n42/n43,n42/n44,n43/n45,n44/n46,n44/n45,n43/n46,n45/n46} 
    \draw (\from) -- (\to);
    
  \coordinate (n51) at (8,0);
  \coordinate (n52) at (8,2);
  \coordinate (n53) at (9,2.5);
  \coordinate (n54) at (10,0);
  \coordinate (n55) at (10,2);
 \draw[fill=black] (n51) circle (0.15cm);
 \draw[fill=black] (n52) circle (0.15cm);
 \draw[fill=black] (n53) circle (0.15cm);
 \draw[fill=black] (n54) circle (0.15cm);
 \draw[fill=black] (n55) circle (0.15cm);

\draw (9,-0.5) node {$(W_5,\phi_4)$};

  \foreach \from/\to in {n51/n52,n51/n53,n51/n55,n52/n53,n52/n54,n53/n55,n53/n54,n54/n55} 
    \draw (\from) -- (\to);

\end{tikzpicture}
    \caption{The $C_s$-symmetric base graphs, with the mirror vertically aligned on the page.}
    \label{BaseGraphs,C_s,cylinder}
\end{figure}

\subsection{Reduction operations}

We will consider \emph{reduction operations}: these are the reverse of the extension operations described in Section \ref{sec:ops}.
While the operations we require vary slightly for each symmetry group, the following are required across the three symmetries we will provide characterisations for, namely symmetrised 0-reduction, symmetrised 1-reduction, symmetrised $C_4$ contraction, symmetrised $K_4$ contraction.

\begin{lem}\label{deg(v)=2}
Let $(G,\phi)$ be $(2,2)$-$C$-tight for $C\in \{C_i,C_2,C_s\}$ and suppose $v \in V$ is a vertex of degree 2.
Then either $C=C_s$, $v =\sigma(v)= v'$ and $H = G - \{v\}$ is is $(2,2)$-$C$-tight or $v\neq v'$ and
$H = G - \{v,v'\}$ is $(2,2)$-$C$-tight.
\end{lem}

\begin{proof}
The case when $C=C_s$ and $v = v'$ is trivial.
Moreover if $C = C_2$ then any degree two vertex $v$ in a $(2,2)$-$C$-tight graph $G$ satisfies $v'=c'_2(v) \neq v$, for otherwise the subgraph $G-v$ would be $(2,2)$-tight but have no fixed edges or vertices, contradicting the fact that $G$ is $(2,2)$-$C$-tight. 
For any $C$, $vv'\notin E$ for otherwise $H = G - \{v,v'\}$ would have $|V(H)| = |V|-2$ but $|E(H)| = |E|-3$, violating the $(2,2)$-sparsity of $G$.
Then, any subgraph of $H$ is a subgraph of $G$, so as $G$ is $(2,2)$-tight, $H$ is.
Also $H$ will be $C$-symmetric, and we do not remove any fixed edges or vertices.
\end{proof}

Most of the technical work in the next four sections involves analysing when we can remove a vertex of degree 3.
Hence, for brevity, we will say that a vertex of degree 3 is called a \emph{node}.

\begin{lem}\label{lem: no blocking}
Let $(G,\phi)$ be $(2,2)$-$C$-tight for $C\in \{C_i,C_2,C_s\}$ and suppose $v \in V$ is a node so that $x,y\in N(v)$ with $xy\notin E$ and $\{x,y\} \neq \{x',y'\}$.
Then $G' = G-\{v,v'\}+\{xy,x'y'\}$ is not $(2,2)$-$C$-tight if and only if at least one of the following hold:
\begin{enumerate}
    \item there exists a 2-critical set $U$ with $x,y\in U$;
    \item there exists a 3-critical set $W$ with $x,y,x',y' \in W$;
    \item $C = C_2$ and there exists a 4-critical set $T$ with $x,y,x',y' \in T$ and $G[T]$ is $C_2$-symmetric with no fixed vertex or edges.
\end{enumerate}
\end{lem}

\begin{proof}
Suppose that $x,y$ (resp. $x',y'$) are contained in a 2-critical set $U$, or $x,y,x',y'$ are contained in a 3-critical set $W$.
Then $U$ and $W$ would, with the new edges, create subgraphs $G'[U] = (U,E_1)$ and $G'[W] = (W,E_2)$ where $|E_1| = 2|U| -1$ and $|E_2| = 2|W| -1$ respectively.
This proves the first two conditions imply $G'$ is not $(2,2)$-$C$-tight.
Additionally for $(2,2)$-$C_2$-tight graphs, all $C_2$-symmetric tight subgraphs must have the fixed vertex or edge constraint.
Any reduction cannot create a tight subgraph which does not satisfy this fixed count.
Therefore a 4-critical $C_2$-symmetric vertex set $T$ where $G[T]$ does not contain fixed edges or vertices, has $G'[T]$ a $C_2$-symmetric $(2,2)$-tight subgraph of $G'$, which is not $(2,2)$-$C_2$-tight.
Hence the third condition implies $G'$ is not $(2,2)$-$C$-tight.

Conversely if conditions (1)-(3) hold then the facts that $G$ is $(2,2)$-$C$-tight, $G'$ is obtained from a subgraph of $G$ by adding 2 distinct edges, and $C_i$ and $C_s$ do not have fixed vertex or edge constraints that need to be preserved in the reduction imply that $G'$ is $(2,2)$-$C$-tight. 
\end{proof}

\begin{lem}\label{lem no W}
Let $(G,\phi)$ be $(2,2)$-$C$-tight for $C \in \{C_i, C_2, C_s\}$ with no fixed edge and
suppose $v \in V$ is a node with $N(v) =\{x,y,z\}$ and $xy \notin E$.
If the pair $x,y$ is not contained in any $2$-critical subset of $V\setminus \{v,v'\}$, then there does not exist $W \subseteq V\setminus \{v,v'\}$ with $x,x',y,y' \in W$ and $i_G(W) = 2|W| - 3$.
\end{lem}

\begin{proof}
Suppose for a contradiction that there exists some $W\subseteq V\setminus \{v,v'\}$, $x,x',y,y' \in W$ with $i(W) = 2|W|-3$.
Observe $i(W') = i(W)$, $i(W\cup W') \leq 2|W\cup W'|-3$ and $i(W\cap W') \leq 2|W\cap W'|-3$ (since $x,x',y,y'\in W \cap W'$).
Now we have
\begin{equation}\label{counting W}
    \begin{split}
        2|W|-3 + 2|W'|-3 & =  i(W) + i(W') = i(W\cup W') + i(W \cap W') - d(W,W') \\
        & \leq 2|W \cup W'| - 3 + 2|W \cap W'| - 3 - d(W,W') \\
        & = 2|W| + 2|W'| - 6 - d(W,W').
    \end{split}
\end{equation}
It follows that we have equality throughout and $d(W,W') = 0$.
However $W\cup W'$ is $C$-symmetric with no fixed edges, so $i(W\cup W')$ is even, a contradiction.
\end{proof}

\begin{rem}\label{critical cup and cap}
Similar counting arguments to Equation (\ref{counting W}) can be used to give the following (and other similar observations) on the union and intersection of $k$-critical sets that we use repeatedly.
Let $(G,\phi)$ be $(2,2)$-tight.
Take $X,Y \subseteq V$.
If $X,Y \subseteq V$ are $2$-critical and $X \cap Y \neq \emptyset$ then $X \cup Y$ and $X \cap Y$ are $2$-critical and $d(X,Y) = 0$.

Further if $X$ is 2-critical, $Y$ is 3-critical and $X\cap Y \neq \emptyset$, then either:
\begin{itemize}
    \item $d(X,Y) = 0$, $i(X\cap Y) = 2|X\cap Y|-3$ and $i(X \cup Y) = 2|X \cup Y| - 2$; or
    \item $d(X,Y) = 0$, $i(X\cap Y) = 2|X\cap Y|-2$ and $i(X \cup Y) = 2|X \cup Y| - 3$; or
    \item $d(X,Y) = 1$ and $X\cap Y$ and $X\cup Y$ are $2$-critical.
\end{itemize}
\end{rem}

\begin{lem}\label{lem: deg(v)=3, cap empty}
Let $(G,\phi)$ be $(2,2)$-$C$-tight for $C\in \{C_i,C_s\}$ and suppose $v \in V$ is a node with $N(v) \cap N(v') = \emptyset$.
Then either $G[N[v]] = K_{4}$, or there exists $x,y \in N(v)$ such that $xy \notin E$, and $G^- = G-\{v,v'\}+\{xy,x'y'\}$ is $(2,2)$-$C$-tight.
\end{lem}

\begin{proof}
Assume that $G[N[v]] \neq K_{4}$.
By Lemma \ref{lem no W}, we only need to show that for one pair of non-adjacent vertices in $N(v)$, there is no $2$-critical set containing them.
We consider cases based on $i(N(v))$.
Let $N(v)=\{x,y,z\}$.
Firstly, where there are no edges on the neighbours of $v$, if all of the pairs $\{x,y\},\{x,z\},\{y,z\}$ are contained in $2$-critical sets $U_1, U_2, U_3 \subseteq V-\{v,v'\}$ say, then by Remark \ref{critical cup and cap}, $U_1 \cup U_2$ is $2$-critical and so $U_1 \cup U_2 \cup \{v\}$ breaks $(2,2)$-sparsity of $G$.
Similarly when $i(N(v))=1$.
Now suppose $i(N(v))=2$, and say $xy \notin E$.
If there existed a $2$-critical $U \subseteq V- \{v,v'\}$ with $x,y \in U$, then $i_G(U \cup \{v,z\}) = 2|U \cup \{v,z\}|-1$  which contradicts $(2,2)$-sparsity of $G$.
Hence $G^- = G-\{v,v'\}+\{xy,x'y'\}$ is $(2,2)$-$C$-tight.
\end{proof}

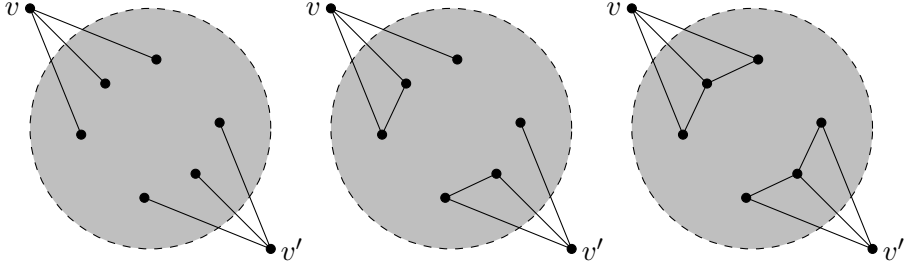
\begin{figure}[ht]
    \centering
    \begin{tikzpicture}
  [scale=.4,auto=left]
  
  \coordinate (n1) at (0,8);
  \coordinate (n2) at (1.7,3.8);
  \coordinate (n3) at (2.5,5.5);
  \coordinate (n4) at (4.2,6.3);
  \coordinate (n5) at (8,0);
  \coordinate (n6) at (3.8,1.7);
  \coordinate (n7) at (5.5,2.5);
  \coordinate (n8) at (6.3,4.2);
  \coordinate (n9) at (4,4);
  
 \draw[dashed, fill=lightgray] (n9) circle (4cm); 
 \draw[fill=black] (n1) circle (0.15cm)
	    node[left] {$v$};
 \draw[fill=black] (n2) circle (0.15cm);
 \draw[fill=black] (n3) circle (0.15cm);
 \draw[fill=black] (n4) circle (0.15cm);
 \draw[fill=black] (n5) circle (0.15cm)
	    node[right] {$v'$};
 \draw[fill=black] (n6) circle (0.15cm);
 \draw[fill=black] (n7) circle (0.15cm);
 \draw[fill=black] (n8) circle (0.15cm);

  \foreach \from/\to in {n1/n2,n1/n3,n1/n4,n5/n6,n5/n7,n5/n8} 
    \draw (\from) -- (\to);

  \coordinate (n11) at (10,8);
  \coordinate (n12) at (11.7,3.8);
  \coordinate (n13) at (12.5,5.5);
  \coordinate (n14) at (14.2,6.3);
  \coordinate (n15) at (18,0);
  \coordinate (n16) at (13.8,1.7);
  \coordinate (n17) at (15.5,2.5);
  \coordinate (n18) at (16.3,4.2);
  \coordinate (n19) at (14,4);
  
 \draw[dashed, fill=lightgray] (n19) circle (4cm); 
 \draw[fill=black] (n11) circle (0.15cm)
	    node[left] {$v$};
 \draw[fill=black] (n12) circle (0.15cm);
 \draw[fill=black] (n13) circle (0.15cm);
 \draw[fill=black] (n14) circle (0.15cm);
 \draw[fill=black] (n15) circle (0.15cm)
	    node[right] {$v'$};
 \draw[fill=black] (n16) circle (0.15cm);
 \draw[fill=black] (n17) circle (0.15cm);
 \draw[fill=black] (n18) circle (0.15cm);

  \foreach \from/\to in {n11/n12,n11/n13,n11/n14,n15/n16,n15/n17,n15/n18,n12/n13,n16/n17} 
    \draw (\from) -- (\to);
	    
  \coordinate (n21) at (20,8);
  \coordinate (n22) at (21.7,3.8);
  \coordinate (n23) at (22.5,5.5);
  \coordinate (n24) at (24.2,6.3);
  \coordinate (n25) at (28,0);
  \coordinate (n26) at (23.8,1.7);
  \coordinate (n27) at (25.5,2.5);
  \coordinate (n28) at (26.3,4.2);
  \coordinate (n29) at (24,4);
  
 \draw[dashed, fill=lightgray] (n29) circle (4cm); 
 \draw[fill=black] (n21) circle (0.15cm)
	    node[left] {$v$};
 \draw[fill=black] (n22) circle (0.15cm);
 \draw[fill=black] (n23) circle (0.15cm);
 \draw[fill=black] (n24) circle (0.15cm);
 \draw[fill=black] (n25) circle (0.15cm)
	    node[right] {$v'$};
 \draw[fill=black] (n26) circle (0.15cm);
 \draw[fill=black] (n27) circle (0.15cm);
 \draw[fill=black] (n28) circle (0.15cm);

  \foreach \from/\to in {n21/n22,n21/n23,n21/n24,n25/n26,n25/n27,n25/n28,n22/n23,n23/n24,n26/n27,n27/n28} 
    \draw (\from) -- (\to);

\end{tikzpicture}
    \caption{The local structure of the cases in Lemma \ref{lem: deg(v)=3, cap empty}.}
    \label{fig CI1}
\end{figure}

\begin{lem}\label{deg(v)=3, x=y'}
Let $(G,\phi)$ be $(2,2)$-$C$-tight for $C \in \{C_i, C_s\}$ and suppose $v \in V$ is a node such that $N(v)=\{x,y,z\}$ and $N(v) \cap N(v') = \{x,y\}$, with $x' = y$ or $C = C_s$ and $x$ and $y$ are fixed vertices. Then one of the following hold:
\begin{enumerate}
    \item $G[\{v,v',x,y,z,z'\}] \cong (F_{1}, \phi_1)$
    \item $C=C_s$ and $G[\{v,v',x,y,z,z'\}] \cong (F_1, \phi_6)$. 
    \item there exists some $v_1 \in \{x,y\}$ such that $G^- = G-\{v,v'\}+\{v_{1}z,v_{1}z'\}$ is $(2,2)$-$C$-tight.
\end{enumerate}
\end{lem}

\begin{proof}
Suppose $\{xz,yz,xz',yz'\} \subset E$.
If $x' = y$, then $G[\{v,v',x,y,z,z'\}] \cong (F_{1},\phi_1)$ as in (1), otherwise $x$ and $y$ are fixed and $G[\{v,v',x,y,z,z'\}] \cong (F_{1},\phi_6)$ as in (2).

When one of the edge pairs $\{xz,yz'\},\{xz',yz\}$ is present, without loss of generality say $\{xz,yz'\} \in E$.
Suppose there exists a $U \subseteq V-v$, $y,z \in U$ which is $2$-critical.
If $U \cap U' \neq \emptyset$, then $U \cup U'$ is $2$-critical by Remark \ref{critical cup and cap}, and $U \cup U' \cup \{v\}$ violates $(2,2)$-sparsity of $G$.
If $U \cap U' = \emptyset$, then $xz', yz' \in d(U,U')$ so $U\cup U'$ is $2$-critical and $U \cup U' \cup \{v\}$ again breaks $(2,2)$-sparsity.
By Lemma \ref{lem no W}, since there is no $2$-critical set on $x,y,z,z'$, we have that $i_{G}(W) \leq 2|W|-4$ for all $W \subseteq V\setminus \{v,v'\}$ such that $x,y,z,z' \in W$, so $G^- = G-\{v,v'\}+\{xz',yz\}$ is $(2,2)$-$C$-tight. 

Now assume we have no edges on $N(v)$.
We want to show that we can add either $xz,yz'$ or $yz,xz'$ to $G-\{v,v'\}$.
Suppose we can add neither $xz$ or $x'z$, that is, there are $2$-critical sets $U_1,U_2 \subseteq V-v$ with $x,z \in U_1$ and $x',z \in U_2$.
Then $U_1 \cap U_2 \neq \emptyset$, so $U_1 \cup U_2$ is $2$-critical by Remark \ref{critical cup and cap}. Thus the subgraph induced by $U_1 \cup U_2 \cup \{v\}$ contradicts $G$ being $(2,2)$-tight. 
We recall Lemma \ref{lem no W} gives for any $W$ containing $\{x,y,z,z'\}$, $i_G(W) \leq 2|W|-4$, giving us the required result.
\end{proof}

\begin{figure}[ht]
    \centering
    \begin{tikzpicture}
  [scale=.4,auto=left]

  \coordinate (n11) at (0,8);
  \coordinate (n12) at (2.5,2.5);
  \coordinate (n13) at (2.5,5.5);
  \coordinate (n14) at (5.5,5.5);
  \coordinate (n15) at (8,0);
  \coordinate (n16) at (2.5,2.5);
  \coordinate (n17) at (5.5,2.5);
  \coordinate (n18) at (5.5,5.5);
  \coordinate (n19) at (4,4);
  
 \draw[dashed, fill=lightgray] (n19) circle (4cm); 
 \draw[fill=black] (n11) circle (0.15cm)
	    node[left] {$v$};
 \draw[fill=black] (n12) circle (0.15cm)
	    node[below left] {$x$};
 \draw[fill=black] (n13) circle (0.15cm)
	    node[below right] {$y$};
 \draw[fill=black] (n14) circle (0.15cm)
	    node[above right] {$x'$};
 \draw[fill=black] (n15) circle (0.15cm)
	    node[right] {$v'$};
 \draw[fill=black] (n16) circle (0.15cm);
 \draw[fill=black] (n17) circle (0.15cm)
	    node[above left] {$y'$};
 \draw[fill=black] (n18) circle (0.15cm);
	    
  \foreach \from/\to in { n11/n12,n11/n13,n11/n14,n15/n16,n15/n17,n15/n18,n12/n13,n13/n14,n16/n17,n17/n18} 
    \draw (\from) -- (\to);

  \coordinate (n21) at (10,8);
  \coordinate (n22) at (12.5,2.5);
  \coordinate (n23) at (12.5,5.5);
  \coordinate (n24) at (15.5,5.5);
  \coordinate (n25) at (18,0);
  \coordinate (n26) at (12.5,2.5);
  \coordinate (n27) at (15.5,2.5);
  \coordinate (n28) at (15.5,5.5);
  \coordinate (n29) at (14,4);
  
 \draw[dashed, fill=lightgray] (n29) circle (4cm); 
 \draw[fill=black] (n21) circle (0.15cm)
	    node[left] {$v$};
 \draw[fill=black] (n22) circle (0.15cm)
	    node[below left] {$x$};
 \draw[fill=black] (n23) circle (0.15cm)
	    node[below right] {$y$};
 \draw[fill=black] (n24) circle (0.15cm)
	    node[above right] {$x'$};
 \draw[fill=black] (n25) circle (0.15cm)
	    node[right] {$v'$};
 \draw[fill=black] (n26) circle (0.15cm);
 \draw[fill=black] (n27) circle (0.15cm)
	    node[above left] {$y'$};
 \draw[fill=black] (n28) circle (0.15cm);

  \foreach \from/\to in { n21/n22,n21/n23,n21/n24,n25/n26,n25/n27,n25/n28,n23/n24,n26/n27} 
    \draw (\from) -- (\to);

  \coordinate (n31) at (20,8);
  \coordinate (n32) at (22.5,2.5);
  \coordinate (n33) at (22.5,5.5);
  \coordinate (n34) at (25.5,5.5);
  \coordinate (n35) at (28,0);
  \coordinate (n36) at (22.5,2.5);
  \coordinate (n37) at (25.5,2.5);
  \coordinate (n38) at (25.5,5.5);
  \coordinate (n39) at (24,4);
  
 \draw[dashed, fill=lightgray] (n39) circle (4cm); 
 \draw[fill=black] (n31) circle (0.15cm)
	    node[left] {$v$};
 \draw[fill=black] (n32) circle (0.15cm)
	    node[below left] {$x$};
 \draw[fill=black] (n33) circle (0.15cm)
	    node[below right] {$y$};
 \draw[fill=black] (n34) circle (0.15cm)
	    node[above right] {$x'$};
 \draw[fill=black] (n35) circle (0.15cm)
	    node[right] {$v'$};
 \draw[fill=black] (n36) circle (0.15cm);
 \draw[fill=black] (n37) circle (0.15cm)
	    node[above left] {$y'$};
 \draw[fill=black] (n38) circle (0.15cm);

  \foreach \from/\to in { n31/n32,n31/n33,n31/n34,n35/n36,n35/n37,n35/n38} 
    \draw (\from) -- (\to);

\end{tikzpicture}
    \caption{The local structure of the cases in Lemma \ref{deg(v)=3, x=y'}.}
    \label{fig CI1.5}
\end{figure}
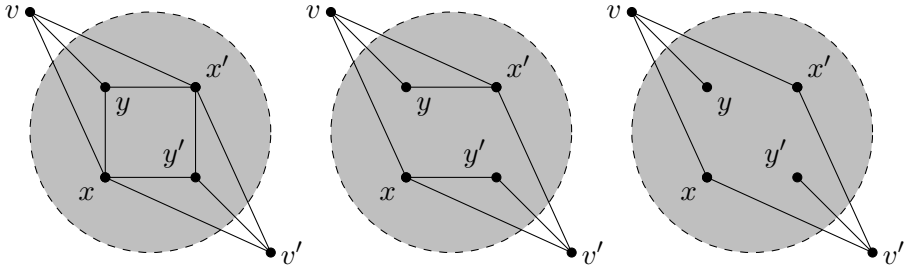

\begin{lem}\label{deg 3 C_2 t}
Let $(G,\phi)$ be $(2,2)$-$C$-tight for $C \in \{C_2,C_s\}$ and suppose $v \in V$ is a node so that $N[v] \cap N[v'] = \{t\}$, where $t$ is a fixed vertex in $G$.
Let $N(v) = \{x,y,t\}$.
Then either $G[N[v]\cup N[v']] = (Wd(4,2),\phi_5)$ or one of $G_1 = G-\{v,v'\}+\{xt,x't\}$, $G_2 = G-\{v,v'\}+\{yt,y't\}$, or $G_3 = G-\{v,v'\}+\{xy,x'y'\}$ is $(2,2)$-$C$-tight.
\end{lem}

\begin{proof}
Since $G$ has no fixed edges, Lemma \ref{lem no W} implies that if $x,y,x',y'$ are in a 3-critical set then they are in a 2-critical set too. Hence, for the remainder of the proof, we only consider 2-critical or 4-critical sets in the case when $C = C_2$.

We break up the proof into cases by considering the number of edges induced by the neighbours of $v$.
Firstly, when all $3$ edges $xy, xt, yt$ are present in the graph, we have a copy of $Wd(4,2)$.
Now, when two edges are present, without loss of generality, we may assume either $xy \notin E$ or  $yt \notin E$.
If $xy \notin E$ (resp. $yt \notin E$), suppose there exists a $2$-critical $U \subset V$ with $x,y \in U$ (resp. $t, y \in U$). Then the subgraph induced by $U \cup \{v,t\}$ (resp. $U \cup \{v,x\}$) violates the $(2,2)$-sparsity of $G$.
There is no $4$-critical $C_2$-symmetric set $T$ containing $x,y$ and not $v,t$ since the subgraph induced by $T \cup \{v,v',t\}$ violates $(2,2)$-sparsity.

Consider now the case where one or zero edges are induced by $\{x,y,t\}$.
No two of the pairs $\{x,y\},\{x,t\},\{y,t\}$ can each be contained in a $2$-critical set, as if any two were contained in $2$-critical sets $U_1, U_2$, then by Remark \ref{critical cup and cap}, $U_1 \cup U_2$ is $2$-critical but the subgraph induced by $U_1 \cup U_2 +v$ violates the $(2,2)$-sparsity of $G$.
For $C= C_2$, to complete the proof we need to confirm that one of these pairs and its symmetric copy is not in a 4-critical set which contains no fixed vertex.
However, for any two sets from $\{x,y,x',y'\},\{x,x',t\},\{y,y',t\}$, at least one contains the fixed vertex of $G$.
Hence we may reduce symmetrically unless $G[N[v]\cup N[v']] \cong (Wd(4,2),\phi_5)$.
\end{proof}

\begin{lem}\label{deg 3 C_2 t x x'}
Let $(G,\phi)$ be $(2,2)$-$C$-tight for $C \in \{C_2,C_s\}$ and suppose $v \in V$ is a node chosen so that $N[v] \cap N[v'] = \{t,x,x'\}$, where $t$ is fixed.
Then either $G[N[v]\cup N[v']] = (W_5,\phi_4)$, or $G' = G-\{v,v'\}+\{xt,x't\}$ is $(2,2)$-$C$-tight.
\end{lem}

\begin{proof}
Since $t$ is a fixed vertex, the edge $xx'$ does not exist.
We therefore only have to consider whether $xt$ and $x't$ are edges of $G$.
If $xt, x't \in E$, then $G[N[v]\cup N[v']] = (W_5,\phi_4)$.
So suppose $xt, x't \notin E$.
Suppose there exist sets $W_1, W_2 \subset V$ that are both $2$-critical, with $x,t \in W_1$, $x',t \in W_2$.
Then $W_1 \cup W_2$ is $2$-critical and the subgraph induced by $W_1 \cup W_2  \cup \{v,v'\}$ contradicts the $(2,2)$-sparsity of $G$.
Similarly, any $3$-critical blocking set $U$ containing $x,x',t$ would induce a subgraph that breaks $(2,2)$-sparsity after adding $v,v'$ and their incident edges.
Finally, for $C = C_2$, $xt$ cannot be blocked by a $4$-critical set $T$, as they cannot contain fixed vertices and $t$ itself is fixed.
\end{proof}

\subsection{Contraction operations}

\begin{lem}\label{deg(v)=3, K4}
Let $(G,\phi)$ be $(2,2)$-$C$-tight for $C\in \{C_i,C_2,C_s\}$.
Suppose $G$ contains a copy of $K_{4}$ with vertices $\{x_{1}, x_{2}, x_{3}, x_{4}\} = X$, and put $\{x_{1}', x_{2}', x_{3}', x_{4}'\} = X'$ where $X\neq X'$.
Let $G^-$ denote the graph obtained from $G$ by contracting $X$ to $w$ and $X'$  to $w'$ so that, for any $v\in V\setminus (X\cup X')$ with $vx_{i} \in E$ (resp. $vx_{i}' \in E$), we have $vw \in E(G^-)$ (resp. $vw' \in E(G^-)$).
Then either
\begin{enumerate}
    \item  $G^-$ is $(2,2)$-$C$-tight, 
    \item there exists $y \in V\setminus X$ such that $yx_{i}, yx_{j} \in E$ for some $1 \leq i <j \leq 4$,
    \item $C=C_i, C_2, C_s$ and $G[X, X'] \cong (F_{2},\phi_2)$,
    \item $C=C_2, C_s$ and $G[X,X'] \cong (Wd(4,2),\phi_5)$.
\end{enumerate}
\end{lem}

\begin{proof}
First note that for any $C$, $|X \cap X'| \leq 1$ since $G$ is $(2,2)$-tight.
If $|X \cap X'| = 1$, this vertex must be fixed by any of the symmetries, so $C = C_2$ or $C_s$, and $G[X,X'] \cong (Wd(4,2),\phi_5)$, which is condition (4).
We may therefore suppose $X \cap X' = \emptyset$.
Let $G^{-}$ be as above.
Observe that $C$-symmetry is preserved in the reduction operation.
We have $|V(G^{-})|= |V|-6$ and $|E(G^{-})| = |E|-12$.
We first show that if $G^{-}$ is simple, then it is $(2,2)$-tight.
By construction, 
$$|E(G^{-})| = |E|-12 = 2|V| - 2-12 = 2(|V|-6) -2 = 2|V(G^{-})| -2.$$
Now consider $F \leq G^{-}$.
If $w, w' \notin V(F)$, then $F$ is a subgraph of $G$. Since $G$ is $(2,2)$-tight, $|E(F)| \leq 2|V(F)| - 2$.
Any subgraph containing $w$ or $w'$ can be compared to a subgraph $F' \leq G$, by replacing $w, w'$ with $X,X'$ respectively, as well as making the appropriate edge set adjustment.
From $F'$ being a subgraph of $G$ it easily follows that $F$ is $(2,2)$-sparse, so $G^{-}$ is $(2,2)$-$C$-tight.

We next consider when the operation could create multiple edges.
Let $t$ denote the number of neighbours in $X$ of a vertex $v \in V\setminus X$. Note that $t \leq 2$ as $i_{G^{-}}(\{x_{1}, x_{2}, x_{3}, x_{4},v\}) = 6 +t \leq 8$.
If $t = 2$, we create an edge of multiplicity two between $v$ and $w$. This gives condition (2).
The other possibility is for a multiple edge between $w$ and $w'$.
This will happen when $d(X,X') \geq 2$.
Since $i_{G^{-}}(\{x_{1}, x_{2}, x_{3}, x_{4},x_{1}', x_{2}', x_{3}', x_{4}'\}) \leq 14$, there can be at most two such edges.
When this is an equality, $G[X, X'] \cong (F_{2},\phi_2)$, giving condition (3).
Thus we may perform the reduction operation and the resulting graph $G^-$ is $(2,2)$-$C$-tight, which is condition (1) and completes the proof.
\end{proof}

\begin{lem}\label{deg(v)=3, cycle contraction}
Let $(G,\phi)$ be $(2,2)$-$C$-tight for $C\in \{C_i,C_2,C_s\}$ and let $X$ be a copy of $K_4$ in $G$ which contains a node $v$ and $X\cap X'=\emptyset$.
Suppose we cannot contract $X$ since there exists $y \in V$ with two edges to distinct vertices, say $a,b$ in $X$.
Then there is a $C$-symmetric $C_{4}$ contraction that results in a $(2,2)$-$C$-tight graph.
\end{lem}

\begin{proof}
Label the final vertex of $X$ as $c$.
We write $H = G[\{a,b,c,v,y\}]$.
Note that $vy \notin E$, and so $G[\{a,b,v,y\}] \cong K_{4}-e$. Hence there is a potential $C_{4}$ contraction, with $v \to y$.
We claim that this $C_{4}$ contraction results in a smaller $(2,2)$-tight graph and hence the $C$-symmetric $C_4$ contraction results in a $(2,2)$-$C$-tight graph.
We begin by noting that there is no 2-critical set $U$ containing $v,y$ and at most one of $a,b$ (otherwise adding the vertices of $H$ not contained in $U$ and their incident edges violates $(2,2)$-sparsity). Similarly there is no 3-critical set containing $v,y$ but not $a,b$. 

Since $v$ is a node, and $a,b \in N(y)$, $c \notin N(y)$, the subgraphs of the contracted graph we are interested in will contain one or both of the edges $cy, c'y'$. 
Suppose there exists a  $2$-critical set $U$ with $\{c,y\} \in U$, and $v \notin U$. Then $U\cup v$ is 3-critical and hence does not exist as above.
Similarly there is no $2$-critical set containing $\{c',y'\}$.
To complete the proof we check that there is no $3$-critical set $W$ containing $c, y, c'$ and $y'$.
Let $L = W + \{a,b,a',b'\}$. Since $ac, ay, bc, by, a'c', a'y', b'c', b'y' \in E$, we have $i_{G}(L) \geq 2|L| - 3$.
However, we then see that 
$$i_{G}(L + \{v, v'\}) \geq 2|L| - 3 + 6 = 2|L + \{v, v'\}| -1,$$ contradicting $G$ being $(2,2)$-tight.
\end{proof}

\begin{lem}\label{lem: 2,2 tight to vertex}
Let $(G,\phi)$ be $(2,2)$-$C$-tight for $C \in \{C_2,C_s\}$ with no fixed edges, $\delta(G) \geq 3$ and let $H \leq G$ be a proper subgraph.
If $H$ is $(2,2)$-$C$-tight then there exists a proper tight subgraph $F$ of $G$, with $H \leq F$ such that $G/F$ is $(2,2)$-$C$-tight.
\end{lem}

\begin{proof}
We begin by noting that unless there exists a $y_1 \in V \setminus V(H)$ that is adjacent to two vertices of $H$, we can contract $H$ to a fixed vertex to create a simple graph $G/H$, and $$|E(G/H)| = |E| - |E(H)| = 2|V| - 2 - 2|V(H)| + 2 = 2(|V(G/H)|-1).$$
Any subgraph of $G/H$ which breaks $(2,2)$-sparsity either does not contain the contracted vertex and hence trivially breaks the $(2,2)$-sparsity of $G$, or does and the obvious corresponding subgraph breaks the $(2,2)$-sparsity of $G$ since $K_1$ and $H$ are both $(2,2)$-tight.
If $G$ is $C_2$-symmetric, for $H$ to be $(2,2)$-$C_2$-tight it must contain the fixed vertex of $G$, and since $H$ contracts to a fixed vertex in $G/H$ it would be the only such fixed vertex.
This contraction preserves $C$-symmetry, so $G/H$ would be $(2,2)$-$C$-tight.
If such a $y_1$ exists, then let $H_1$ be the subgraph of $G$ including $H$ and $y_1$ (and $y_1'$ if $y_1$ is not fixed).
Note that $H_1$ is also $(2,2)$-$C$-tight.
By the same reasoning as above, $H_1$ can be contracted to a fixed vertex unless there exists $y_2$ adjacent to two vertices of $H_1$.
This sequence must end with a proper tight subgraph $F = H_k$ as $\delta(G) \geq 3$, completing the proof.
\end{proof}

\section{$C_i$-symmetric isostatic graphs}
\label{sec:inversion}

We now focus exclusively on $C_i$ symmetry and put together the combinatorial analysis to this point to prove the following recursive construction. From this we then deduce our characterisation of completely $C_i$-regular isostatic frameworks. We need one final lemma first.

\begin{lem}\label{deg(v)=3, lem G2edgesep}
Let $(G,\phi)$ be a $(2,2)$-$C_i$-tight graph distinct from $(F_{1}, \phi_1)$ and $(F_{2},\phi_2)$.
If all nodes are in copies of $(F_1,\phi_1)$ or $(F_2,\phi_2)$, then $G$ contains a 2-edge-separating set $S$. Further, let $G_1,G_2$ be the connected components of $G-S$. Then both $G_1$ and $G_2$ are $(2,2)$-$C_i$-tight with one of the $G_i$ being isomorphic to $(F_1,\phi_1)$ or $(F_2,\phi_2)$.
\end{lem}

\begin{proof}
Let $k$ be the number of $(2,2)$-$C_i$-tight subgraphs which are isomorphic to $(F_1,\phi_1)$ or $(F_2,\phi_2)$.
We first show that these $k$ $(2,2)$-$C_i$-tight subgraphs cannot have intersecting vertex sets.
Two $C_i$-symmetric subgraphs cannot have an intersection of size $1$, since the intersection is $C_i$-symmetric and there are no fixed vertices.
Since $G$ is $(2,2)$-sparse the intersection of any two $F_i$'s is $2$-critical.
Hence the intersection is of size at least four.
Since each $F_i$ is $C_i$-symmetric, their intersection must be, so $H = F_i \cap F_j$ is a proper $(2,2)$-$C_i$-tight subgraph with $4\leq |V(H)|\leq 6$.
Since $F_1$ is not a subgraph of $F_2$, and $K_4$ is not $C_i$-symmetric, this means that all of the $F_i$'s are pairwise vertex disjoint.
Let $v_{0}$ be the number of vertices of $G$ in these $k$ $(2,2)$-$C_i$-tight subgraphs, $r = |V| - v_{0}$,  $e_{0} = 2v_{0} - 2k$ be the number of edges of $G$ in these $k$ subgraphs, and $s = |E| - e_{0}$.

Since $|E| = 2|V| -2$, we can now deduce, with substitutions from above, that $s + e_{0} = 2r + 2v_{0} - 2$, and hence 
$$s + 2v_{0} - 2k = 2r +2v_{0} -2.$$ 
This implies that $s = 2r + 2k -2$.
Let $H_1,H_2,\dots, H_k$ denote the $k$ copies of $F_1,F_2$.
For any $1\leq j \leq k$, $G\setminus H_j$ is $(2,2)$-$C_i$-tight and $d(H_j, G\setminus H_j)$ is even, since no edges of $G$ are fixed by the inversion.
Each of the $r$ vertices not in some $H_j$ are of degree at least four.
Counting incidences, we see $2s \geq 4r + \sum_{i=1}^{k} a_{i}$ where for each $i$, $a_{i} \in \{2,4,\dots\}$ is counting the number of edges incident to each $H_{j}$.
We can substitute $s$ from the above to obtain
$$2(2r + 2k - 2) \geq 4r + \sum_{i=1}^{k} a_{i},$$ 
and cancelling gives $4k - 4 \geq \sum_{i=1}^{k} a_{i}$.
This means at least two of the $a_{i}$ are equal to two, so at least two $(2,2)$-$C_i$-tight subgraphs can be separated from $G$ with the removal of two edges, hence $G$ contains a 2-edge-separating set $S$. 

Let $G_1, G_2$ be the components of $G-S$.
We know from the above that one component is isomorphic to $(F_1,\phi_1)$ or $(F_2,\phi_2)$, without loss of generality say $G_2$.
Then $G_1$ is $(2,2)$-tight and contains a copy of $F_1$ or $F_2$ which is $(2,2)$-$C_i$-tight.
This gives us that $\varphi (G_{1}) \cap G_{1} \neq \emptyset$.
Further, we note that $G_1$ inherits inversion symmetry from $G$ and $S \cap \varphi(G_1) = \emptyset$. Since $\varphi (G_{1})$ is connected, this implies $\varphi (G_{1}) = G_{1}$.
Since $\varphi$ fixes no vertices or edges of $G$, it will not fix any vertices or edges of $G_{i}$. Hence $G_1$ is $(2,2)$-$C_i$-tight.
\end{proof}

\begin{thm}\label{thm:cirecursion}
A graph $(G,\phi)$ is $(2,2)$-$C_i$-tight if and only if $(G,\phi)$ can be generated from $(F_{1},\phi_1)$ or $(F_{2},\phi_2)$ by symmetrised 0-extensions, 1-extensions, vertex-to-$K_{4}$ operations, vertex-to-$C_{4}$ operations, and joining such a graphs to a copy of $(F_1,\phi_1)$ or $(F_2,\phi_2)$ by two new distinct edges that are images of each other under $\varphi$.
\end{thm}

\begin{proof}
We first show that if $G$ can be generated from the stated operations, then it is $(2,2)$-$C_i$-tight.
Note that $(F_{1},\phi_1)$ and $(F_{2},\phi_2)$ are independent and $(2,2)$-$C_i$-tight.
In Section \ref{sec:ops} we showed that the named operations preserve independence.
It is clear these operations preserve the top count and do not introduce fixed edges.
Thus, if we apply these operations to an independent and $(2,2)$-$C_i$-tight graph, the result will also be independent and have the correct edge counts described in Section \ref{sec:necrep}.
Thus the new graph must be $(2,2)$-$C_i$-tight from Theorem \ref{fixed e&v on cy}.

For the converse, we show by induction that any $(2,2)$-$C_i$-tight graph $G$ can be generated from a copy of $(F_{1},\phi_1)$ or $(F_{2},\phi_2)$.
Suppose the induction hypothesis holds for all graphs with $|V| < n$.
Now let $|V|= n$ and suppose $G$ is not isomorphic to either of the base graphs $(F_1,\phi_1)$ and $(F_2,\phi_2)$.
We wish to show that there is an operation from our list taking $G$ to a $(2,2)$-$C_i$-tight graph $G^- = (V^-,E^-)$ with $|V^-| < n$.
Then we know that $G^-$ can be generated from a copy of $(F_{1},\phi_1)$ or $(F_{2},\phi_2)$, and hence so can $G$.
We first note that any $(2,2)$-$C_i$-tight graph $G$ has $2 \leq \delta(G) \leq 3$. There is no $v \in V$ with $d(v) = 0,1$, as then $G-v$ would break sparsity. By the handshaking lemma, if all vertices are at least degree $4$, then $|E| \geq 2|V|$.
If $\delta(G) = 2$, then we remove any degree 2 and its symmetric copy. This yields a $(2,2)$-$C_i$-tight graph by Lemma \ref{deg(v)=2}, and this graph $G^- = (V^-,E^-)$ has $|V^-| = n-2$ as required.
Otherwise $\delta(G) = 3$.

If there exists a degree three vertex $v \in V$ with $N(v) \cap N(v') = \emptyset$, with $G[N[v]] \ncong K_{4}$, then we perform a $C_i$-symmetric 1-reduction, which is possible by Lemma \ref{lem: deg(v)=3, cap empty}. If $N(v) \cap N(v') \neq \emptyset$ and $G[N[v] \cup N[v']] \ncong F_{1}$, then we again perform a symmetrised 1-reduction which is possible by Lemma \ref{deg(v)=3, x=y'}. In both cases, the new graph $G^- = (V^-,E^-)$ also has $|V^-| = n-2$ as required.
Otherwise, all nodes are in copies of $K_{4}$ or $(F_{1},\phi_1)$.

Now suppose $G$ contains a subgraph isomorphic to $K_4$ and consider a contraction of this $K_4$.
By Lemma \ref{deg(v)=3, K4}, this $K_4$ can be reduced unless there is a vertex with two neighbours in the $K_{4}$, or the $K_{4}$ is part of a subgraph isomorphic to $(F_{2},\phi_2)$. In the former case, we use Lemma \ref{deg(v)=3, cycle contraction}, and $G^- = (V^-,E^-)$ is a $(2,2)$-$C_i$-tight graph with $|V^-| < n$.
In the latter case, all nodes are in $(2,2)$-$C_i$-tight subgraphs isomorphic to $(F_1,\phi_1)$ or $(F_2,\phi_2)$ and we recall $G$ is not isomorphic to $(F_{2},\phi_2)$.
Hence we may apply Lemma \ref{deg(v)=3, lem G2edgesep} to deduce that $G$ contains a two edge seperating set $S$, so that $G - S$ has two connected components $G_{1}, G_{2}$, where without loss of generality $G_{1}$ is $(2,2)$-$C_i$-tight and $G_2$ is isomorphic to $(F_1,\phi_1)$ or $(F_2,\phi_2)$.
Writing $G_{1} = (V_{1},E_{1})$, we have $|V_{1}| < n$ so $G_1$ and by extension $G$ can be generated from $(F_{1},\phi_1)$ or $(F_{2},\phi_2)$.
Finally, since $G$ is not isomorphic to $(F_1,\phi_1)$ or $(F_{2},\phi_2)$, we are finished.
\end{proof}

\begin{thm}\label{thm:mainci}
A graph $(G,\phi)$ is $C_i$-isostatic if and only if it is $(2,2)$-$C_i$-tight.
\end{thm}

\begin{proof}
Necessity was proved in Theorem \ref{fixed e&v on cy}.
It is easy to check using any computer algebra package that the base graphs $(F_1,\phi_1)$ and $(F_2,\phi_2)$ are $C_i$-isostatic.
Sufficiency follows from Theorem \ref{thm:cirecursion} and the results of Section \ref{sec:ops}, namely Lemmas \ref{0-ext rigid}--\ref{cycle rigid},
 by induction on $|V|$.
\end{proof}

\section{$C_2$-symmetric isostatic graphs}
\label{sec:half}

In this section we turn our attention to $C_2$-symmetric graphs on the cylinder.
In our recursive construction we will take care to maintain the number of fixed edges and vertices in each operation, and hence we will essentially view the cases of two fixed edges and no fixed vertex as disjoint from the case of no fixed edge and one fixed vertex.

\subsection{Reduction operations}

In the $C_{i}$-symmetric case, when looking at 1-reductions, we considered the induced subgraphs on open neighbourhoods of the vertex we wished to remove.
However, for $C_2$ symmetry, we must consider closed neighbourhoods, as we may have fixed edges.
The options for the intersection of the closed neighbour sets of a node, say $v$, and its image $v'$ are: empty intersection; one vertex in the intersection, where the vertex in the intersection will be fixed; two vertices in the intersection, where $v$ and $v'$ are both adjacent to a vertex and its image under the half-turn or where $vv'\in E$; three vertices in the intersection, with one vertex fixed and no fixed edges; four vertices in the intersection, and the vertices form either $K_{4}$ or $K_{4}-e$ as an induced subgraph.
Note that the two cases above with fixed vertices were shown to be reducible in Section $\ref{sec:merged}$.

We recall from Lemma \ref{lem: no blocking}, for $C_i$ and $C_s$ symmetry, we had to consider $2$- and $3$-critical sets which prevent a symmetrised $1$-reduction.
These both need to be considered with $C_2$ symmetry, but the conditions that a $(2,2)$-$C_2$-tight graph has one fixed vertex and no fixed edges or no fixed vertex and two fixed edges means we must now also consider $4$-critical sets which do not have any fixed edges or vertices.
Performing a symmetrised $1$-reduction which adds two edges to such a set would violate our conditions for $(2,2)$-$C_2$-tightness.

\begin{lem}\label{4-crit no int}
Let $(G,\phi)$ be $(2,2)$-$C_2$-tight and suppose $v\in V$ is a node with $N[v] \cap N[v']$ either empty or consisting of only one fixed vertex and suppose $i_G(N(v)) \leq 1$.
If there is a $4$-critical $C_2$-symmetric subset $T\subset V-\{v,v'\}$, with $G[T]$ containing no fixed edges or vertices, and containing two non-adjacent vertices of $N(v)$, then there exists a $C_2$-symmetric $1$-reduction at $v$ that results in a $(2,2)$-$C_2$-tight graph.
\end{lem}

\begin{proof}
Let $N(v) = \{x,y,z\}$ and let $T\subset V-\{v,v'\}$ be a $4$-critical $C_2$-symmetric set containing two non-adjacent vertices of $N(v)$.
Without loss of generality, we may suppose $xy, xz \notin E$ and $x,y\in T$. Note that $z\notin T$.
We show that either $G' = G - \{v,v'\} + \{xz,x'z'\}$ or $yz \notin E$ and $G' = G - \{v,v'\} + \{yz,y'z'\}$ is $(2,2)$-$C_2$-tight.
We first prove that there cannot exist a $4$-critical $C_2$-symmetric set $T_1$ such that $G[T_1]$ contains no fixed edges or vertices and $x,z\in T_1$.
Suppose to the contrary, that $T_1$ exists.
As $T_1 \cap T \neq \emptyset$, and both $T_1 \cap T$ and $T_1 \cup T$ are $C_2$-symmetric and the induced subgraphs do not contain fixed edges or vertices, we have $i(T_1 \cap T) \leq 2|T_1 \cap T|-4$ and $i(T_1 \cup T) \leq 2|T_1 \cup T|-4$.
Then
\begin{align*}
    2|T_1| -4 + 2|T| -4 & = i(T_1) + i(T) = i(T_1 \cup T) + i(T_1 \cap T) - d(T_1, T)\\
    & \leq 2|T_1 \cup T| -4 + 2|T_1 \cup T| -4 =  2|T_1| + 2|T| -8.
\end{align*}
Hence equality holds and $T_1 \cap T$ and $T_1 \cup T$ are $4$-critical.
This is a contradiction as $T_1 \cup T \cup \{v,v'\}$ would be $2$-critical with no fixed edge and no fixed vertex induced by this set.
Similarly if $yz \notin E$, then there does not exist a $4$-critical $C_2$-symmetric set $T_2$ such that $G[T_2]$ contains no fixed edges or vertices and $y,z\in T_2$.

Assume now that there exist two $2$-critical sets $U_1$ and $U_2$ containing $\{x,z\}$ and $\{x',z'\}$ respectively.
We may assume $U_2 = U_1'$, for otherwise we could consider $U_3 = U_1 \cup U_2'$ and $U_3' = U_1' \cup U_2$. 
Let $U = U_1 \cup U_2$.
Note that if $U_1 \cap U_2 = \emptyset$ then $U$ is $4$-critical.
Otherwise, by Remark \ref{critical cup and cap}, $U$ is $2$-critical, and since $G[U]$ is $C_2$-symmetric, it contains the fixed edges or vertex.
It follows that $T \cup U = T \cup U_1 \cup U_2$ is 4-critical. We have
\begin{align*}
    2|T| -4 + 2|U| -a & = i(T) + i(U) = i(T \cup U) + i(T \cap U) - d(T, U)\\
    & \leq 2|T \cup U| -4 + 2|T \cap U| -2 =  2|T| + 2|U| -6.
\end{align*}
If $U$ is 4-critical this would imply that $G[T\cup U \cup \{v,v'\}]$ is $(2,2)$-tight and $C_2$-symmetric but does not contain fixed elements, which contradicts Theorem \ref{fixed e&v on cy}.
So we may suppose $U$ is 2-critical and $d(T,U)=0$, that is $yz, y'z' \notin E$.
Then there cannot exist a 2-critical set on $\{y,z\}$ or $\{y',z'\}$, as say $y,z \in X$ was 2-critical, $G[U\cup X \cup \{v,v'\}]$ would not be sparse.

Finally, assume there exists a $3$-critical set $W$ containing $\{x,z,x',z'\}$ or when $yz \notin E$, $\{y,z,y',z'\}$.
We can assume this set is $C_2$-symmetric by taking $W \cup W'$.
Since the induced subgraph contains only one fixed edge, $i(W\cup T) \leq 2|W\cup T| -3$, and $i(W\cap T) \leq 2|W\cap T| -4$.
By similar calculations as we did for $4$ and $2$-critical sets, we see that in the equations above equality holds throughout, and hence $T\cup W \cup \{v,v'\}$ breaks $(2,2)$-sparsity of $G$.
Then by Lemma \ref{lem: no blocking}, either $G' = G - \{v,v'\} + \{xz,x'z'\}$ or $yz \notin E$ and $G' = G - \{v,v'\} + \{yz,y'z'\}$ is $(2,2)$-$C_2$-tight as required.
\end{proof}

\begin{lem}\label{deg 3 C_2 empty}
Let $(G,\phi)$ be $(2,2)$-$C_2$-tight and suppose $v \in V$ is a node with $N[v] \cap N[v'] = \emptyset$. Then either $G[N[v]] = K_{4}$, or there exists $x,y \in N(v)$ such that $xy \notin E$, and $G^- = G-\{v,v'\}+\{xy,x'y'\}$ is $(2,2)$-$C_2$-tight.
\end{lem}

\begin{proof}
We break up this proof into cases by looking at the number of edges amongst the neighbours of $v$.
Label the neighbours of $v$ by $x,y,z$.
Firstly, when all $3$ edges $xy, xz, yz$ are present in the graph, we have a $K_{4}$.
Next suppose two edges are present, say without loss of generality $xy \notin E$.
Suppose there exists a $2$-critical set $U \subset V-v$ with $x,y \in U$. Then the subgraph induced by $U \cup \{v,z\}$ violates the $(2,2)$-sparsity of $G$.
To do the $1$-reduction symmetrically, we must check that there is no $W \subset V-v$ with $x,y,x',y' \in W$ such that $|E(W)|=2|V(W)|-3$.
This follows since the subgraph induced by $W\cup\{v,z,v',z'\}$ breaks $(2,2)$-sparsity.
If there exists a $4$-critical $C_2$-symmetric subset $T\subset V-\{v,v'\}$ containing $x,y,x',y'$ then $T \cup \{v,z,v',z'\}$ is $2$-critical and $C_2$-symmetric, so all fixed edges and/or vertices are contained in $G[T]$.

For the case with one or zero edges amongst $x,y,z$, we begin by noting that 
no two of the pairs $\{x,y\},\{x,z\},\{y,z\}$ can each be contained in a $2$-critical set, as if any two were contained in $2$-critical sets $U_1, U_2$, then, by Remark \ref{critical cup and cap}, $U_1 \cup U_2$ is $2$-critical and $U_1 \cup U_2 +v$ violates the $(2,2)$-sparsity of $G$.
If $\{v_1,v_2\} \in \{\{x,y\},\{x,z\},\{y,z\}\}$ was the only pair not in a $2$-critical set, and $\{v_2,v_3\} \in \{\{x,y\},\{x,z\},\{y,z\}\}\setminus \{v_1,v_2\}$ is in a $2$-critical set $U$ (note this implies $v_1v_3 \in E$ or we are in the final case below), but $\{v_1,v_2,v_1',v_2'\}$ was in a $3$-critical set $W$, then $W\cup U \cup U'$ is $3$-critical and there is a subgraph containing $W\cup U \cup U' \cup \{v,v'\}$ which breaks $(2,2)$-sparsity.
If there exists a $4$-critical $C_2$-symmetric subset $T\subset V-\{v,v'\}$ containing $v_1,v_2,v_1',v_2'$ then $T \cup \{v,v_3,v',v_3'\}$ is $2$-critical and $C_2$-symmetric.
Hence by
Lemma \ref{4-crit no int}, if $\{v_1,v_2\}$  was the only pair not in a $2$-critical set we can perform a $C_2$-symmetric $1$-reduction at $v$ in this case.

Finally we must consider when there does not exist $2$-critical sets containing $\{v_1,v_2\}$ and $\{v_2,v_3\}$ respectively with $\{v_1,v_2,v_3\} = \{x,y,z\}$. (Whether there is a 2-critical set containing $v_1,v_3$ is not important for the argument that follows.)
Assume for a contradiction that $W_1,W_2 \subset V-v$ are $3$-critical with $\{v_1, v_2,v_1',v_2'\} \in W_1$, $\{v_2,v_3,v_2',v_3'\} \in W_2$.
By counting similar to Remark $\ref{critical cup and cap}$, the union and intersection of two $3$-critical sets are either both $3$-critical or one is $2$-critical and the other is $4$-critical.
Since $W_1 \cap W_1'$ and $W_1 \cup W_1'$ contain $\{v_1,v_2\}$ neither are 2-critical (similarly $W_2 \cap W_2'$ and $W_2 \cup W_2'$ are not 2-critical since they both contain $\{v_2,v_3\}$).
Hence both $W_1 \cup W_1'$ and $W_2 \cup W_2'$ are $3$-critical and $C_2$-symmetric, so the subgraphs induced by these sets must each contain exactly $1$ fixed edge.
If they do not contain the same fixed edge, $(W_1 \cup W_1')\cap (W_2 \cup W_2')$ is $C_2$-symmetric and contains no fixed edges or vertices, so must be $4$-critical, which would imply $(W_1 \cup W_1')\cup (W_2 \cup W_2')$ is $2$-critical but then $\{v_1,v_2\}$ is contained in a $2$-critical set.
If the induced subgraphs do contain the same fixed edge, both $(W_1 \cup W_1')\cap (W_2 \cup W_2')$ and $(W_1 \cup W_1')\cup (W_2 \cup W_2')$ would be $3$-critical, but then the subgraph induced by $(W_1 \cup W_1')\cup (W_2 \cup W_2') + \{v,v'\}$ would violate $(2,2)$-sparsity. Hence one of the pairs $v_i,v_j$ is not contained in a 2-critical or a 3-critical subset of $V-v$.
It remains to deal with the case when this pair $v_i,v_j$ is blocked by a 4-critical subset $T\subset V-\{v,v'\}$. It follows from
Lemma \ref{4-crit no int} that we can reduce $v$ symmetrically and the proof is complete.
\end{proof}

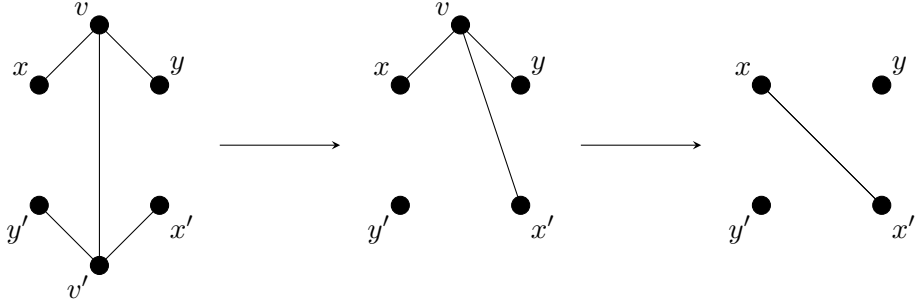
\begin{figure}
    \centering
    \begin{tikzpicture}
  [scale=.8,auto=left]
  
  \coordinate (n1) at (1,4);
  \coordinate (n2) at (0,3);
  \coordinate (n3) at (2,3);
  \coordinate (n4) at (0,1);
  \coordinate (n5) at (2,1);
  \coordinate (n6) at (1,0);
  
 \draw[fill=black] (n1) circle (0.15cm)
    node[above left] {$v$};
 \draw[fill=black] (n2) circle (0.15cm)
    node[above left] {$x$};
 \draw[fill=black] (n3) circle (0.15cm)
    node[above right] {$y$};
 \draw[fill=black] (n4) circle (0.15cm)
    node[below left] {$y'$};
 \draw[fill=black] (n5) circle (0.15cm)
    node[below right] {$x'$};
 \draw[fill=black] (n6) circle (0.15cm)
    node[below left] {$v'$};

  \foreach \from/\to in {n1/n6,n1/n2,n1/n3,n6/n4,n6/n5} 
    \draw (\from) -- (\to);
    
  \coordinate (n11) at (7,4);
  \coordinate (n12) at (6,3);
  \coordinate (n13) at (8,3);
  \coordinate (n14) at (6,1);
  \coordinate (n15) at (8,1);
  
 \draw[fill=black] (n11) circle (0.15cm)
    node[above left] {$v$};
 \draw[fill=black] (n12) circle (0.15cm)
    node[above left] {$x$};
 \draw[fill=black] (n13) circle (0.15cm)
    node[above right] {$y$};
 \draw[fill=black] (n14) circle (0.15cm)
    node[below left] {$y'$};
 \draw[fill=black] (n15) circle (0.15cm)
    node[below right] {$x'$};

  \foreach \from/\to in {n11/n12,n11/n13,n11/n15} 
    \draw (\from) -- (\to);

  \coordinate (n22) at (12,3);
  \coordinate (n23) at (14,3);
  \coordinate (n24) at (12,1);
  \coordinate (n25) at (14,1);
  
 \draw[fill=black] (n22) circle (0.15cm)
    node[above left] {$x$};
 \draw[fill=black] (n23) circle (0.15cm)
    node[above right] {$y$};
 \draw[fill=black] (n24) circle (0.15cm)
    node[below left] {$y'$};
 \draw[fill=black] (n25) circle (0.15cm)
    node[below right] {$x'$};

  \foreach \from/\to in {n25/n22} 
    \draw (\from) -- (\to);

\foreach \from/\to in {n25/n22} 
    \draw (\from) -- (\to);

  \coordinate (n31) at (3,2);
  \coordinate (n32) at (5,2);
  \coordinate (n33) at (9,2);
  \coordinate (n34) at (11,2);
\draw [-stealth] (n31) -- (n32); \draw [-stealth] (n33) -- (n34);

\end{tikzpicture}
    \caption{Reduction schematic when the degree three vertex is adjacent to its symmetric image.}
    \label{extended base}
\end{figure}

\begin{lem}\label{deg 3 C_2 v,v'}
Let $(G,\phi)$ be $(2,2)$-$C_2$-tight and suppose $v \in V$ is a node with $N(v)=\{x,y,v'\}$ and $N[v] \cap N[v'] = \{v,v'\}$. 
\begin{enumerate}
    \item Suppose $xx', yy' \notin E$. Then either $G_1' = G-\{v,v'\}+\{xx'\}$ or $G_2' = G-\{v,v'\}+\{yy'\}$ is $(2,2)$-$C_2$-tight.
    \item Suppose $xx'\in E$ or $yy' \in E$. Then there is another node in $G$ and it is not of this type.
\end{enumerate}
\end{lem}

The following proof has two cases, firstly assuming the edges $xx'$ and $yy'$ are not present among the neighbours of $v$ and $v'$, and secondly assuming one is. (Note that it is not possible for both to be since $vv'\in E$ would give three fixed edges.)

\begin{proof}
For (1), we may perform a non-symmetric 1-reduction at $v'$ as it cannot happen that $\{v,x'\}$ and $\{v,y'\}$ can be in $2$-critical blocking sets, else the union of these sets, say $W$, is $2$-critical and $W + v'$ breaks sparsity.
To perform a second non-symmetric $1$-reduction at $v$, we see that neither $\{x,x'\}$ or $\{y,y'\}$ can be contained in a $2$-critical set.
If there were such a set, without loss of generality call it $U$ and let it contain $x,x'$, then $U \cup U'$ is $2$-critical ($x,x' \in U \cap U'$), $C_2$-symmetric, but cannot contain both of the fixed edges of $G$, which is a contradiction.

For (2), assume without loss of generality that $xx'\in E$.
Since $G$ is $(2,2)$-tight and $\delta(G)=3$ there are at least four nodes in $G$.
If $v,v',x,x'$ are the only nodes, then $G-\{v,v',x,x'\}$ is $(2,2)$-$C_2$-tight.
We now simply note that this arrangement can only appear once in each graph, since it has both of the fixed edges.
\end{proof}

In the above proof, for the second 1-reduction we are still considering $G$, rather than $G-v+xx'$.
We can do this since the blocking set in the reduced graph does not use $v$, therefore it does not include the edge $vv'$ so the blocking set without $v'$ would still be $2$-critical.

\begin{lem}\label{deg 3 C_2 x,x'}
Let $(G,\phi)$ be $(2,2)$-$C_2$-tight, suppose $v \in V$ is a node such that $N[v] \cap N[v'] = \{x,x'\}$ and let the other neighbour of $v$ be $z$.
Then $G_1' = G-\{v,v'\}+\{xz,yz'\}$ or $G_2' = G-\{v,v'\}+\{yz,xz'\}$ is $(2,2)$-$C_2$-tight.
\end{lem}

\begin{proof}
Again we apply Lemma \ref{lem no W}, $G$ has no fixed edges, if $v_1,v_2,v_1',v_2'$ are in a 3-critical set then they are in a 2-critical set too.
For the remainder of the proof we only consider 2-critical or 4-critical sets.
We prove this by case analysis, counting if the edges $xz$, $yz$, and $xy$ are present.
Firstly, $xz, x'z$ and $xx'$ cannot all be present, as the subgraph induced by $N[v]\cup N[v']$ breaks $(2,2)$-sparsity.
Further we do not have $xz,x'z \in E$ and $xx'\notin E$, as $N[v]\cup N[v']$ is $2$-critical, $C_2$-symmetric and the induced subgraph does not contain the correct fixed elements.
Our first case where a 1-reduction is possible is when one edge of $xz,x'z$ is present with $xx'$, say $x'z, xx' \in E$, $xz \notin E$.
If there exists a $2$-critical set $U$ containing $x,z$, not containing $v$, then the subgraph induced by $U \cup \{v,x'\}$ contradicts the $(2,2)$-sparsity of $G$.
If there exists a $3$-critical set $W$, with $x,x',z,z' \in U$, $v,v' \notin W$, then the subgraph induced by $W \cup \{v,v'\}$ also breaks the $(2,2)$-sparsity of $G$.
For any $4$-critical $T$ containing $x,z,x',z'$, $G[T]$ contains a fixed edge, namely $xx'$.
By counting, if $G[T]$ contains one fixed edge and $T$ is $4$-critical, it must contain both fixed edges, therefore there is no $4$-critical blocking set for the $1$-reduction at $v$ and $v'$.

Consider the case when one of the edges $xz,x'z$ is present, say $x'z \in E$, $xz, xx' \notin E$.
There does not exist $2$-critical $U_1, U_2$ with $x,z \in U_1$, $x,x' \in U_2$, as this contradicts Remark \ref{critical cup and cap} as $d(U_1, U_2) \neq 0$.
We therefore know that one of $G_1 = G - v + xx'$ or $G_2 = G - v + xz$ is $(2,2)$-tight, although not $C_2$-symmetric.
We want to show that it is always the case that we can perform a (non-symmetric) $1$-reduction at $v$ by adding the edge $xz$.
Suppose we add $xx'$.
Consider $1$-reductions at $v' \in G_1$.
If there exists a $2$-critical set $U$ containing $\{x',z'\}$, then the subgraph induced by $W\cup\{x',z'\}$ contradicts the $(2,2)$-sparsity of $G_1$.
Hence we can perform a $1$-reduction at $v'$ in $G_1$ adding the edge $x'z'$.
Since we could perform this one reduction in $G_1$, we know a $2$-critical set $U^*$ in $G$ preventing a $1$-reduction adding the edge $x'z'$ must contain $v$.
However, then the subgraph $H$ induced by $U^*\cup\{x,v'\}$ contradicts the $(2,2)$-sparsity of $G$, as $H$ contains the edges $xv,xz',v'x',v'z',v'x$.
Hence, we may perform a $1$-reduction at $v'$ in $G$ by adding the edge $x'z'$.

Now when $xz, x'z \notin E$, if both $\{x,z\}$ and $\{x',z\}$ are in $2$-critical sets $U_1$ and $U_2$ respectively, $U_1 \cup U_2$ is $2$ critical so $U_1 \cup U_2 \cup \{v\}$ contradicts $(2,2)$-sparsity of $G$.
There is no $4$-critical set $T$ containing $x,z,x',z'$.
Observe that such a set $T+\{v,v'\}$ would be $2$-critical, $C_2$-symmetric, but $G[T\cup\{v,v'\}]$ contains no fixed vertex or edge..
\end{proof}

\begin{lem}\label{deg 3 C_2 v v' x x'}
Let $(G,\phi)$ be $(2,2)$-$C_2$-tight and suppose $v \in V$ is a node such that $N[v] \cap N[v'] = \{v,v',x,x'\}$ and $xx' \notin E$. Then $G' = G-\{v,v'\}+\{xx'\}$ is $(2,2)$-$C_2$-tight.
\end{lem}

\begin{proof}
$G'$ is not $(2,2)$-$C_2$-tight if and only if there exists a $2$-critical set $X$ in $G-\{v,v'\}$ containing $x$ and $x'$. However $vv'$ is not in $G[X]$ so such a set $X$ cannot exist and the lemma follows.
\end{proof}

\subsection{Combinatorial characterisation}

We can now put together the combinatorial results of this section to prove the following recursive construction and then apply this result alongside the results of Section \ref{sec:ops} to deduce our characterisation of $C_2$-isostatic graphs.

\begin{thm}\label{thm:recursionc2}
A graph $(G,\phi)$ is $(2,2)$-$C_2$-tight if and only if $(G,\phi)$ can be generated from\break $(K_4,\phi_3), (W_5,\phi_4), (Wd(4,2),\phi_5), (F_2,\phi_2)$ (these graphs were depicted in Figure \ref{BaseGraphs,C_2,cylinder}) by symmetrised 0-extensions, 1-extensions, vertex-to-$K_{4}$ operations and vertex-to-$C_{4}$ operations.
\end{thm}

\begin{proof}
Each of the base graphs are independent and tight and Section \ref{sec:ops} showed the symmetrised 0-extension, 1-extension, double 1-extension, vertex-to-$K_4$, vertex-to-$C_4$ and vertex-to-$(2,2)$-$C_2$-tight operations preserve independence.
It is easy to see the operations also preserve $(2,2)$-tightness and the number of fixed elements.
We can therefore apply Theorem \ref{fixed e&v on cy} to the extended graph.

Conversely, we show by induction that any $(2,2)$-$C_2$-tight graph $G$ can be generated from our base graphs.
Suppose the induction hypothesis holds for all graphs with $|V| < n$.
Now let $|V|= n$ and suppose $G$ is not isomorphic to one of the base graphs in Figure $\ref{BaseGraphs,C_2,cylinder}$.
Obviously any $(2,2)$-$C_2$-tight graph contains a vertex of degree 2 or 3. The former case is dealt with by Lemma \ref{deg(v)=2}.
Hence suppose $\delta(G)=3$ and $v$ is a vertex of minimum degree.
There are 6 cases depending on the closed neighbourhood of $v$, namely with labelling from this section, $N[v]\cap N[v'] \in \{\emptyset,\{t\}, \{v,v'\},\{x,x'\}, \{t,x,x'\}, \{v,v', x,x'\}\}$.
By Lemmas \ref{deg 3 C_2 t}, \ref{deg 3 C_2 t x x'}, and \ref{deg 3 C_2 empty}--\ref{deg 3 C_2 v v' x x'} 
we see that the only blocks to reducing any given node are $K_4$ (either $C_2$-symmetric or non-symmetric) and the base graphs $(Wd(4,2),\phi_5)$ and $(W_5,\phi_4)$. (Note that if the option in Lemma \ref{deg 3 C_2 v,v'}(2) occurs then we may reduce the other node unless it is contained in a non-symmetric $K_4$.)

Suppose one of the base graphs in Figure \ref{BaseGraphs,C_2,cylinder} is a subgraph of $G$, denoted by $H$. 
If $H \cong (K_4, \phi_3)$ or $(F_2, \phi_2)$, $H$ contains all the fixed edges of $G$ and there can be no other base graph copy.
Otherwise, $H \cong (W_5,\phi_4)$ or $(Wd(4,2),\phi_5)$ and if another copy of  either $(W_5,\phi_4)$ or $(Wd(4,2),\phi_5)$ exist, call it $H_1$, then note that $H_1\cap H$ is precisely the fixed vertex.
Then $H_1$ is a proper $(2,2)$-$C_2$-tight subgraph of $G$ and we apply Lemma \ref{lem: 2,2 tight to vertex}. 
We may now suppose that $H$ is the only subgraph of $G$ which is a copy of a base graph depicted in Figure \ref{BaseGraphs,C_2,cylinder}.

We will show there is a node in $G$ not contained in $H$.
Note that $H$ has at least four degree three vertices.
Observing that $d(V(H), V(G\setminus H)) \geq 2$, the sum of the degrees in $H$ increases by at least two, meaning there must be two nodes in $G\setminus H$.
Hence we may assume all vertices of degree three are in a unique $(2,2)$-$C_2$-tight base graph or a $K_4$ copy which is not $C_2$-symmetric. We may now, in all cases, suppose that $G$ has a degree 3 that is contained in a $K_4$. We now apply Lemma
\ref{deg(v)=3, K4} and \ref{deg(v)=3, cycle contraction}
to complete the proof.
\end{proof}

\begin{thm}\label{thm:mainc2}
A graph $(G,\phi)$ is $C_2$-isostatic if and only if it is $(2,2)$-$C_2$-tight.
\end{thm}

\begin{proof}
Since $C_2$-isostatic graphs are $(2,2)$-tight, necessity follows from Theorem \ref{fixed e&v on cy}.
It is easy to check using any computer algebra package that the base graphs depicted in Figure \ref{BaseGraphs,C_2,cylinder} are $C_2$-isostatic.
Hence the sufficiency follows from Theorem \ref{thm:recursionc2} and Lemmas \ref{0-ext rigid}, \ref{1-ext rigid}, \ref{minrigidK_4}, \ref{cycle rigid}, \ref{double 1-ext rig} by induction on $|V|$.
\end{proof}

\section{$C_s$-symmetric isostatic graphs}
\label{sec:mirror}

We turn our attention to $C_s$-symmetric graphs on the cylinder.
Here $C_s$ is generated by a single reflection $\sigma$ which could contain the cylinder axis or be perpendicular to it.

\subsection{Reduction Operations}

For a $(2,2)$-$C_s$-tight graph, there are 6 possible cases for the structure of $N(v) \cap N(v')$, namely $N(v) \cap N(v') \in \{\emptyset, \{t\}, \{x,x'\}, \{t_1,t_2\}, \{x,x',t\}, \{t_1,t_2,t_3\}\}$, where vertices fixed by non-trivial element are denoted $t$, and those not fixed $x$.
In section \ref{sec:merged}, Lemmas \ref{lem: deg(v)=3, cap empty} -\ref{deg 3 C_2 t x x'} dealt with the first five of these cases.
These lemmas showed the reduction is possible, or the node is contained in a $(2,2)$-$C_s$-tight subgraph of $G$.
This leaves only the toughest case when all three neighbours of a node lie on the mirror.




Hence, for the remainder of this section we assume that all nodes have all neighbours on the mirror.
The following lemmas require some new notation for describing our graphs.
We will consider a vertex partition $V = V^r \cup V^b \cup V^g$ into red, blue and green vertices. The partition is chosen so that a vertex which is fixed by the mirror symmetry is red, any vertex which is adjacent to a red vertex is blue, and the remaining vertices are green.
This also gives us a notion of edge colouring.
We colour an edge red-blue if its endpoints are one red and one blue, blue-blue if its endpoints are blue, blue-green if its endpoints are one blue and one green, and green-green if its endpoints are green.
Note that red-red edges are not possible in a $(2,2)$-$C_s$-tight graph, and red-green edges are not possible by the choice of the partition.
We can therefore write $E = E^{rb} \cup E^{bb} \cup E^{bg} \cup E^{gg}$.

It will also be useful to consider the subgraphs of $G$ which consist of red-blue and blue-blue edges.
We will call these \textit{red-blue connected components} or \emph{rb-components} for shorthand.
We label the rb-components of a graph $A_1, \dots, A_k$, so the component $A_i = (V_i, E_i)$ is $k_i$-critical, has red vertex set $V_{i}^r$ and blue vertex set $V_{i}^b$, and has red-blue edges $E_{i}^{rb}$ and blue-blue edges $E_{i}^{bb}$ as in $G$.
A natural extension of this is to say that the subset of the edges $E^{bg}$ that are incident to a vertex in $A_i$ form a new set denoted $E_{i}^{bg}$.
Lastly, let $S \subset V^b$ be the nodes with all three neighbours on the mirror.
Let $s=|S|$ and $s_i = |S\cap V_i|$. We illustrate these definitions in Figure \ref{fig: rb-components}.

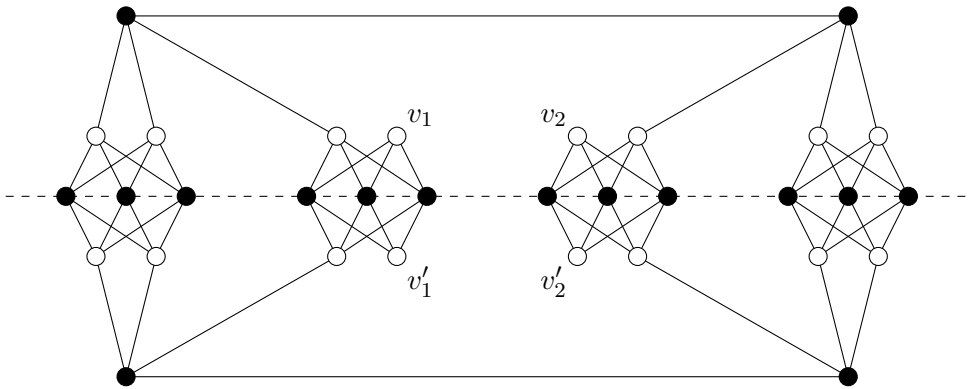
\begin{figure}[ht]
    \centering
    \begin{tikzpicture}
  [scale=.8,auto=left]
  
  \coordinate (n1) at (1,0);
  \coordinate (n2) at (1,6);
  \coordinate (n3) at (13,0);
  \coordinate (n4) at (13,6);
  \coordinate (n5) at (-1,3);
  \coordinate (n6) at (15,3);
  
 \draw[fill=black] (n1) circle (0.15cm);
 \draw[fill=black] (n2) circle (0.15cm);
 \draw[fill=black] (n3) circle (0.15cm);
 \draw[fill=black] (n4) circle (0.15cm);
  
  \coordinate (n11) at (0.5,2);
  \coordinate (n12) at (0.5,4);
  \coordinate (n13) at (0,3);
  \coordinate (n14) at (1,3);
  \coordinate (n15) at (2,3);
  \coordinate (n16) at (1.5,2);
  \coordinate (n17) at (1.5,4);

 \draw[fill=black] (n13) circle (0.15cm);
 \draw[fill=black] (n14) circle (0.15cm);
 \draw[fill=black] (n15) circle (0.15cm);

  \foreach \from/\to in {n11/n13,n11/n14,n11/n15,n12/n13,n12/n14,n12/n15,n13/n16,n13/n17,n14/n16,n14/n17,n15/n16,n15/n17} 
    \draw (\from) -- (\to);
    
  \coordinate (n21) at (4.5,2);
  \coordinate (n22) at (4.5,4);
  \coordinate (n23) at (4,3);
  \coordinate (n24) at (5,3);
  \coordinate (n25) at (6,3);
  \coordinate (n26) at (5.5,2);
  \coordinate (n27) at (5.5,4);
  
 \draw[fill=black] (n23) circle (0.15cm);
 \draw[fill=black] (n24) circle (0.15cm);
 \draw[fill=black] (n25) circle (0.15cm);

  \foreach \from/\to in {n21/n23,n21/n24,n21/n25,n22/n23,n22/n24,n22/n25,n23/n26,n23/n27,n24/n26,n24/n27,n25/n26,n25/n27} 
    \draw (\from) -- (\to);
    
  \coordinate (n31) at (8.5,2);
  \coordinate (n32) at (8.5,4);
  \coordinate (n33) at (8,3);
  \coordinate (n34) at (9,3);
  \coordinate (n35) at (10,3);
  \coordinate (n36) at (9.5,2);
  \coordinate (n37) at (9.5,4);
  
 \draw[fill=black] (n33) circle (0.15cm);
 \draw[fill=black] (n34) circle (0.15cm);
 \draw[fill=black] (n35) circle (0.15cm);

  \foreach \from/\to in {n31/n33,n31/n34,n31/n35,n32/n33,n32/n34,n32/n35,n33/n36,n33/n37,n34/n36,n34/n37,n35/n36,n35/n37} 
    \draw (\from) -- (\to);
    
  \coordinate (n41) at (12.5,2);
  \coordinate (n42) at (12.5,4);
  \coordinate (n43) at (12,3);
  \coordinate (n44) at (13,3);
  \coordinate (n45) at (14,3);
  \coordinate (n46) at (13.5,2);
  \coordinate (n47) at (13.5,4);

 \draw[fill=black] (n43) circle (0.15cm);
 \draw[fill=black] (n44) circle (0.15cm);
 \draw[fill=black] (n45) circle (0.15cm);

  \foreach \from/\to in {n41/n43,n41/n44,n41/n45,n42/n43,n42/n44,n42/n45,n43/n46,n43/n47,n44/n46,n44/n47,n45/n46,n45/n47} 
    \draw (\from) -- (\to);

  \foreach \from/\to in {n1/n11,n1/n16,n1/n21,n1/n3,n2/n12,n2/n17,n2/n22,n2/n4,n3/n36,n3/n41,n3/n46,n4/n37,n4/n42,n4/n47} 
    \draw (\from) -- (\to);

  \draw[dashed] (n5) -- (n6);
  
 \draw[fill=white] (n11) circle (0.15cm);
 \draw[fill=white] (n12) circle (0.15cm);
 \draw[fill=white] (n16) circle (0.15cm);
 \draw[fill=white] (n17) circle (0.15cm);
 \draw[fill=white] (n21) circle (0.15cm);
 \draw[fill=white] (n22) circle (0.15cm);
 \draw[fill=white] (n26) circle (0.15cm)
     node[below right] {$v_1'$};
 \draw[fill=white] (n27) circle (0.15cm)
     node[above right] {$v_1$};
 \draw[fill=white] (n31) circle (0.15cm)
     node[below left] {$v_2'$};
 \draw[fill=white] (n32) circle (0.15cm)
     node[above left] {$v_2$};
 \draw[fill=white] (n36) circle (0.15cm);
 \draw[fill=white] (n37) circle (0.15cm);  
 \draw[fill=white] (n41) circle (0.15cm);
 \draw[fill=white] (n42) circle (0.15cm);
 \draw[fill=white] (n46) circle (0.15cm);
 \draw[fill=white] (n47) circle (0.15cm);
    
\end{tikzpicture}
    \caption{A $(2,2)$-$C_s$-tight graph $G$. The red vertices lie on the mirror line, the blue vertices are depicted as unfilled circles and the green vertices are the filled vertices not on the mirror. Each copy of $K_{3,4}$ in $G$ is a rb-component and $S=\{v_1,v_1',v_2,v_2'\}$.}
    \label{fig: rb-components}
\end{figure}

\begin{lem}\label{lem: Vgempty}
Let $(G,\phi)$ be $(2,2)$-$C_s$-tight with $\delta(G) \geq 3$.
Then $G[V^r \cup V^b]$ is $(2,2)$-$C_s$-tight if and only if $V^g = \emptyset$.
Moreover, if $V^g\neq \emptyset$ then there exists an $i \in \{1,\dots,k\}$ such that $|E^{bg}_{i}| \leq 2k_i - 2$ and $s_i \geq 2k_i -|E^{bg}_{i}|$.
\end{lem}

\begin{proof}
If $V^g=\emptyset$ then $G[V^r \cup V^b]=G$ and hence it is $(2,2)$-$C_s$-tight. Conversely, 
we begin by noting that 
\begin{equation}\label{rbgtight}
    |E^{rb}|+|E^{bb}|+|E^{bg}|+|E^{gg}| = 2|V^r| + 2|V^b| + 2|V^g| -2.
\end{equation}
Then, for each $i \in \{1,\dots,k\}$, $|E^{rb}_{i}| + |E^{bb}_{i}| = 2|V^{r}_{i}|+2|V^{b}_{i}|-k_i$.
Summing gives
\begin{equation}\label{rbtight}
   |E^{rb}| + |E^{bb}| = 2|V^{r}|+2|V^{b}|- \sum_{i=1}^{k}k_i
\end{equation}
and then, by subtracting (\ref{rbtight}) from (\ref{rbgtight}), we obtain $|E^{bg}|+|E^{gg}| = 2|V^g| -2 + \sum_{i=1}^{k}k_i$.
Counting vertex degrees gives  $|E^{bg}|+2|E^{gg}| = \sum _{v \in V^g} d_G(v) \geq 4|V^g|$.
Therefore, $2|E^{bg}|+2|E^{gg}| = 4|V^g| -4 + 2\sum_{i=1}^{k}k_i \leq |E^{bg}|+2|E^{gg}| -4 + 2\sum_{i=1}^{k}k_i$. Rearranging and simplifying gives
\begin{equation}\label{bgedgeUB}
    |E^{bg}| \leq 2\sum_{i=1}^{k}k_i -4,
\end{equation}
which, for $k = 1$ and $k_1 = 2$, completes the first statement of the proof.

If $|E^{bg}_i| \geq 2k_i$ for all $i$, we would contradict Equation (\ref{bgedgeUB}).
Again counting vertex degrees,
\begin{equation}\label{reddeg}
    |E^{rb}_{i}| = \sum_{v \in V^{r}_{i}} d_G(v) \geq 4|V^{r}_{i}|
\end{equation}
and since the vertices of $S_i$ are nodes,
\begin{equation}\label{bluedeg}
    |E^{rb}_{i}| + 2|E^{bb}_{i}| + |E^{bg}_{i}| = \sum_{v \in V^{b}_{i}} d_G(v) \geq 4|V^{b}_{i}| - s_i.
\end{equation}
Adding Equations (\ref{reddeg}) and (\ref{bluedeg}), we see that $2|E^{rb}_{i}| + 2|E^{bb}_{i}| + |E^{bg}_{i}| \geq 4|V^{r}_{i}| + 4|V^{b}_{i}| - s_i$.
Now recalling Equation (\ref{rbtight}) (restricted to $A_i$), we obtain $$4|V^{r}_{i}| + 4|V^{b}_{i}| - s_i - |E^{bg}_{i}| \leq 2|E^{rb}_{i}| + 2|E^{bb}_{i}| \leq 4|V^{r}_{i}| + 4|V^{b}_{i}| - 2k_i,$$ 
which completes the proof.
\end{proof}

\begin{lem}\label{lem: node all mirror}
Let $(G,\phi)$ be a $(2,2)$-$C_s$-tight graph, distinct from $K_{3,4}$, with $\delta(G)= 3$. Suppose that the neighbour set of every node consists only of fixed vertices and that no proper subgraph $H$ of $G$ is $(2,2)$-$C_s$-tight.
Then there exists a $C_4$-contraction (which contracts two fixed vertices) that results in a $(2,2)$-$C_s$-tight graph.
\end{lem}

\begin{figure}[ht]
    \centering
    \begin{tikzpicture}
  [scale=.8,auto=left]
  
  \coordinate (n11) at (-1,0);
  \coordinate (n12) at (0,-1.5);
  \coordinate (n13) at (0,0);
  \coordinate (n14) at (0,1.5);
  \coordinate (n15) at (1,-1.5);
  \coordinate (n16) at (1,0);
  \coordinate (n17) at (1,1.5);
  \coordinate (n18) at (2,0);
  
 \draw[fill=black] (n11) circle (0.15cm);
 \draw[fill=black] (n12) circle (0.15cm)
 node [below left] {$v_{11}'$};
 \draw[fill=black] (n13) circle (0.15cm);
 \draw[fill=black] (n14) circle (0.15cm)
 node [above left] {$v_{11}$};
 \draw[fill=black] (n15) circle (0.15cm)
 node [below left] {$v_{12}'$};
 \draw[fill=black] (n16) circle (0.15cm);
 \draw[fill=black] (n17) circle (0.15cm)
 node [above left] {$v_{12}$};
 \draw[fill=black] (n18) circle (0.15cm);

  \foreach \from/\to in {n11/n12,n11/n14,n12/n13,n12/n16,n13/n14,n13/n15,n13/n17,n14/n16,n15/n16,n16/n17,n17/n18,n15/n18} 
    \draw (\from) -- (\to);
    
  \coordinate (n21) at (3.5,0);
  \coordinate (n22) at (4.5,-1.5);
  \coordinate (n23) at (4.5,0);
  \coordinate (n24) at (4.5,1.5);
  \coordinate (n25) at (5.5,0);
  
 \draw[fill=black] (n21) circle (0.15cm);
 \draw[fill=black] (n22) circle (0.15cm)
 node [below left] {$v_{21}'$};
 \draw[fill=black] (n23) circle (0.15cm);
 \draw[fill=black] (n24) circle (0.15cm)
 node [above left] {$v_{21}$};
 \draw[fill=black] (n25) circle (0.15cm);

  \foreach \from/\to in {n21/n22,n21/n24,n22/n23,n23/n24,n22/n25,n24/n25} 
    \draw (\from) -- (\to);
    
  \coordinate (n31) at (7,0);
  \coordinate (n32) at (8,-1.5);
  \coordinate (n33) at (8,0);
  \coordinate (n34) at (8,1.5);
  \coordinate (n35) at (9,-1.5);
  \coordinate (n36) at (9,0);
  \coordinate (n37) at (9,1.5);
  \coordinate (n38) at (10,-1.5);
  \coordinate (n39) at (10,0);
  \coordinate (n40) at (10,1.5);
  \coordinate (n41) at (11,0);
  \coordinate (n42) at (12,-1.5);
  \coordinate (n43) at (12,0);
  \coordinate (n44) at (12,1.5);
  \coordinate (n45) at (13,-1.5);
  \coordinate (n46) at (13,0);
  \coordinate (n47) at (13,1.5);
  \coordinate (n48) at (14,0);
  
 \draw[fill=black] (n31) circle (0.15cm);
 \draw[fill=black] (n32) circle (0.15cm)
 node [below left] {$v_{31}'$};
 \draw[fill=black] (n33) circle (0.15cm);
 \draw[fill=black] (n34) circle (0.15cm)
 node [above left] {$v_{31}$};
 \draw[fill=black] (n35) circle (0.15cm)
 node [below left] {$v_{32}'$};
 \draw[fill=black] (n36) circle (0.15cm);
 \draw[fill=black] (n37) circle (0.15cm)
 node [above left] {$v_{32}$};
 \draw[fill=black] (n38) circle (0.15cm)
 node [below left] {$v_{33}'$};
 \draw[fill=black] (n39) circle (0.15cm);
 \draw[fill=black] (n40) circle (0.15cm)
 node [above left] {$v_{33}$};
 \draw[fill=black] (n41) circle (0.15cm);
 \draw[fill=black] (n42) circle (0.15cm)
 node [below left] {$v_{41}'$};
 \draw[fill=black] (n43) circle (0.15cm);
 \draw[fill=black] (n44) circle (0.15cm)
 node [above left] {$v_{41}$};
 \draw[fill=black] (n45) circle (0.15cm)
 node [below left] {$v_{42}'$};
 \draw[fill=black] (n46) circle (0.15cm);
 \draw[fill=black] (n47) circle (0.15cm)
 node [above left] {$v_{42}$};
 \draw[fill=black] (n48) circle (0.15cm);

  \foreach \from/\to in {n31/n32,n31/n34,n32/n33,n32/n36,n33/n34,n33/n35,n33/n37,n34/n36,n35/n36,n36/n37,n37/n39,n35/n39,n36/n38,n36/n40,n38/n39,n39/n40,n38/n41,n40/n41,n41/n42,n41/n44,n42/n43,n42/n46,n43/n44,n43/n45,n43/n47,n44/n46,n45/n46,n46/n47,n47/n48,n45/n48} 
    \draw (\from) -- (\to);
    
\end{tikzpicture}
    \caption{The depicted graph $H$ is a subgraph of some $(2,2)$-$C_s$-tight graph $G$. All nodes in $G$ have all their neighbours on the mirror and $H$ is induced by $S$ and the neighbours of vertices in $S$. We have labeled so that $v_{ji},v_{ji}' \in S_j$. Note that $v_{31}$ and $v_{33}'$ are in the same set of the partition since there is a 4-cycle containing $v_{31}$ and $v_{32}$ and another containing $v_{32}$ and $v_{33}'$, but $v_{33}$ and $v_{41}$ are in different sets since no two common neighbours exist.}
    \label{fig:13partition}
\end{figure}
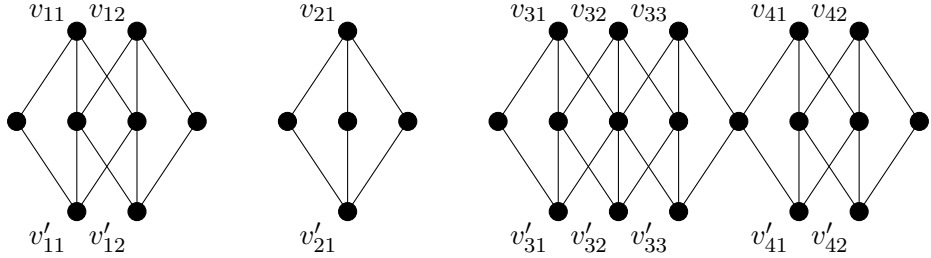

\begin{proof}
If $V^g \neq \emptyset$, then by Lemma \ref{lem: Vgempty} there exists a rb-component $A_i$ with $|E^{bg}_{i}| \leq 2k_i - 2$ and $s_i \geq 2k_i -|E^{bg}_{i}|$.
Suppose $S\cap V_i=\{u_1,u_2,\dots,u_r\}$.
We define $S_1$ recursively.
Let $u_1$ be in $S_1$.
For any $u_q \in S\cap V_i$, $u_q \in S_1$ if there exists $t_1,t_2 \in V^r$ and $u_p \in S_1$ so that $u_qt_1u_pt_2$ is a $4$-cycle in $G$.
If $S_1 \neq S\cap V_i$, take $u_k \in (S \cap V_i)\setminus S_1$ and put it in $S_2$, then define $S_2$ analogously.
In this manner we obtain the partition $S\cap V_i = S_1 \sqcup S_2 \sqcup \dots \sqcup S_l$. (See Figure \ref{fig:13partition} for an illustration.)
Since there is no $K_{3,4}$ subgraph of $G$ we may assume, for a contradiction, that all red pairs of neighbours of nodes are contained in at least two $4$-cycles. Hence the degree of any red vertex adjacent to a node is at least six.
It is possible that a vertex of $S_j$ shares exactly one neighbour with a vertex of $S_k$ for $j \neq k$. Let there be $p$ such vertices. All such vertices have degree at least 12.

For each $S_j$ there are $\frac{1}{2}|S_j|+2$ red vertices of degree at least six (this double counts the $p$ vertices of minimum degree 12), and at least $|S_j|+4$ blue vertices of degree at least 4.
Let $r$ and $b$ be the number of red and blue vertices respectively of $V_i$ not already counted.
Then, $|V_i| = \frac{1}{2}s_i + 2l - p + r + s_i + 4l + b + s_i$, the first four summands representing red vertices, the next three blue vertices, and the last summand the nodes (which are also blue).
Once again we turn to counting degrees. We have
\begin{equation}
    2|E_i| + |E^{bg}_{i}| \geq 6(\tfrac{1}{2}s_i + 2l -2p) +12p + 4r + 4(s_i +4l) +4b + 3s_i = 10s_i +28l + 4r + 4b.
\end{equation}
Also, since each $A_i$ has $|E^{bg}_{i}| \leq 2k_i - 2$ and $|V_i| = \frac{5}{2}s_i + 6l + r +b - p$, we have
\begin{equation}\label{Aiedgecount}
    2|E_i| + |E^{bg}_{i}| \leq 4|V_i| - 2k_i +2k_i - 2 = 10s_i + 24l + 4r +4b - 4p -2.
\end{equation}
This implies that $4l +4p \leq -2$, contradicting our assumption that all pairs of neighbours of a node are in two $C_4$.
If $V^g = \emptyset$, the proof is unchanged except that $|E^{bg}| = 0$ and $|E| = 2|V|-2$, so Equation (\ref{Aiedgecount}) is 
$$2|E| \leq 4|V| - 4 = 10s +24l +4r +4b -4p -4$$ 
and so $4l +4p \leq -4$ instead.

Finally, we need to check that this $C_4$-contraction preserves sparsity.
Label the vertices $v,v',t,t_1$ where $t,t_1$ are fixed and are contracted and labelled $t$ in the new graph, and let the final neighbour of $v$ be $t_2$.
Indeed, if a subgraph $H$ of the reduced graph breaks sparsity, then $H = (V_t, E_t)$ has $|E_t| \geq 2|V_t|-1$.
If $H$ is $C_s$-symmetric this must be $|E_t| \geq 2|V_t|$, and if $H$ is not symmetric, $H$ has at least one fixed vertex (namely $t$), so $V_t \cap \sigma V_t \neq \emptyset$, and by similar counting arguments to Remark \ref{critical cup and cap}, one of $H \cap H'$ and $H \cup H'$ has $|\Tilde{E}| = 2|\Tilde{V}|$.
We may therefore assume $H$ is $C_s$-symmetric.
Noting that $E_t \subset E$, we draw the following conclusions.

If $v,v' \in V_t$ then $i(V_t + \{t_1\}) = |E_t| + 2$ so $V_t + \{t_1\}$ breaks sparsity.
Else  we have $v,v' \notin V_t$ and $i(V_t + \{t_1,v,v'\}) = |E_t| + 4$ if $t_2 \notin V_t$ and $i(V_t + \{t_1,v,v'\}) = |E_t| + 6$ if $t_2 \in V_t$.
Therefore $V_t + \{t_1,v,v'\}$ breaks sparsity unless $t_2 \notin V_t$ and $|E_t| = 2|V_t|$.
However, in this final case $G[V_t + \{t_1,v,v'\}]$ is a $(2,2)$-$C_s$-tight proper subgraph of $G$ contradicting the conditions of the lemma.
\end{proof}

\subsection{Combinatorial characterisation}

We can now put together the combinatorial results of this section to prove the following recursive construction and then apply this result alongside the results of Section \ref{sec:ops} to deduce our characterisation of $C_2$-isostatic graphs.

\begin{thm}\label{thm:recursioncs}
A graph $(G,\phi)$ is $(2,2)$-$C_s$-tight if and only if $(G,\phi)$ can be generated from the graphs $(F_2,\phi_2),(W_5,\phi_4),(Wd(4,2),\phi_5),(F_1,\phi_1),(F_1,\phi_6),(K_{3,4},\phi_7)$ (these graphs were depicted in Figure \ref{BaseGraphs,C_s,cylinder}) by fixed-vertex 0-extensions, fixed-vertex-to-$C_4$ and symmetrised 0-extensions, 1-extensions, vertex-to-$K_{4}$, vertex-to-$C_{4}$, and vertex-to-$(2,2)$-$C_s$-tight operations.
\end{thm}

\begin{proof}
Each of the base graphs are independent and tight and it can be seen that the fixed-vertex 0-extension, fixed-vertex-to-$C_4$ and symmetrised 0-extension, 1-extension, vertex-to-$K_4$, vertex-to-$C_4$ and vertex-to-$(2,2)$-$C_s$-tight operations preserve independence, tightness and do not introduce fixed edges.
By Theorem \ref{fixed e&v on cy}, any graph after such operations is $(2,2)$-$C_s$-tight.

Conversely, we show by induction that any $(2,2)$-$C_s$-tight graph $G$ can be generated from our base graphs.
Suppose the induction hypothesis holds for all graphs with $|V| < n$.
Now let $|V|= n$ and suppose $G$ is not isomorphic to one of the base graphs in Figure $\ref{BaseGraphs,C_s,cylinder}$.
Obviously any $(2,2)$-$C_s$-tight graph contains a vertex of degree 2 or 3. The former case is dealt with by Lemma \ref{deg(v)=2}.
We can also apply Lemma \ref{lem: 2,2 tight to vertex} to assume there are no $(2,2)$-$C_s$-tight proper subgraphs of $G$.
Hence suppose $\delta(G)=3$ and $v$ is a vertex of minimum degree.
There are 6 cases depending on the closed neighbourhood of $v$, namely with labelling from this section, $N(v)\cap N(v') \in \{\emptyset,\{t\},\{x,x'\}, \{t_1,t_2\}, \{t,x,x'\}, \{t_1,t_2,t_3\}\}$.
By Lemmas \ref{lem: deg(v)=3, cap empty} ($\emptyset$), \ref{deg 3 C_2 t} ($\{t\}$), \ref{deg(v)=3, x=y'} ($\{x,x'\}$ and $\{t_1,t_2\}$), \ref{deg 3 C_2 t x x'} ($\{t,x,x'\}$) we see that the only remaining blocks to reducing any given node is $K_4$ or all three neighbours being fixed vertices.
If $G$ has a degree 3 that is contained in a $K_4$ then by Lemmas \ref{deg(v)=3, K4} and \ref{deg(v)=3, cycle contraction} we may assume that the $K_4$ and its symmetric copy intersect non-trivially. Since there are no $(2,2)$-$C_s$-tight proper subgraphs of $G$ this gives a contradiction.
Finally we may suppose that all nodes have all neighbours on the mirror, and by Lemma \ref{lem: node all mirror} there exists a $C_4$ contraction, completing the proof.
\end{proof}

\begin{thm}\label{thm:maincs}
A graph $(G,\phi)$ is $C_s$-isostatic if and only if it is $(2,2)$-$C_s$-tight.
\end{thm}

\begin{proof}
Since $C_s$-isostatic graphs are $(2,2)$-tight, necessity follows from Theorem \ref{fixed e&v on cy}.
It is easy to check using any computer algebra package that the base graphs depicted in Figure \ref{BaseGraphs,C_s,cylinder} are $C_s$-isostatic.
Hence the sufficiency follows from Theorem \ref{thm:recursioncs}, Lemmas \ref{0-ext rigid}, \ref{1-ext rigid}, \ref{minrigidK_4}, \ref{2addition rigid}, \ref{cycle rigid}, and Remarks \ref{rem: Cs fixed 0-ext}, \ref{rem: C4 mirror} by induction on $|V|$.
\end{proof}

\section{Concluding remarks} \label{sec:final}

It is classical \cite{N-W} that every $(2,2)$-tight graph can be decomposed into the edge-disjoint union of two spanning trees, and such packing or decomposition results are often of interest in combinatorial optimisation \cite{Fra}.
We derive symmetric decomposition results for $C_2$, $C_i$ and $C_s$ in the following corollaries. 

\begin{cor}
A graph $(G,\phi)$ is $C_2$-isostatic if and only if it is the edge-disjoint union of two $\mathbb{Z}_2$-symmetric spanning trees $(T_1,\phi), (T_2,\phi)$.
\end{cor}

\begin{proof}
To show sufficiency, note that $(T_1,\phi), (T_2,\phi)$ can be labelled so that if $u$ is the symmetric copy of $v$ in $T_1$, then they are symmetric copies in $T_2$.
By parity, each tree will either have one fixed vertex, which will be the same vertex in $G$, or one fixed edge.
Since the spanning trees are edge-disjoint, $G$ will either have one fixed vertex and no fixed edge, or no fixed vertex and two fixed edges.
Further, it is known that the edge-disjoint union of two spanning trees is $(2,2)$-tight.
The fact that $G$ is $(2,2)$-$C_2$-tight now follows from the $C_2$-symmetry of the two spanning trees.

We prove the necessity of the symmetric decomposition by applying Theorem \ref{thm:recursionc2}.
It will be convenient to think of the edges of the two trees as being coloured red and blue respectively.
We illustrate appropriate colourings of the base graphs in Figure \ref{ColouredBaseGraphs,C_2}.
To check that the operations preserve the coloured trees, we describe the edge colourings for each operation.

Firstly, the symmetrised 0-extension has one edge coloured red and the other blue, with the symmetric edges coloured the same as their preimage.
For a symmetrised 1-extension, say $xy$ and $x'y'$ in $G$ are deleted and the new vertices added in $G^+$ are $v$ and $v'$, then colour $vx,vy,v'x',v'y'$ in $G^+$ the colour of $xy$ in $G$, and set the third edge incident to $v$ (resp. $v'$) as the other colour.
A double 1-extension can be thought of in the same way; if $vv'$ was deleted in $G$, the path containing $v,w,w',v'$ will be coloured the same as $ww'$.
In a symmetrised vertex-to-$K_4$ operation, the two new $K_4$ subgraphs should be coloured as in Figure \ref{ColouredBaseGraphs,C_2} in such a way to preserve the symmetry.
A vertex-to-$(2,2)$-$C_2$-tight subgraph operation replaces a fixed vertex with a $(2,2)$-$C_2$-tight subgraph.
As seen in Lemma \ref{lem: 2,2 tight to vertex}, the new subgraph can be constructed from either $(W_5,\phi_4)$ or $(Wd(4,2)\phi_5)$ with a series of symmetrised 0-extensions. We therefore colour the subgraph starting with the $(W_5,\phi_4)$ or $(Wd(4,2),\phi_5)$ copy, and colour the edges of the 0-extensions as previously described.

Finally, we note that we do not perform fixed-vertex-to-$C_4$ operations when considering $(2,2)$-$C_2$-tight graphs. A symmetrised vertex-to-$C_4$ operation can have two possibilities.
The path of length 2 on $v_1,w,v_2$ (with $w$ to be split into $w$ and $u$ in the operation, $N_G(w)=v_1,v_2,\dots,v_r$ and $v_1,v_2$ becoming adjacent to both) can be coloured with both edges the same colour, or each edge different. In both cases, colour the edges of $\hat{G} = G^{+} \setminus \{wv_1, wv_2, uv_1, uv_2, w'v_1', w'v_2', u'v_1', u'v_2'\}$ as in $G$.
Now suppose first that $wv_1$ is red and $wv_2$ is blue in $G$.
Then in $G^+$, we colour $wv_1, uv_1$ red and $wv_2, uv_2$ blue, and $\mu v_i$ the same colour as $wv_i$ for all $\mu \in \{w,u\}$ and $i \in \{3,\dots,r\}$ (colouring the edges in the orbit analogously).

Hence we may suppose both $wv_1$ and $wv_2$ are coloured red in $G$.
We claim that for any arrangement of the edges from $v_3,\dots,v_r$ to either $w$ or $u$ in $G^+$, there is a colouring in $G^+$ of $wv_1, wv_2, uv_1, uv_2$ with three red edges and one blue edge that will result in $G^+$ being the edge-disjoint union of two $C_2$-symmetric spanning trees.
Note that such a colouring gives $|V(G^+)|-1$ blue and red edges.
Necessarily, $w$ and $u$ are in different connected components of the $\hat{G}$ induced by the blue edges, say $X_w$ and $X_u$ respectively.
The vertex $v_1$ will be in one of these components, without loss of generality say $X_w$.
Colouring the edge $uv_1$ blue will connect these two components and hence give a blue spanning tree. Since $wv_1$ and $wv_2$ are coloured red in $G$ it is easy to see that colouring the edges $uv_2,wv_1,wv_2$ red in $G^+$ will produce a red spanning tree.
Applying this colouring symmetrically completes the proof.
\end{proof}

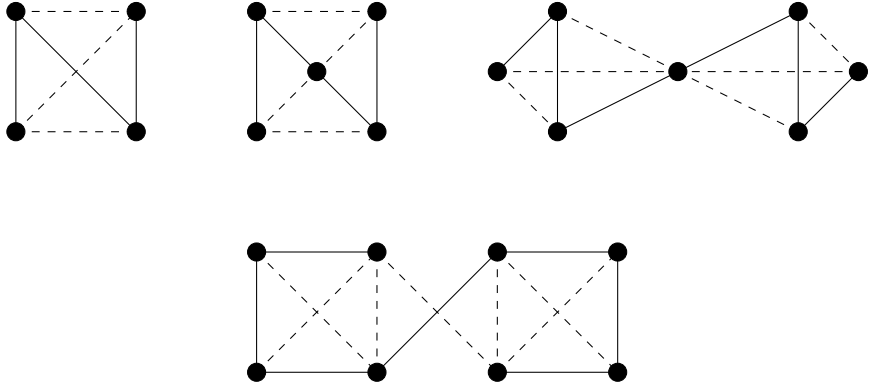
\begin{figure}[ht]
    \centering
    \begin{tikzpicture}
  [scale=.8,auto=left]
  
  \coordinate (n21) at (0,4);
  \coordinate (n22) at (0,6);
  \coordinate (n23) at (2,4);
  \coordinate (n24) at (2,6);
 \draw[fill=black] (n21) circle (0.15cm);
 \draw[fill=black] (n22) circle (0.15cm);
 \draw[fill=black] (n23) circle (0.15cm);
 \draw[fill=black] (n24) circle (0.15cm);

  \foreach \from/\to in {n21/n22,n23/n24,n22/n23} 
    \draw (\from) -- (\to);
  \foreach \from/\to in {n21/n23,n22/n24,n21/n24} 
    \draw[dashed] (\from) -- (\to);
  
  \coordinate (n1) at (4,4);
  \coordinate (n2) at (4,6);
  \coordinate (n3) at (6,4);
  \coordinate (n4) at (6,6);
  \coordinate (n5) at (5,5);
 \draw[fill=black] (n1) circle (0.15cm);
 \draw[fill=black] (n2) circle (0.15cm);
 \draw[fill=black] (n3) circle (0.15cm);
 \draw[fill=black] (n4) circle (0.15cm);
 \draw[fill=black] (n5) circle (0.15cm);

  \foreach \from/\to in {n1/n2,n3/n4,n2/n5,n3/n5} 
    \draw (\from) -- (\to);
  \foreach \from/\to in {n1/n3,n2/n4,n1/n5,n4/n5} 
    \draw[dashed] (\from) -- (\to);
    
  \coordinate (n11) at (8,5);
  \coordinate (n12) at (9,4);
  \coordinate (n13) at (9,6);
  \coordinate (n14) at (13,4);
  \coordinate (n15) at (13,6);
  \coordinate (n16) at (14,5);
  \coordinate (n17) at (11,5);
 \draw[fill=black] (n11) circle (0.15cm);
 \draw[fill=black] (n12) circle (0.15cm);
 \draw[fill=black] (n13) circle (0.15cm);
 \draw[fill=black] (n14) circle (0.15cm);
 \draw[fill=black] (n15) circle (0.15cm);
 \draw[fill=black] (n16) circle (0.15cm);
 \draw[fill=black] (n17) circle (0.15cm);

  \foreach \from/\to in {n11/n13,n12/n13,n14/n15,n14/n16,n12/n17,n15/n17} 
    \draw (\from) -- (\to);
  \foreach \from/\to in {n11/n12,n15/n16,n11/n17,n13/n17,n14/n17,n16/n17} 
    \draw[dashed] (\from) -- (\to);

  \coordinate (n41) at (4,0);
  \coordinate (n42) at (4,2);
  \coordinate (n43) at (6,0);
  \coordinate (n44) at (6,2);
  \coordinate (n45) at (8,0);
  \coordinate (n46) at (8,2);
  \coordinate (n47) at (10,0);
  \coordinate (n48) at (10,2);  
  
 \draw[fill=black] (n41) circle (0.15cm);
 \draw[fill=black] (n42) circle (0.15cm);
 \draw[fill=black] (n43) circle (0.15cm);
 \draw[fill=black] (n44) circle (0.15cm);
 \draw[fill=black] (n45) circle (0.15cm);
 \draw[fill=black] (n46) circle (0.15cm);
 \draw[fill=black] (n47) circle (0.15cm);
 \draw[fill=black] (n48) circle (0.15cm);

  \foreach \from/\to in {n41/n42,n41/n43,n42/n44,n43/n46,n45/n47,n46/n48,n47/n48} 
    \draw (\from) -- (\to);
  \foreach \from/\to in {n41/n44,n42/n43,n43/n44,n44/n45,n45/n46,n45/n48,n46/n47} 
    \draw[dashed] (\from) -- (\to);
        
\end{tikzpicture}
    \caption{The $C_2$-symmetric  base graphs decomposed into two $C_2$-symmetric edge disjoint trees, coloured red and blue (depicted with dashed and solid edges respectively).}
    \label{ColouredBaseGraphs,C_2}
\end{figure}

With a similar proof (see \cite{Wall} for details) we can establish the following analogous result.
We illustrate appropriate colourings of the base graphs in Figures \ref{ColouredBaseGraphs, C_i} and \ref{ColouredBaseGraphs, C_s}.

\begin{cor}
For $\tau(\Gamma) \in \{C_i, C_s\}$, a $\Gamma$-symmetric graph $(G,\phi)$ is $\tau(\Gamma)$-isostatic if and only if it is the edge-disjoint union of two spanning trees $T_1,T_2$, where $\phi(\gamma)T_1=T_2$ for the non-trivial element $\gamma$ of $\Gamma$.
\end{cor}

\begin{figure}[ht]
    \centering
    \begin{tikzpicture}
  [scale=.8,auto=left]
  
  \coordinate (n1) at (0,1);
  \coordinate (n2) at (1,0);
  \coordinate (n3) at (1,2);
  \coordinate (n4) at (3,0);
  \coordinate (n5) at (3,2);
  \coordinate (n6) at (4,1);
  
 \draw[fill=black] (n1) circle (0.15cm);
 \draw[fill=black] (n2) circle (0.15cm);
 \draw[fill=black] (n3) circle (0.15cm);
 \draw[fill=black] (n4) circle (0.15cm);
 \draw[fill=black] (n5) circle (0.15cm);
 \draw[fill=black] (n6) circle (0.15cm);

  \foreach \from/\to in {n1/n2,n1/n4,n2/n3,n3/n5,n4/n6} 
    \draw (\from) -- (\to);
  \foreach \from/\to in {n1/n3,n2/n4,n3/n6,n4/n5,n5/n6} 
    \draw[dashed] (\from) -- (\to);
    
  \coordinate (n11) at (6,0);
  \coordinate (n12) at (6,2);
  \coordinate (n13) at (8,0);
  \coordinate (n14) at (8,2);
  \coordinate (n15) at (10,0);
  \coordinate (n16) at (10,2);
  \coordinate (n17) at (12,0);
  \coordinate (n18) at (12,2);  
  
 \draw[fill=black] (n11) circle (0.15cm);
 \draw[fill=black] (n12) circle (0.15cm);
 \draw[fill=black] (n13) circle (0.15cm);
 \draw[fill=black] (n14) circle (0.15cm);
 \draw[fill=black] (n15) circle (0.15cm);
 \draw[fill=black] (n16) circle (0.15cm);
 \draw[fill=black] (n17) circle (0.15cm);
 \draw[fill=black] (n18) circle (0.15cm);

  \foreach \from/\to in {n11/n12,n11/n13,n12/n14,n13/n15,n15/n16,n15/n18,n16/n17} 
    \draw (\from) -- (\to);
  \foreach \from/\to in {n11/n14,n12/n13,n13/n14,n14/n16,n15/n17,n16/n18,n17/n18} 
    \draw[dashed] (\from) -- (\to);
    
\end{tikzpicture}
    \caption{The $C_i$-symmetric base graphs decomposed into two edge disjoint spanning trees, coloured red and blue (depicted with dashed and solid edges respectively), which are symmetric copies of one another.}
    \label{ColouredBaseGraphs, C_i}
\end{figure}
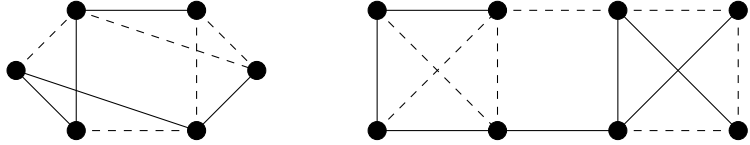

\begin{figure}[ht]
    \centering
    \begin{tikzpicture}
  [scale=.8,auto=left]
  
  \coordinate (n1) at (1,-3);
  \coordinate (n2) at (2,-4);
  \coordinate (n3) at (2,-2);
  \coordinate (n4) at (4,-4);
  \coordinate (n5) at (4,-2);
  \coordinate (n6) at (5,-3);
 \draw[fill=black] (n1) circle (0.15cm);
 \draw[fill=black] (n2) circle (0.15cm);
 \draw[fill=black] (n3) circle (0.15cm);
 \draw[fill=black] (n4) circle (0.15cm);
 \draw[fill=black] (n5) circle (0.15cm);
 \draw[fill=black] (n6) circle (0.15cm);

  \foreach \from/\to in {n1/n2,n1/n4,n1/n5,n2/n3,n5/n6} 
    \draw (\from) -- (\to);
  \foreach \from/\to in {n1/n3,n2/n6,n3/n6,n4/n5,n4/n6} 
    \draw[dashed] (\from) -- (\to);

  \coordinate (n11) at (0,0);
  \coordinate (n12) at (0,2);
  \coordinate (n13) at (2,0);
  \coordinate (n14) at (2,2);
  \coordinate (n15) at (4,0);
  \coordinate (n16) at (4,2);
  \coordinate (n17) at (6,0);
  \coordinate (n18) at (6,2);  
  
 \draw[fill=black] (n11) circle (0.15cm);
 \draw[fill=black] (n12) circle (0.15cm);
 \draw[fill=black] (n13) circle (0.15cm);
 \draw[fill=black] (n14) circle (0.15cm);
 \draw[fill=black] (n15) circle (0.15cm);
 \draw[fill=black] (n16) circle (0.15cm);
 \draw[fill=black] (n17) circle (0.15cm);
 \draw[fill=black] (n18) circle (0.15cm);

  \foreach \from/\to in {n11/n12,n11/n13,n12/n14,n13/n16,n15/n16,n15/n18,n16/n17} 
    \draw (\from) -- (\to);
  \foreach \from/\to in {n11/n14,n15/n17,n17/n18,n12/n13,n13/n14,n14/n15,n16/n18} 
    \draw[dashed] (\from) -- (\to);
    
  \coordinate (n21) at (12,1);
  \coordinate (n22) at (13,0);
  \coordinate (n23) at (13,2);
  \coordinate (n24) at (17,0);
  \coordinate (n25) at (17,2);
  \coordinate (n26) at (18,1);
  \coordinate (n27) at (15,1);
 \draw[fill=black] (n21) circle (0.15cm);
 \draw[fill=black] (n22) circle (0.15cm);
 \draw[fill=black] (n23) circle (0.15cm);
 \draw[fill=black] (n24) circle (0.15cm);
 \draw[fill=black] (n25) circle (0.15cm);
 \draw[fill=black] (n26) circle (0.15cm);
 \draw[fill=black] (n27) circle (0.15cm);

  \foreach \from/\to in {n21/n23,n24/n25,n22/n27,n25/n27,n21/n22,n26/n27} 
    \draw (\from) -- (\to);
  \foreach \from/\to in {n22/n23,n24/n26,n25/n26,n21/n27,n23/n27,n24/n27} 
    \draw[dashed] (\from) -- (\to);
    
  \coordinate (n31) at (13,-4);
  \coordinate (n32) at (13,-2);
  \coordinate (n33) at (15,-5);
  \coordinate (n34) at (15,-3);
  \coordinate (n35) at (15,-1);
  \coordinate (n36) at (17,-4);
  \coordinate (n37) at (17,-2);
 \draw[fill=black] (n31) circle (0.15cm);
 \draw[fill=black] (n32) circle (0.15cm);
 \draw[fill=black] (n33) circle (0.15cm);
 \draw[fill=black] (n34) circle (0.15cm);
 \draw[fill=black] (n35) circle (0.15cm);
 \draw[fill=black] (n36) circle (0.15cm);
 \draw[fill=black] (n37) circle (0.15cm);

  \foreach \from/\to in {n31/n33,n31/n34,n32/n34,n32/n35,n35/n36,n33/n37} 
    \draw (\from) -- (\to);
  \foreach \from/\to in {n31/n35,n32/n33,n34/n36,n34/n37,n35/n37,n33/n36} 
    \draw[dashed] (\from) -- (\to);
    
  \coordinate (n41) at (7,-4);
  \coordinate (n42) at (7,-2);
  \coordinate (n43) at (9,-4);
  \coordinate (n44) at (9,-2);
  \coordinate (n45) at (11,-4);
  \coordinate (n46) at (11,-2);
 \draw[fill=black] (n41) circle (0.15cm);
 \draw[fill=black] (n42) circle (0.15cm);
 \draw[fill=black] (n43) circle (0.15cm);
 \draw[fill=black] (n44) circle (0.15cm);
 \draw[fill=black] (n45) circle (0.15cm);
 \draw[fill=black] (n46) circle (0.15cm);

  \foreach \from/\to in {n41/n42,n41/n43,n42/n44,n44/n45,n43/n46} 
    \draw (\from) -- (\to);
  \foreach \from/\to in {n41/n44,n42/n43,n43/n45,n44/n46,n45/n46} 
    \draw[dashed] (\from) -- (\to);
    
  \coordinate (n51) at (8,0);
  \coordinate (n52) at (8,2);
  \coordinate (n53) at (9,2.5);
  \coordinate (n54) at (10,0);
  \coordinate (n55) at (10,2);
 \draw[fill=black] (n51) circle (0.15cm);
 \draw[fill=black] (n52) circle (0.15cm);
 \draw[fill=black] (n53) circle (0.15cm);
 \draw[fill=black] (n54) circle (0.15cm);
 \draw[fill=black] (n55) circle (0.15cm);

  \foreach \from/\to in {n51/n53,n52/n53,n52/n54,n54/n55} 
    \draw (\from) -- (\to);
  \foreach \from/\to in {n51/n52,n51/n55,n53/n55,n53/n54} 
    \draw[dashed] (\from) -- (\to);

\end{tikzpicture}
    \caption{The $C_s$-symmetric base graphs decomposed into two edge disjoint spanning trees, coloured red and blue (depicted with dashed and solid edges respectively), which are symmetric copies of one another.}
    \label{ColouredBaseGraphs, C_s}
\end{figure}
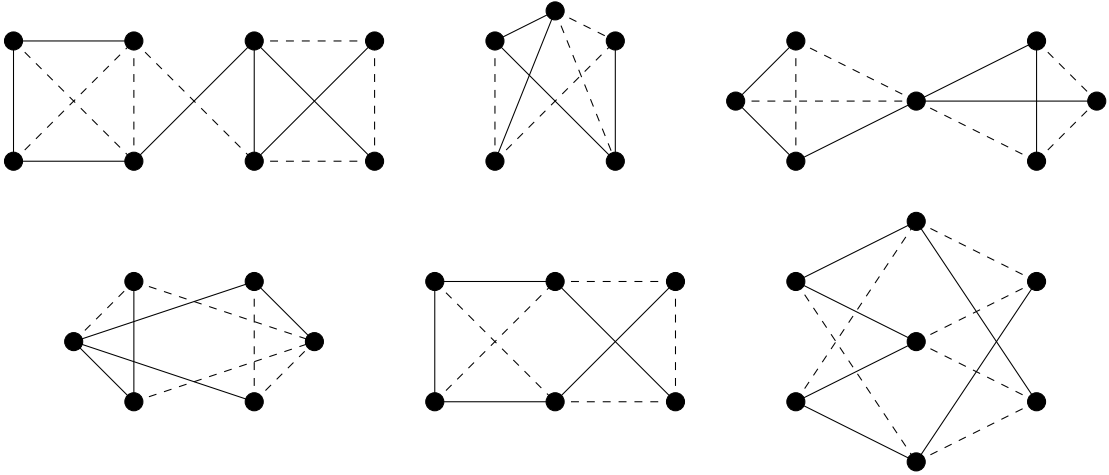

An immediate consequence of  Theorems~\ref{thm:mainci}, \ref{thm:mainc2}, and \ref{thm:maincs} is that there are efficient, deterministic algorithms for determining whether a given graph is  $C_i$-, $C_2$-, or $C_s$-isostatic since the $(2,2)$-sparsity counts can be checked using the standard pebble game algorithm \cite{Hen,ls} and  the additional symmetry conditions for the number of fixed vertices and edges can be checked in constant time, from the group action $\phi$.

The next obvious challenge would be to extend the characterisations in Theorems \ref{thm:mainci}, \ref{thm:mainc2}, and \ref{thm:maincs} to deal with the remaining groups described in Theorem \ref{fixed e&v on cy}.
While it is conceivable these groups could be handled by an elaboration of our techniques there will be many more cases and technical details to consider due to the multiple symmetry conditions. Moreover the corresponding problems in the Euclidean plane \cite{schulze,BS4} remain open, providing a note of caution.

Analogous to the situation for frameworks in the Euclidean plane, an infinitesimally rigid $C_2$-symmetric framework on $\Y$ does not necessarily have a spanning isostatic subframework with the same symmetry.  An example is depicted in Figure \ref{fig C2 no iso sub}. 
Thus, for symmetric frameworks on $\Y$, infinitesimal rigidity can in general
not be characterised in terms of symmetric isostatic subframeworks. To analyse symmetric frameworks for infinitesimal rigidity, rather than isostaticity, a different approach (similar to the one in \cite{ST}, for example) may be needed. Surprisingly, it turns out that for $C_i$ and $C_s$ the situation is special and a simplified version of the approach in \cite{ST} may be applied 
in combination with Theorems \ref{thm:mainci} and \ref{thm:maincs} 
to deduce the following characterisation of symmetric infinitesimal rigidity. We refer the reader to \cite{Wall} for the proof.

\begin{figure}[ht]
    \centering
    \begin{tikzpicture}
  [scale=.8,auto=left]
    
  \coordinate (n1) at (6,0);
  \coordinate (n2) at (6,2);
  \coordinate (n3) at (8,0);
  \coordinate (n4) at (8,2);
  \coordinate (n5) at (10,0);
  \coordinate (n6) at (10,2);
  \coordinate (n7) at (12,0);
  \coordinate (n8) at (12,2);  
  
 \draw[fill=black] (n1) circle (0.15cm);
 \draw[fill=black] (n2) circle (0.15cm);
 \draw[fill=black] (n3) circle (0.15cm);
 \draw[fill=black] (n4) circle (0.15cm);
 \draw[fill=black] (n5) circle (0.15cm);
 \draw[fill=black] (n6) circle (0.15cm);
 \draw[fill=black] (n7) circle (0.15cm);
 \draw[fill=black] (n8) circle (0.15cm);

  \foreach \from/\to in {n1/n2,n1/n3,n1/n4,n2/n3,n2/n4,n3/n5,n3/n4,n4/n6,n5/n6,n5/n7,n5/n8,n6/n7,n6/n8,n7/n8,n2/n5,n4/n7} 
    \draw (\from) -- (\to);
    
\end{tikzpicture}
    \caption{A $C_2$-rigid graph where no vertex or edge is fixed by the half-turn. There is no $(2,2)$-$C_2$-tight spanning subgraph.}
    \label{fig C2 no iso sub}
\end{figure}
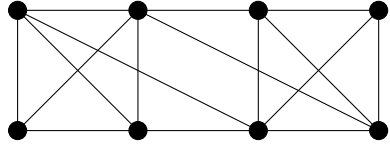

\begin{thm}
For $\tau(\Gamma) \in \{C_i, C_s\}$, a graph $(G,\phi)$ is $\tau(\Gamma)$-infinitesimally rigid if and only if $(G,\phi)$ has a spanning subgraph $H$ that is $(2,2)$-$\tau(\Gamma)$-tight.
\end{thm}

It would be interesting to continue the line of investigation initiated in this paper by looking at other surfaces. As mentioned in the introduction in the case of the sphere we can use the proof technique of Section \ref{sec:necrep} to derive the following theorem. See \cite{Wall} for details.

\begin{thm}
Any $\tau(\Gamma)$-symmetric isostatic framework on the sphere will have the following restrictions:
\begin{center}
    \begin{tabular}{|c|c|}
    \hline
    $\tau(\Gamma)$ & $\text{Number of edges and vertices fixed by symmetry operations}$ \\
    \hline
    $C_s$ & $e_{\sigma} = 1$ \\ 
    $C_2$ & $e_{2} = 1$\\ 
    $C_3$ & $e_3 = v_3 = 0$\\
    $C_{2v}$ & $e_{\sigma} = e_{2} = 1, v_{2} = 0$ \\
    $C_{3v}$ & $e_{3} = 0, v_{3}=0, e_{\sigma} = e_{\sigma'} = e_{\sigma''} = 1$ \\
    $D_{3}$ & $e_{3} = 0, v_{3}= v_{2'} = 0, e_{2'} = 1$ \\
    \hline
    \end{tabular}
\end{center}
\end{thm}

There is a precise geometric correspondence between infinitesimal rigidity in the plane and on the sphere (see \cite{EJNSTW} for details) and this extends to symmetric frameworks for any plane symmetry group \cite{CNSW}. By applying this projective correspondence alongside the results of \cite{CNSW} alongside \cite{schulze} and \cite{BS4} we immediately obtain precise analogues of our main result for $C_s,C_2$ and $C_3$ symmetry groups on the sphere. The next two cases, $C_{2v}$ and $C_{3v}$, remain open in the plane and hence the same tactic is unavailable. That leaves $D_3$ which does not exist as a symmetry group in the plane. However this group, due to the higher order of the group, is unlikely to be straightforward.

In \cite{NOP14} surfaces with one ambient rigid motion were analysed and, combinatorially, the necessary change in the count is from $(2,2)$-tight to $(2,1)$-tight\footnote{A graph $G=(V,E)$ is $(2,1)$-tight if $|E|=2|V|-1$ and every proper subgraph on $V'$ vertices has at most $2|V'|-1$ edges.}. We expect that elaborations of our techniques would be possible for a small number of groups; interestingly, which groups are tractable may depend on the choice of which surface with one ambient rigid motion is chosen.
However, a possibly more tractable and more widely applicable problem is to move to linearly constrained frameworks (see, for example, \cite{CGJN}). These frameworks, in the 3-dimensional case, model `generic surfaces', but the concept is easily adaptable to arbitrary dimension where some interesting results are known \cite{CGJN,JNT}.

\section*{Acknowledgements}

AN was partially supported by EPSRC grants EP/W019698/1 and EP/X036723/1.

\bibliographystyle{abbrv}
\def\lfhook#1{\setbox0=\hbox{#1}{\ooalign{\hidewidth
  \lower1.5ex\hbox{'}\hidewidth\crcr\unhbox0}}}

\end{document}